\documentclass[12pt]{amsart}
\usepackage{latexsym,amssymb,amsmath,amsthm}
\usepackage{float}
\usepackage{graphicx}
\newtheorem{thm}{Theorem}

\newtheorem{lemma}[thm]{Lemma}
\newtheorem{propo}[thm]{Proposition}
\theoremstyle{definition}
\newtheorem{defn}{Definition}

\newcommand{\Z}{\mathbb{Z}}

\title{Dynamically generated Networks}
\author{Oliver Knill}
\date{November 17, 2013}
\address{
        Department of Mathematics \\
        Harvard University \\
        Cambridge, MA, 02138, 
        Harvard University 
        }

\begin{document}
\begin{abstract}
Simple algebraic rules can produce complex networks with rich
structures. These graphs are obtained when looking at a monoid
operating on a ring. There are relations to dynamical systems theory 
and number theory. This document illustrates this class of networks
introduced together with Montasser Ghachem in \cite{GK1,GK2}. 
Besides showing off pictures, we look at elementary results related to
the Chinese remainder theorem, the Collatz problem, the
Artin constant, Fermat primes and Pierpont primes.
\end{abstract}
\maketitle

\section{Introduction}

In September 2013, we stumbled upon networks generated by finitely many maps $T_i$ 
on a ring $R$ \cite{GK1,GK2}. The rule is that two different points $x,y$ in $R$ 
are connected if there is a map $T_i$ from $x$ to $y$. Some constructions of finite
simple graphs can be seen below in this document.
The idea is based on the old concept of {\bf Cayley graphs} which visualizes finitely presented groups
equipped with finitely many generators $T_i(x) = a_i x$ on groups. For a single transformation 
$T$ on a ring $R$, one has a dynamical system: $T: R \to R$. The networks visualize the orbit 
structure of these systems if we think of the monoid generated by $T$ as ``time". We disregard here however
the digraph structure, self-loops and multiple connections and look at finite simple graphs only. 
As in complex dynamics, where the simplest polynomial maps already produce a rich variety of 
fractals, the discrete structures can be complex. The maps $T_i$ on the ring do not have to be algebraic, 
they could be any permutation on a finite set. One can see these networks as graph homomorphic images
of Cayley graphs generated by the transformations in the permutation groups of the vertex set. 
An other motivation comes from dynamical systems theory: since a computer always simulates
a system on a finite set, it is interesting to see what the relation between this discrete
and continuum is. The relation between the continuum and arithmetic systems have been investigated
for example in \cite{vivaldi,Rannou,lanford98}; a dynamical system like $T(x) = cx(1-x)$ on the interval 
$[0,1]$ is realized on the computer as a map on a finite set leading to finitely many cyclic attractors. 
For two or more maps, this can become a rather complex and geometric network.  \\

The fact that simple polynomial maps produce arithmetic chaos is exploited by {\bf pseudo 
random number generators}. There is a mild justification in that random variables like
$x,T(x)=x^2+c$ are asymptotically independent \cite{knillprobability} on $Z_n$ in the limit
$n \to \infty$. While the pseudo random nature given by the quadratic map $Z_n$ is unclear, the orbits are
sufficiently random to be exploited or example in the ``Pollard rho method" for integer factorization \cite{Riesel}.
Maps like $x \to x^2$ modulo $p$ or cellular automata maps were visualized in \cite{WolframState1,WolframState2}
using state transition digraphs. \\

What is new? By allowing more transformations, we can get more structure and more variety. Our point of view is
motivated heavily from the study of deterministic and random networks. 
The generalization from groups to monoids and by looking at the state space instead of the group itself, we 
break symmetries: while Cayley graphs look more like rigid crystals, the monoid graphs tend to produce 
rather organic structures resembling networks we see in social networks, computer networks,
synapses, chemical or biological networks and especially in peer to peer networks.
This happens already in the simplest cases: we can for example take quadratic maps $T(x) = x^2+a, S(x)=x^2+b$ 
or affine maps like $T(x)=3x+1$ and $S(x)=3x+1$ on $Z_n$. 
Whether we take affine, or nonlinear or arithmetic functions, the networks obtained in such a dynamical way 
can be intriguing:  \\

{\bf a)} The dynamical graphs show {\bf visually interesting structures} which bring arithmetic relations to live. 
{\bf b)} Their {\bf statistical properties} of path length, global cluster and vertex degree are interesting. 
{\bf c)} Some examples lead to {\bf deterministic small world examples} with small diameter and large cluster. 
{\bf d)} The graphs display {\bf rich-club phenomena}, where high degree nodes are more interconnected.
{\bf e)} Many  feature {\bf garden of eden states}, unreachable configurations like transient trees.
{\bf f)} They can feature {\bf attractors} like cycle sub graphs but also more complex structures.
{\bf g)} By definition, these graphs are {\bf factors of Cayley graphs} on the permutation group of V.
{\bf h)} They are {\bf universal} in the sense that any finite simple graph can be obtained like that. 
{\bf i)} In many cases, classes of networks produce natural probability spaces as we can parametrize maps.
{\bf j)} In certain cases, the graphs appear to be {\bf triangularizations of manifolds} or varieties. 
{\bf k)} In the arithmetic case, the topology like connectivity and dimension leads to Diophantine problems.   \\

Some pictures can be seen at the end of the article. Here are two experimental observations: \\

{\bf A)} (\cite{GK1}) We measure that the mean length $\mu(G)$ and the global clustering coefficient $\nu(G)$ 
have the property that $\lambda(G) = -\mu(G)/\log(\nu(G))$ often has a compact limit set if the number of nodes go
to infinity. When choosing random permutations and averaging, we see actual convergence in the limit $n \to \infty$. 
The two quantities $\mu(G)$ and $\nu(G)$ are essential to see small world phenomena as seen in \cite{WattsStrogatz}.
A reasonable conjecture is that in the probability space of all pairs of random permutations $f,g$ on $Z_n$,
the random variable $\lambda(G(f,g))$ has an expectation which converges for $n \to \infty$. There 
is strong numerical evidence for that. In \cite{GK1} we also showed how one can naturally construct
large bipartite or multipartite graphs using dynamical constructions. \\

{\bf B)} (\cite{GK2}) We get deterministic examples of networks which feature all the 
statistical properties of Watts-Strogatz \cite{WattsStrogatz} in the sense that $\mu,\nu$ and vertex distributions
behave in the same way. The statistic is almost indistinguishable from W-S. These examples are of the 
form $T_i(x) = [x^{1+\epsilon_i} + i]$, where at least one $\epsilon_i$ is $0$ and 
the others are equal to $p$, a permutation parameter.
For $\epsilon_i=0$ and $k$ maps, we have the initial wiring setup of Watts-Strogatz for $p=0$. 
While in Watts-Strogatz, the rewiring is done in a probabilistic way, this is taken care
by increasing the nonlinearity. If the maps are nonlinear but close to linear, we see interesting
geometric features appearing discussed in \cite{GK2}. For example: 
if maps are close to linear maps, interesting topological structures can appear. \\

In this paper, we prove a couple of elementary which indicate how these graphs
can relate to elementary number theory. To do so, we have questions from dynamical systems as a guide. 
The subject of networks has exploded in the last decade, as a 
look onto the library shows \cite{Goyal,BornholdtSchuster,GoodmanORourke,CohenHavlin,newman2010,nbw2006,
WassermanFaust,Jackson,SmallWorld, BallobasKozmaMiklo,ibe,vansteen,newman2010,Meester,Easley,shen}.
It has been made accessible to a larger audience in books like 
\cite{SixDegrees,Buchanan, Linked,Sync,Connected}.  \\

The topics in the next sections are in an obvious way motivated from corresponding problems in 
complex dynamics, where the question of connectivity and dimension of the Julia sets is of interest. One can 
also look at the analogue of the Mandelbrot set, the set of parameters for which graphs 
are connected. If we have a class of dynamically generated graphs, we can ask how the clique size is distributed
on the parameter space. As in dynamical system terminology, we look for the 
connectivity locus and the dimension of the object. The subject can lead to relatively simple but unexplored 
questions. It is suited for experimentation (as we have accessed it ourselves primarily) and is 
almost unexplored. We illustrate this by formulating some simple questions related to connectivity. 
Computer algebra code will be available on the project website and the Wolfram demonstration project. 

\begin{figure}
\scalebox{0.43}{\includegraphics{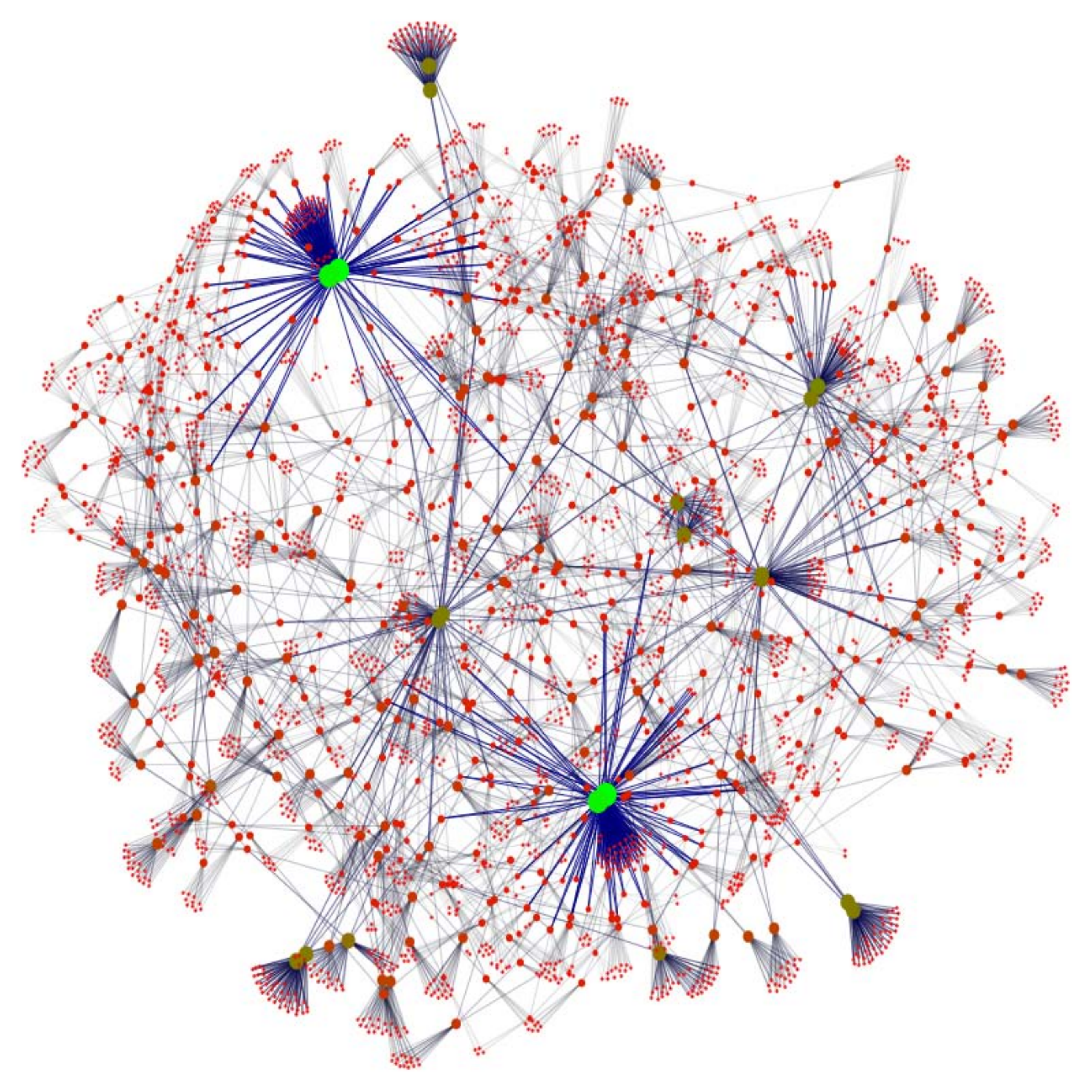}}
\caption{
A graph is generated by $f(x)=x^2+1,g(x)=x^2+2$ on $Z_{3000}$. It has diameter $9$,
average vertex degree $3.99$, characteristic path length $\mu=5.8$, mean clustering $\nu=0.00074$
and a length-cluster coefficient $\lambda=0.806994$.
}
\end{figure}

\section{Affine maps and smooth numbers}

We first look at graphs generated by a single affine map $T(x)=ax+b$ on the ring $Z_n=Z/(nZ)$. 
Given parameters $(a,b)$, for which $n$ is the graph connected on $Z_n$? Lets 
assume first $T(x)= a x$ on $Z_n$. The following result will be used later on: 

\begin{lemma}
The graph on $Z_n$ generated by $T(x)=2x$ is connected if and only if $n$ is a power of $2$. 
\label{lemma1}
\end{lemma}
\begin{proof}
If $n$ is a power of $2$, then $T^k x = 2^k x$ is divisible by $n$ if 
$k$ is larger or equal than $n$. We see that every $x$ is eventually attracted by $0$ and
that the graph is a tree. If $n = p 2^m$ with $p$ being relatively prime to $2$, we can look 
at the orbit $T^k x = x 2^k$ modulo $p$. If $x$ is divisible by $p$, then $T^kx$ stays 
divisible by $p$ and the graph is not connected.
\end{proof}

If we look at $Z_n^*$, we get additionally some primes. Which ones? 
We see that $3, 5, 11, 13, 19, 29, 37, 53, 59, 61, 67, 83, 101, \dots $ lead to connected graphs while
$7, 17, 23, 31, 41, 43, 47, 71, 73, 79, 89, 97, 103, \dots $ lead to disconnected ones. 

\begin{propo}[Miniature I: Artin]
The graph on $Z_n^*$ generated by $T(x)=2x$ has one component if and only if $n$
is a power of $2$ or a prime $p$ for which $2$ is a primitive root.
\end{propo}
\begin{proof}
We look at the dynamics on $Z_n^* = Z_n \setminus \{0\}$, where $n$ is a prime. 
If $2$ is a primitive root modulo $n$, then by definition, 
the orbit $ \{ 2^k \; {\rm mod} \; n \; \}$ covers the multiplicative group $Z_n^*$
and the graph is connected. If $2$ is not a primitive root, then there exists $x$
for which the discrete logarithm problem $2^k = x \; {\rm mod} \; n$ has no
solution. The graph is not connected. 
\end{proof} 

{\bf Remarks.} \\
{\bf 1)} It is an open problem to determine the fraction of the set
of primes is for which $2$ is a primitive root. 
Among the first $n=10^6$ primes, 374023 have this property. 
A conjecture of Artin implies that this probability should converge to the 
{\bf Artin constant} $\prod_{p \; {\rm prime}} (1-1/(p (p-1))) = 0.3739558...$. \\
{\bf 2)} There are analogous results, when $2$ is replaced with an other prime.
What matters for prime $n$  whether $a$ is a primitive root in the field 
$Z_n$ or not. 

\begin{defn}
Given a finite set $P$ of primes, we call the set of all products 
$\{ \prod_{p_i \in P} p_i^{n_i} \; | \; n_i \geq 0 \; \}$ 
the set of all {\bf $P$-smooth numbers}. If $P$ is the set of $P$-smooth 
numbers, lets call the set $P \cup 2P$ the set of 
{\bf double $P$-smooth numbers}. If $a$ is an integer, we call the set of 
numbers which have only prime factors from $a$ to be {\bf $a$-smooth }.
Similarly, if $a,b$ are integers, the set of numbers which have only prime factors 
from $a,b$ are called $(a,b)$-smooth. 
\end{defn}

Double smooth numbers appear as the "connectivity locus" in the case $x \to 3x+1$: the graph is connected for
$n=1, 2, 3, 6, 9, 18, 27, 54, 81, \dots $ which is the set of double $\{3\}$ -smooth numbers,
numbers which are a power of $3$ or twice a power of $3$. We do not have a complete picture yet but state only: 

\begin{lemma}
The set of $n$ which are connected for $T(x)=ax+b$ is a subset of all $\{ a,a-1 \; \}$ smooth numbers.
\label{lemma2}
\end{lemma}
\begin{proof}
If $q$ is a prime factor of $n$ which does not divide $a-1$, then the graph 
is not connected: there is a congruence class
modulo $q$ which is a fixed point of $T$. The reason is that $a x+b = x \; {\rm mod}(q)$
has a solution $x=-b (a-1)^{-1}$. 
For example, if $T(x) = 5x+1$ and $n=21$ which has a factor $q=3$, then the 
congruence class $2$ modulo $3$ is a fixed point of $T$. Since this
congruence class is invariant, the graph is not connected. 
If $a$ divides $n$, then $n=k a$ and $T^a$ produces a translation on $Z_q$
and there is a chance that we have several graphs. If $a$ is prime and 
does not divide $n$,
\end{proof}

Here are some results, where we denote by $(p_1,...,p_n)$ the $P=\{p_1,...,p_n\}$ smooth 
numbers and by  $(2^*,p_1,..p_n)$ the double $P$ smooth numbers. We computed the table by 
constructing the graphs, seeing which are connected and matching it with $P$ smooth sequences.

\begin{tabular}{|l|lllllll|} \hline
       & $b=0$    & $b=1$            &  $b=2$         & $b=3$           & $b=4$         & $b=5$               &  $b=6$ \\ \hline
$a=2$  & $(2)$    & $(2)$            &                &                 &               &                     &        \\
$a=3$  & $(3)$    & $(2^*,3)$        &  $(3)$         &                 &               &                     &        \\
$a=4$  & $(2)$    & $(2,3)$          &  $(2,3)$       & $(2)$           &               &                     &        \\
$a=5$  & $(5)$    & $(2,5)$          &  $(5)$         & $(2,5)$         & $(5)$         &                     &        \\
$a=6$  & $(2,3)$  & $(2,3,5)$        &  $(2,3,5)$     & $(2,3,5)$       & $(2,3,5)$     & $(2,3)$             &        \\
$a=7$  & $(7)$    & $(2^*,3,7)$      &  $(3,7)$       & $(2^*,7)$       & $(3,7)$       & $(2^*,3,7)$         & $(7)$  \\
$a=8$  & $(2)$    & $(2,7)$          &  $(2,7)$       & $(2,7)$         & $(2,7)$       & $(2,7)$             & $(2,7)$  \\ \hline
\end{tabular}

We always get smoothness sequences or double smoothness sequences involving the prime factors of $a$ and $a-1$.
For $T(x)=px$ with prime $p$, we have connectivity for $\{p\}$-smooth numbers $\{p,p^2,p^3, ... \}$. 
For $T(x)=px+1$ with prime $p$ we have connectivity for $\{p,(p-1)\}$ smooth numbers if $p-1$ is divisible by $4$
and double $\{p,P(p-1) \setminus \{2\} \}$-smooth numbers if $p-1$ is divisible by $2$.  \\

We should also look at the case $a=1$. Now $T^k(x) = x+k b$ and the graph is connected if and only if 
$b$ has no common divisor with $n$. If $R=Z_n^m$ and $T(\vec{x}) = x+ k \vec{b}$, then the graph is connected if 
all $b_i$ have no common divisor with $n$. This is the Chinese remainder theorem. 

\section{Quadratic maps and Fermat primes}

An other simple example of a dynamically generated graph is obtained with $T(x)=x^2$ on $Z_n$. 
In \cite{WolframState1}, this is attributed to a suggestion of Stan Wagon. How many components 
does the graph have? We see that for $n=2^k$ with $k \in N$, there are two components, the even 
and odd numbers.

\begin{defn}
A prime of the form $n=2^{2^k}+1$ is called a {\bf Fermat prime}. 
\end{defn} 

\begin{propo}[Miniature II: Fermat]
The graph on $Z_n^*$ is connected if and only if $n=2$ or if $n$ is a Fermat prime $n=2^{2^k}+1$.
\end{propo}
\begin{proof}
If $n$ is not prime $n=pq$, then the orbit of $x=p$ has the property that every point
$T^k(x)$ is divisible by $p$ modulo $n$ and the graph is not connected. 
We can therefore assume that $n$ is a prime. 
This allows find a primitive root $a$ and write $Z_n^* = \{a^k\}$.
If $n$ is a Fermat prime, then the graph is a tree centered at $1$. The
elements $1$ and the quadratic non residues have one neighbor, while the
quadratic residues have $3$ neighbors $x^2, \pm \sqrt{x}$.
If $n$ is not a Fermat prime, then $n-1$ is not a power $2$ and has
therefore a factor $q$ different from a power of $2$. We look now on the 
dynamics of the system $x \to 2x$ modulo $q 2^l$.  We have seen in Lemma~(\ref{lemma1})
that the graph is connected if and only if $q=1$. In other words, the graph is connected if
and only if $n$ is a Fermat prime. 
\end{proof} 

The only Fermat primes known are $F_0,F_1,F_2,F_3,F_4$. 
It would be interesting to know the Euler characteristic of the graphs $G_n$ generated 
by $T(x)=x^2$ on $Z_n$.  The list of Euler characteristics starts with 
$$ 1, 2, 2, 2, 2, 4, 3, 2, 3, 4, 2, 4, 3, 6, 4, 2, 2, 6, 3, 4, 6, 4, 2, ...  \; . $$
The graphs do not need to be simply connected. An example is $G_{59}$, a case with $3$ components
and one large cycle. To get triangles, we have to solve $((x^2)^2)^2=x^8=x$ which is only possible if 
$x-0$ or $x^7-1$ is a multiple of $n$. Indeed, for $n=127$ we get the first such graph with $2$
triangles. Since graphs generated by one map never has a tetrahedron, the Euler characteristic
of a graph $G_n$ is $v-e+f$, where $v$ is the number of vertices, $e$ the number of edges and 
$f$ the number of triangles.  \\

If we take $T(x)=x^3$, we see no nontrivial graphs with $1$ or $2$ components. 
Graphs with $n=3^k$ have three components.  
For $T(x)=x^n$ with even $n$ we have the same list of graphs with $2$ components
as in the case of $n=2$. For $n=5$, the list of integers on which $T(x)=x^5$  has 
a graph with $3$ components starts with $3,4,11,251, \dots  $. \\

In the case $T(x) = 2^x$, the distribution of the number of components is smaller
and we measure about $M/\log(M)$ graphs among all graphs with $n=1,\dots ,M$ which have
one component. Now, we can ask for which $n$ the graphs generated by $x \to 2^x$ 
are connected. 

\begin{figure}
\scalebox{0.30}{\includegraphics{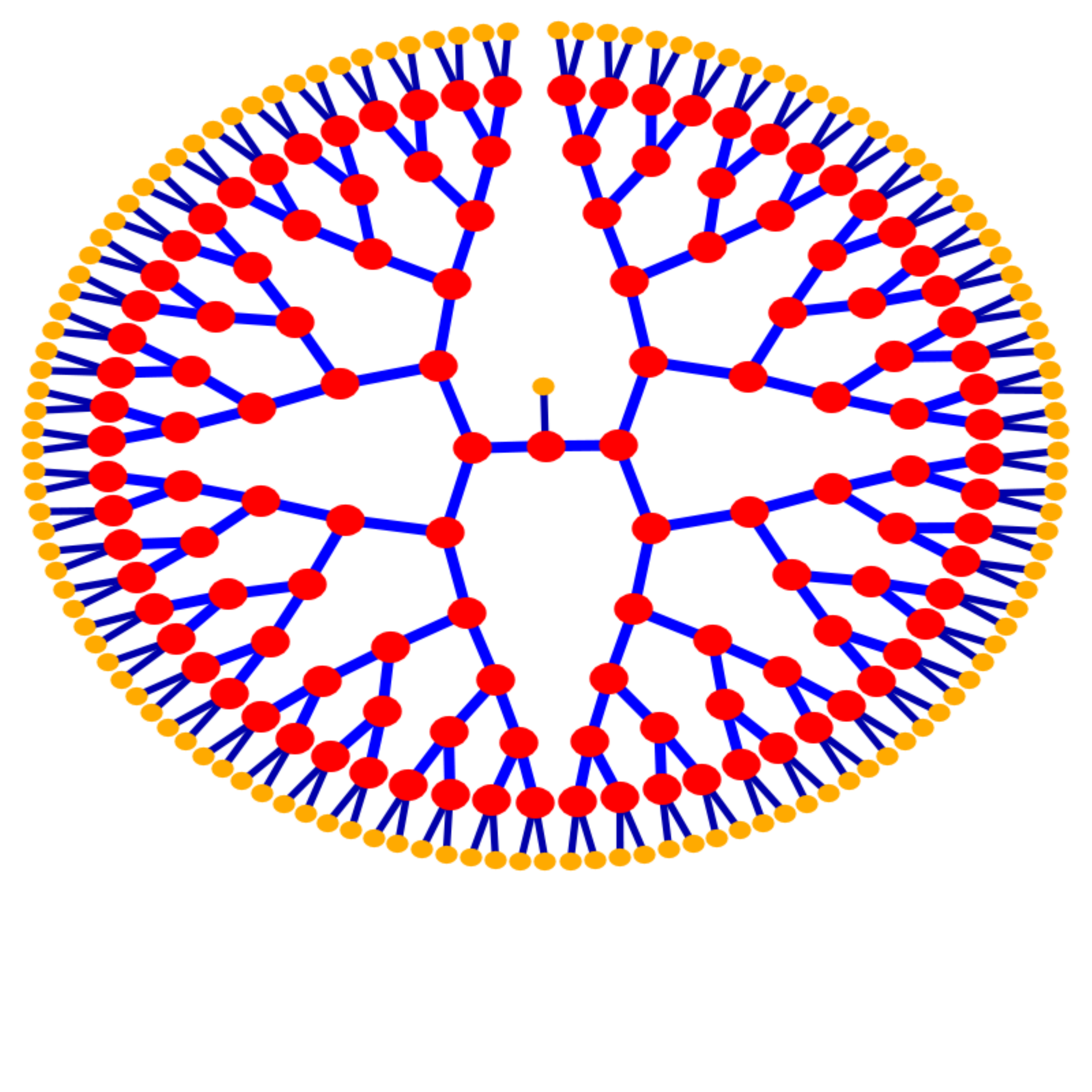}}
\caption{
The Fermat graph for $F_3 = 2^{2^3}+1=257$.
}
\end{figure}

\begin{figure}
\scalebox{0.30}{\includegraphics{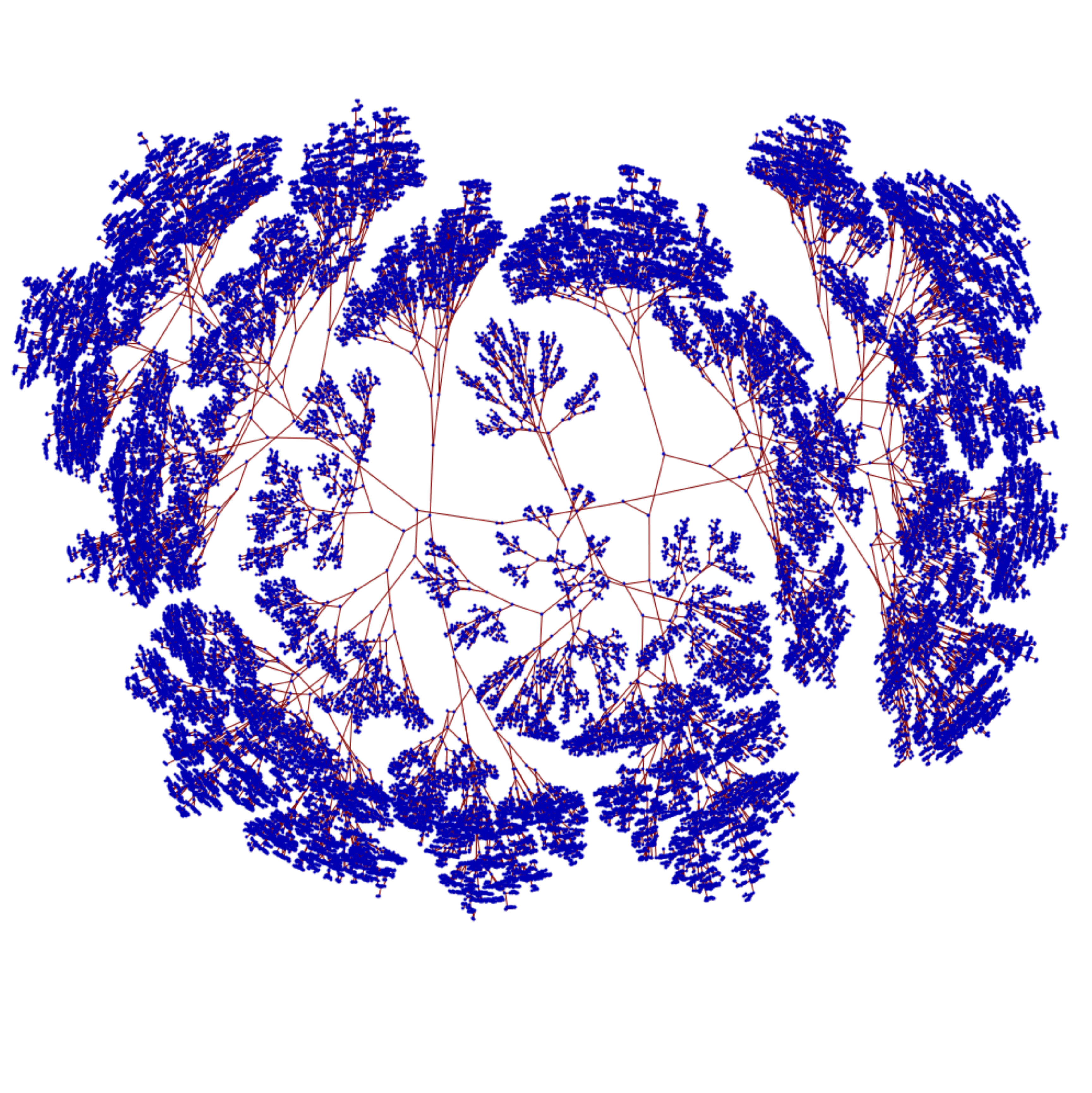}}
\caption{ 
The Fermat graph for the largest known Fermat prime $F_4 = 2^{2^4}+1=65537$. }
\end{figure}

\section{A class of Collatz type networks}

This example is inspired by the famous Collatz $3x+1$ problem,  a
"prototypical example of an extremely simple to state, extremely hard to solve, problem" to cite
\cite{lagarias}. It deals with two maps: a number $n$ is divided by $2$ if it is even
and mapped to $3n+1$ if it is odd. The problem is whether every starting point converges
to $1$ when applying this rule. We can look at networks generated by affine maps
$T(x) = 2x$ and $S(x) = 3x+1$ on $Z_n$, disregarding the conditioning  and ask about the structure 
of this network. Of course, we can not expect any trivial relations with the Collatz problem any more. 
The graphs look surprisingly random. It is the non-commutativity of the monoid generated by 
$T,S$ as well as the conditional application of these two maps which makes the 
original Collatz problem difficult. When the graph is generated by $T(2x+1)=6x+4, T(2x)=2x, S(2x)=x/2,
S(2x+1)=x$, then this leads to the Collatz graph which is believed to be a tree. Note that the original
Collatz problem is generated by one map. We look at the two affine maps unconditionally. 

\begin{figure}
\scalebox{0.22}{\includegraphics{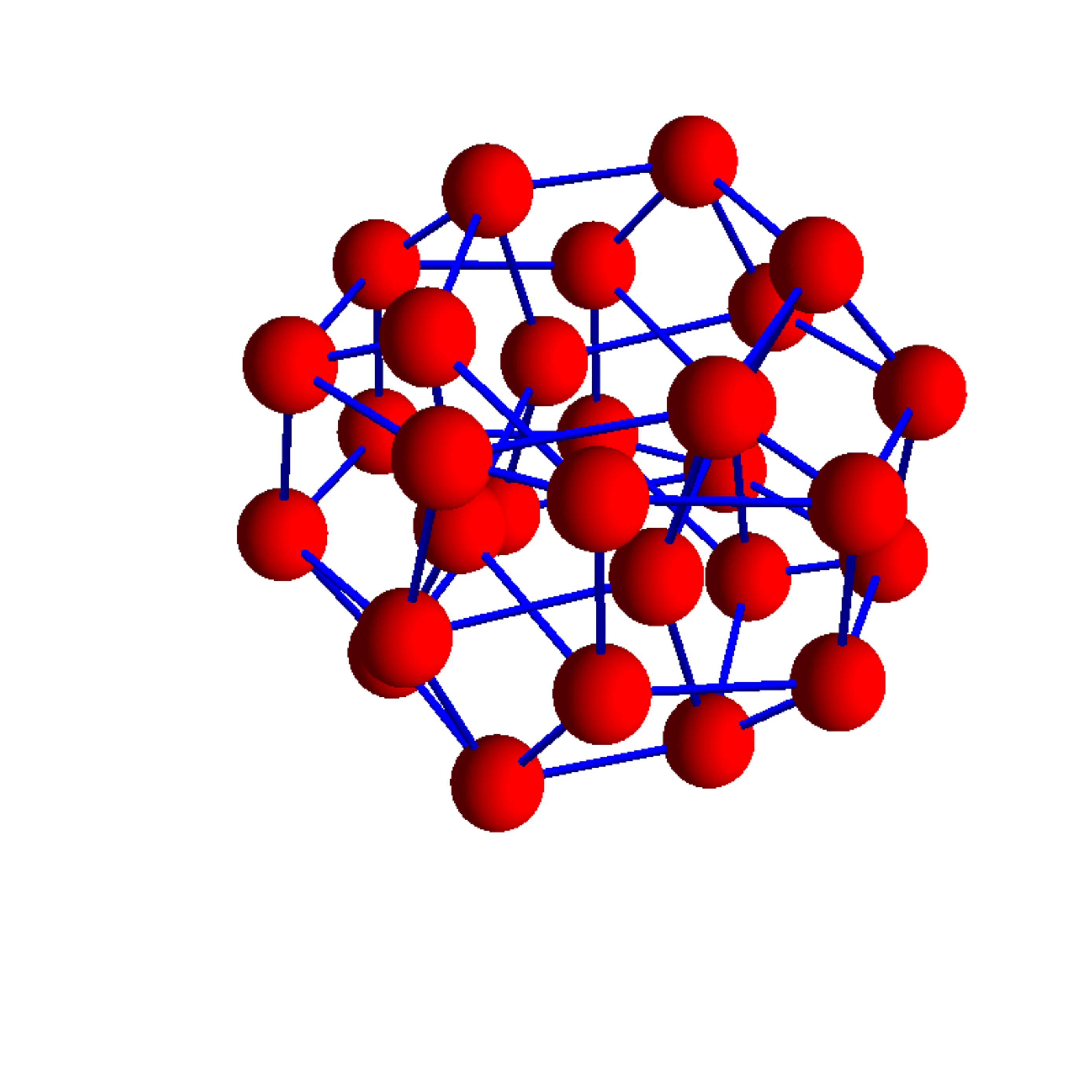}}
\scalebox{0.22}{\includegraphics{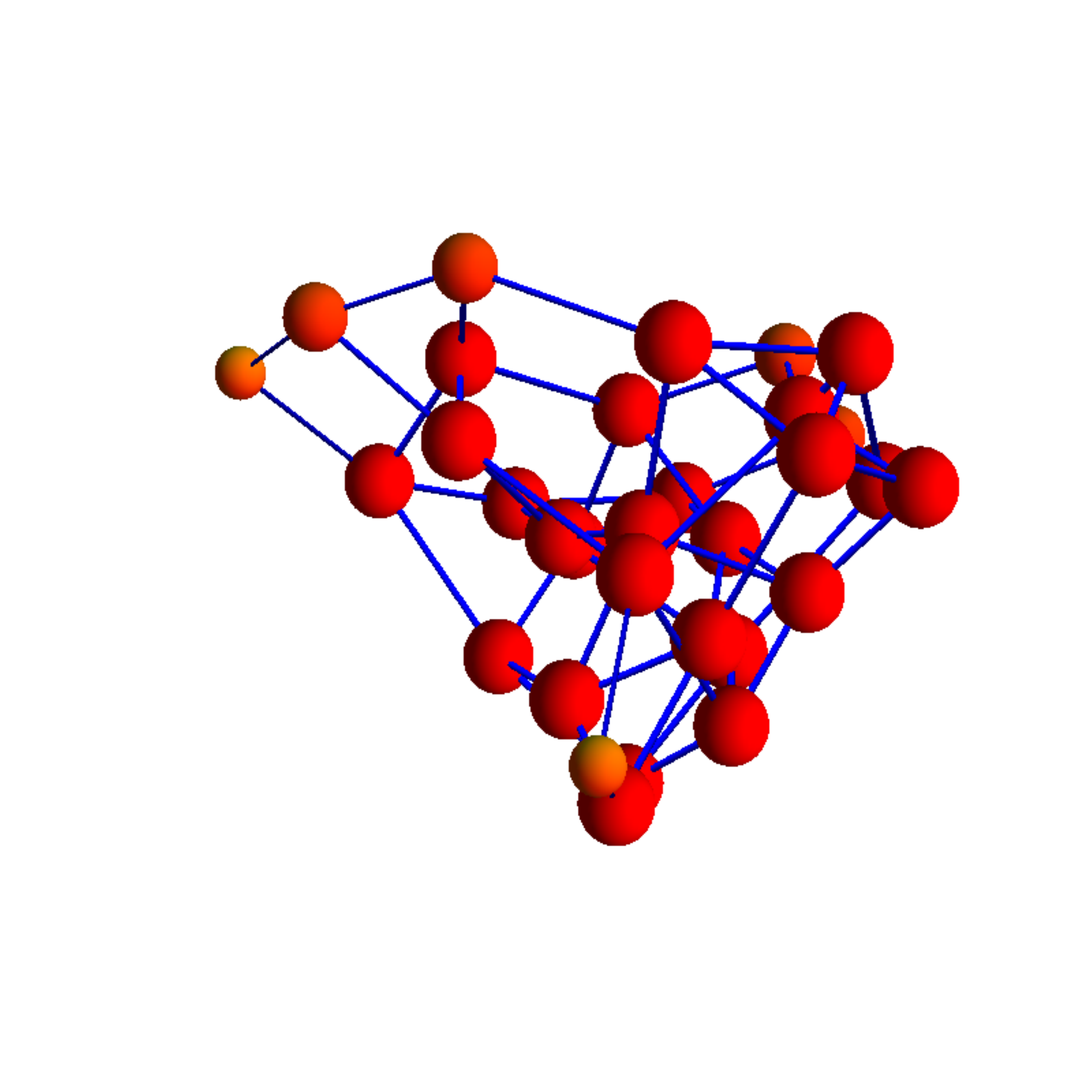}}
\caption{
The first graph shows the graph generated by $T(x)=2x,S(x)=3x$ on $Z_{31}$. This is
a Cayley graph on the multiplicative group. The second shows the
Collatz graph generated by $T(x)=2x,S(x)=3x+1$ on $Z_{31}$. The symmetry is broken:
the maps $T,S$ do not commute any more. }
\end{figure}

We see experimentally that all networks $C_n$ on $Z_n^*$ generated by the 
dynamical system $(Z_n,T(x)=3x+1,S(x)=2x)$ are connected. This is not related 
to the Collatz problem, where connections are only done under conditions leading to a digraph
called Collatz graph which if the Collatz conjecture is true, is a connected tree. 
The connectivity question is interesting for general affine maps. For $d \geq 2$-dimensional
graphs, non-connectivity (ignoring the isolated vertex $0$) 
is more rare but still can happen. For the graph generated by $5x+2,3x+1$ for example, we have for prime $n$
about $2/3$ connected graphs and $1/3$ disconnected; the number of connected components looks exponentially
distributed. For the graph generated by $5x+1,3x+1$ we appear to have connectivity for all primes. \\

The connectivity question can be seen as a non-commutative Chinese remainder theorem problem because we want to 
solve $T^{n_1} S^{n_2} \dots T_{n_k} S_{n_k} x = y$. In the commutative case, this reduces to
$T^n S^m x = y$. An example is $T(x) = a x$ and $S(x)=bx$ where the maps commute.
It is enough to assure that for every $x$ we can find $u,v$ such that $a^u b^v = x {\rm mod}(n)$. 
If either $a$ or $b$ is a primitive root of unity, then $b=a^l$ and we have the problem to solve
$a^{u+lv} = x$, which is no problem already for $v=0$.  \\

Here is an amusing fact for Collatz networks: 

\begin{propo}[Miniaturet III: Collatz]
All Collatz networks $C_n = (Z_n,2x,3x+1)$ have exactly $4$ triangles 
if $n$ is prime and larger than $17$. 
\end{propo}
\begin{proof}
Since $Z_p$ is a field if $n=p$ is prime, a linear equation modulo a prime can be solved
in a unique way. Let $T(x)=2x$ and $S(x)=3x+1$. Since $T^3(x)=8x=x$ has only the solution $x=0$
and $S^3(x) = 13+27x=x$ has a unique solution which is already a solution of $S(x)=x$, 
we see that every triangle must be formed with $2$ maps $T$ and one map $S$ or two maps $S$ 
and one map $T$. In each of the four possible cases $T^2(S(x))$ and $T(S(T(x))$
and $S^2 T(x)$ $STS(x))$ we have exactly one solution. 
Now for small primes, there are either less or more solutions.
For $n=13$ for example, the equation $S^3(x)=x$ is solved by any $x$. There are $6$ triangles. 
The combinations of three maps from $T,S,T^{-1},S^{-1}$ produces $4^3=64$ possible polynomials.
The largest coefficient which appears is $27$. For $p>27$ there are exactly $4$ solutions. 
We check by hand that also for $p=19,23$, there are exactly $4$ solutions. 
\end{proof}

\begin{figure}[H]
\scalebox{0.43}{\includegraphics{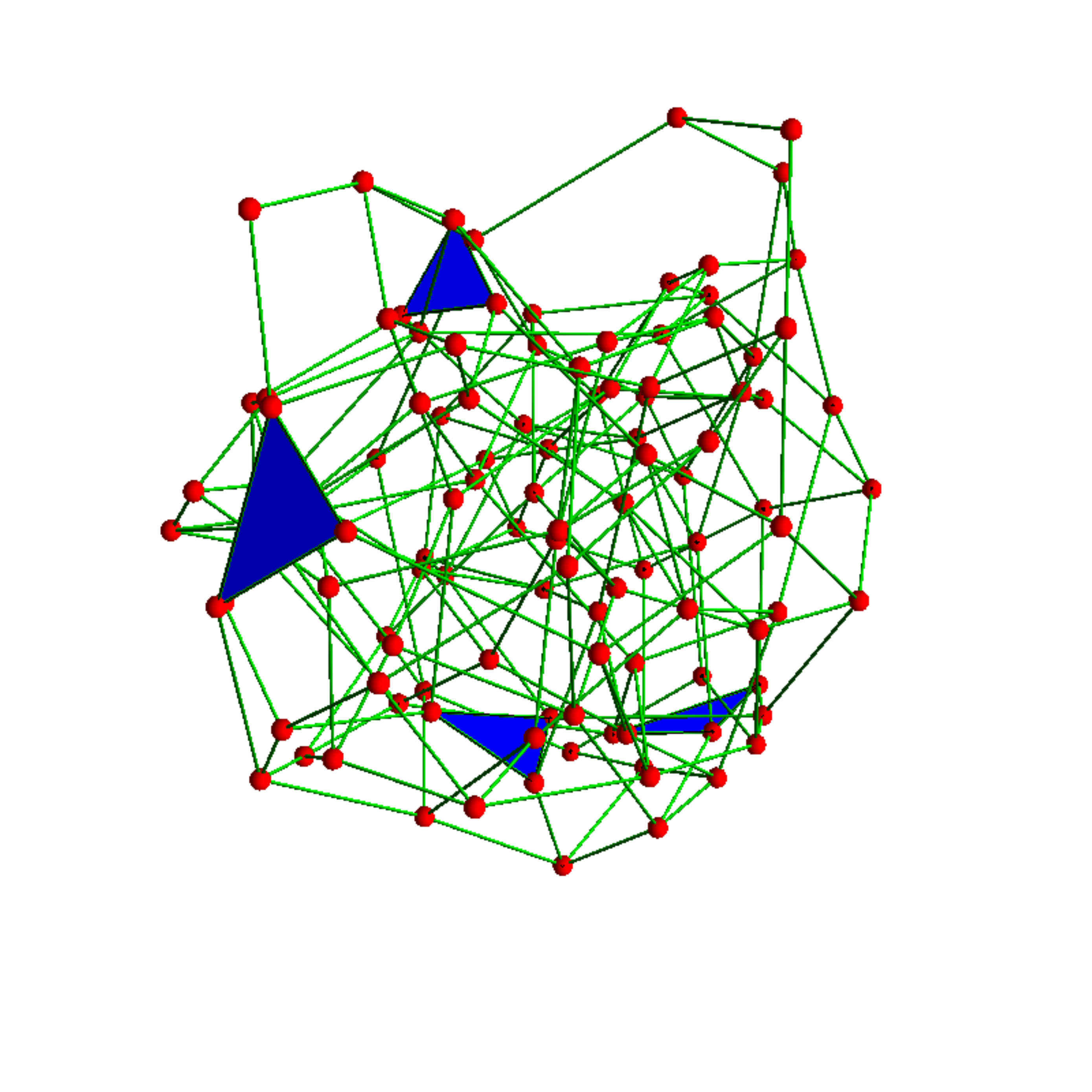}}
\caption{
The Collatz graph $C_{113}$ generated by $T(x)=3x+1,S(x)=2x$ on $Z_{113}$ with the four triangles.
}
\end{figure}

{\bf Remark}. This result is neither special nor universal. The dynamical graph generated by 
$2x+1,3x+1$ also has $4$ triangles for prime $n>17$.
Similarly, the graph generated by $5x+1,3x+1$, but the one generated by $T(x)=5x+2,S(x) =3x+1$ has
no triangles for prime $n>37$. Why? For $n=41$ for example, we have $T(S(T(20))=20$ but 
also $T(20)=20$ and $S(20)=20$. For prime $n$, the point $(n-1)/2$ is a fixed point of both $T$ and $S$
so that $x,T(x),S(T(x))$ is not a triangle but a point. 
What about non-prime $n$? It can be subtle for $n=pq$ already, where we can see different numbers
of triangles.  \\

The statistics of the Collatz networks on $Z_n$ generated by $T(x)=3x+1,S(x)=2x$ is typical 
among other dynamical networks. We see that the mean path length-mean cluster coefficient relation
$\lambda(C_n) = - \mu(C_n)/\log(\nu(C_n))$ converges for $n \to \infty$. 
The dimension of the Collatz graphs converges to $1$ for $n \to \infty$ as the number 
of triangles are rare. It would be nice to know the exact number of triangles in the general
case (not only for primes). This might gives hope that we might be able to
compute the Euler characteristic of $C_n$ exactly. 

\section{Arithmetic functions}

The map $T(x) = \sigma(x)-x$ gives the sum of the proper divisors of $x$. Fixed points of 
$T$ are called {\bf perfect numbers}. While the structure of even perfect numbers is known it is 
an ancient problem whether there are odd perfect numbers. 
Periodic points of period $2$ are {\bf amicable numbers}. It is not known whether there are infinitely
many.  The dynamical system generated by $T$ on $\Z$ is the Dickson dynamical system \cite{dicksonI}. 
We call the graphs on $Z_n$ generated by $T$ Dickson graphs. Their topological properties depend 
on number theoretical properties. We can now play with other graphs like graphs generated by $T(x)$
and $S(x)=x+1$. 

\begin{figure}[H]
\scalebox{0.35}{\includegraphics{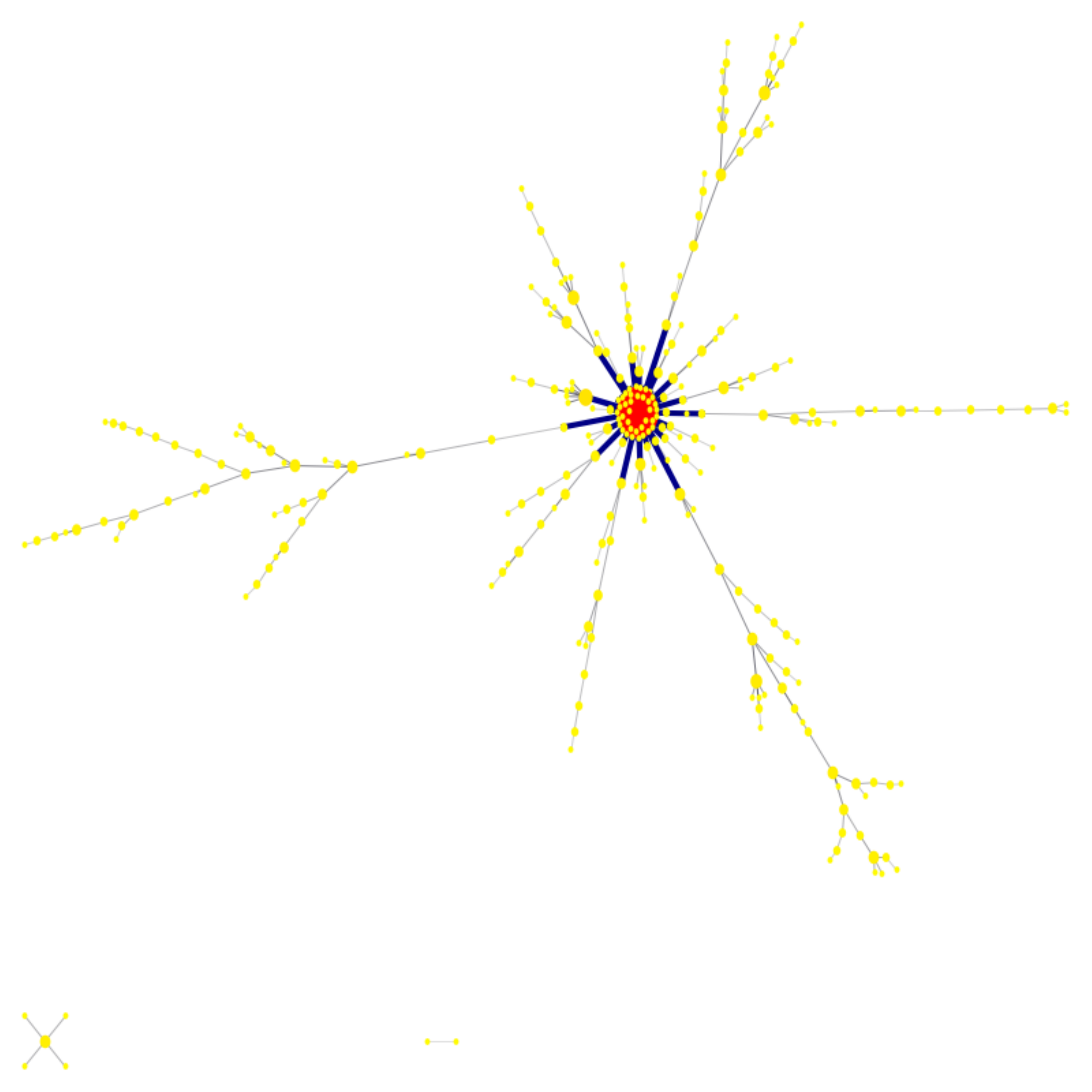}}
\scalebox{0.35}{\includegraphics{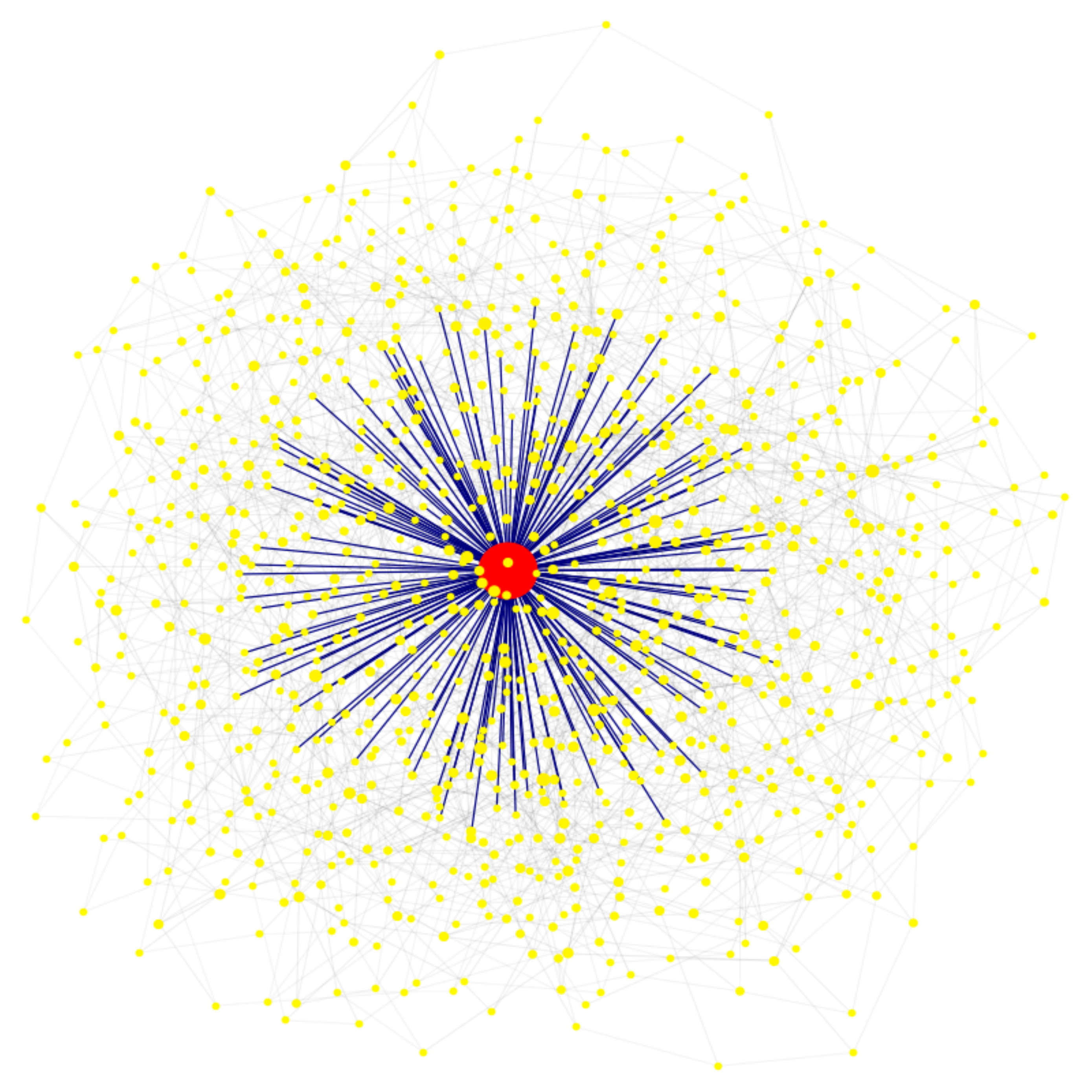}}
\caption{
Dickson graphs: the first graph is generated by $T(x)=\sigma(x)-x$. 
The second is the graph generated by by $T$ and $S(x)=x+1$.
}
\end{figure}

As the Dickson dynamical system illustrates, questions involving the iteration of arithmetic
functions can be difficult. The oldest known open problem in mathematics is related to it.  

\section{Pierpont primes}

One of the simplest nonlinear cases with two dimensional time 
is $T(x) = x^2$ and $S(x) = x^3$ on $Z_n$. This is a case where the monoid is commutative. 
One can ask, for which $n$, the graph on $\{2, \dots, n-1\}$ generated by $x^2,x^3$ is connected.
This can be answered because $T(x) = x^2$ and $S(x)=x^3$ commute. We have 
$T^n(S^m(x)) = x^{2^n 3^m}$. 

\begin{defn}
A prime $n$ is called a Pierpont prime if is of the 
form $2^t 3^s+1$, where $s,t$ are integers. 
\end{defn}

The list of Pierpont primes starts with 
$$  2,3,5,7,13,17,19,37,73,97,109,163,193,257,433,487,577, ... \; . $$
It is unknown whether there are infinitely many.  Gleason has shown that these primes play 
an important role with ruler, compass and angle trisection \cite{gleason88}. 
He was also the first to conjecture that there are infinitely many Pierpont primes.
Here is an elementary proposition which illustrates the relation 
between arithmetic and graph theory: 

\begin{propo}[Miniature IV: Pierpont]
The graph on $Z_n^*$ generated $T(x)=x^2$ and $S(x)=x^3$ is connected 
if and only if $n$ is a Pierpont prime. 
\end{propo}

\begin{proof} 
If $n=pq$, then any multiple of $p$ will remain a multiple of $p$ after applying $T$ or $S$. 
This implies that the graph is not connected. We therefore need that $n$ is prime. 
If $n$ is prime, then the multiplicative group of $Z_n$ is cyclic. A generator is a primitive
root. We can now look at indices (discrete logarithms) and get maps $x \to 2x$ and $x \to 3x$ on the 
multiplicative group $\{1, \dots ,n-1 \; \}$. 
Either $n$ is even and equal to $2$ in which it is a Pierpont prime or $n$ is odd and of the form 
$2^t k + 1$, where $k$ is an integer. 
The proof is concluded with two things: (i) if $k=3^s$, then the graph is connected. (ii) If $k$ has a 
factor different from $2$ and $3$, then the graph is disconnected. 
Part (i) is clear because the maps $U(x) = 2x, V(x) = 3x$ on $Z_{n-1} = Z_{2^t 3^s}$ eventually lead to 
$0$ modulo $2^t 3^s$, showing that the graph is connected. 
Part (ii) follows from Lemma~(\ref{lemma2}). 
\end{proof} 

This can be generalized.  For example a graph generated by $x^2,x^5$ is connected if and only if the primes are
of the form $2^t 5^u+1$. The sequence of these primes is $2,3,5,11,17,41,101,...$. \\

\begin{figure}[H]
\scalebox{0.34}{\includegraphics{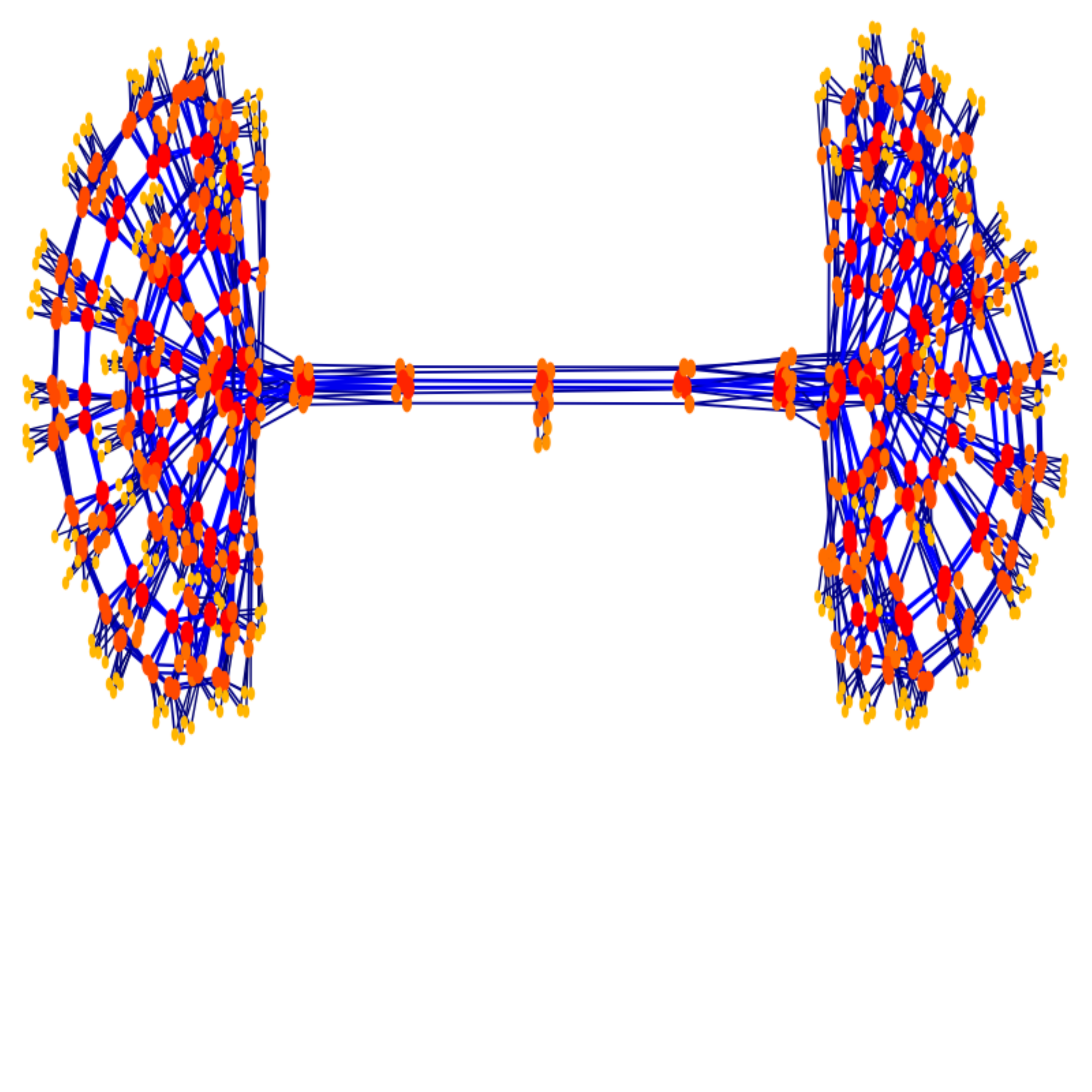}}
\scalebox{0.34}{\includegraphics{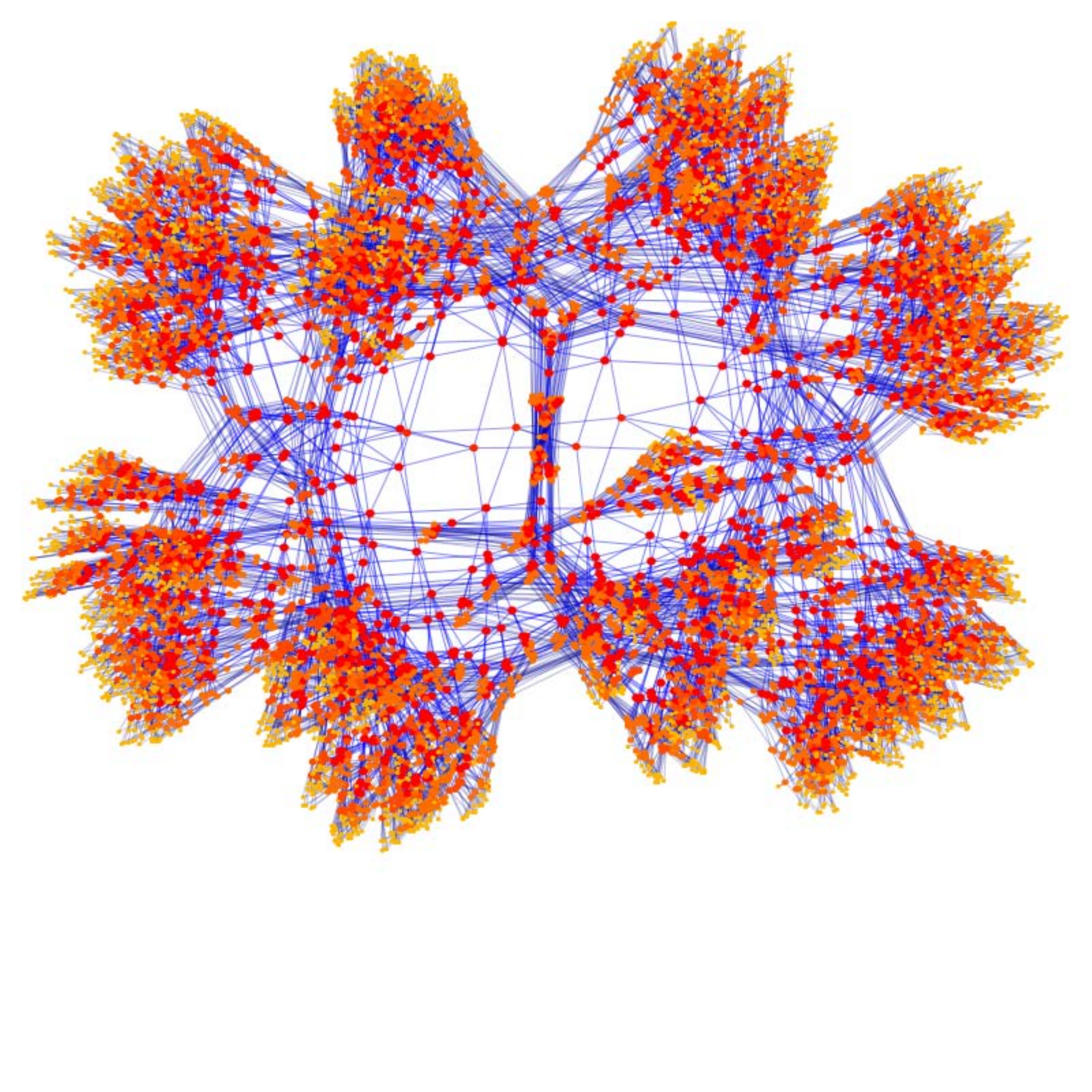}}
\caption{ 
Pierpont graphs generated by $x^2,x^3$ are connected on $Z_n^*$ if $n$ is a 
Pierpont prime. We see three cases: $n=769$ and $n=10369$. 
}
\end{figure}

\section{Rings with more structure}

Lets take the non-commutative ring $R=M(2,Z_n)$. If $T(x)=x^2$, then since diagonal matrices 
are left invariant by $T$ we can not have one component but have at least two components in the 
graph. We need a Fermat prime $n$ to have two components. 

\begin{figure}
\scalebox{0.32}{\includegraphics{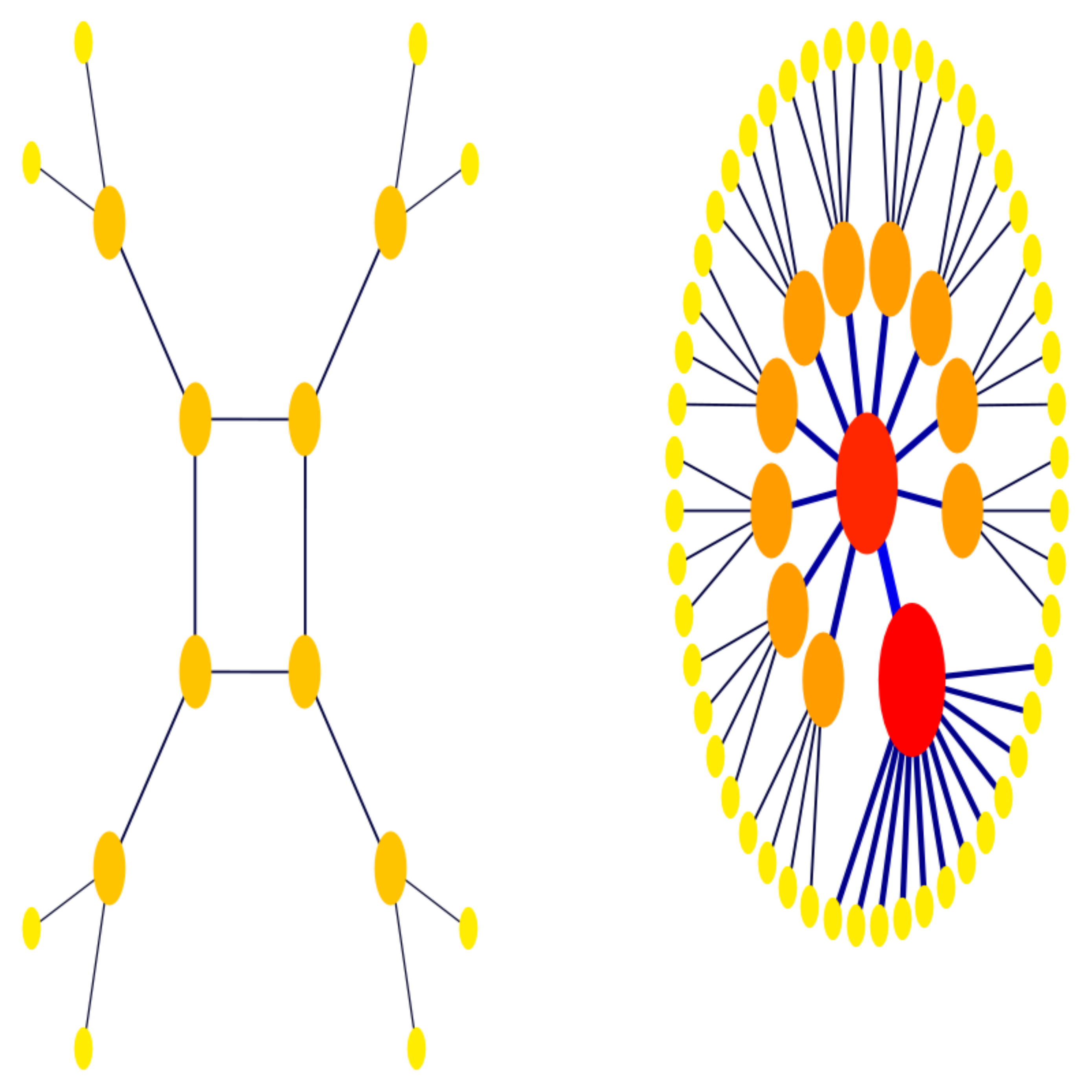}}
\caption{
The graph given by of $T(x) = x^2$ on the ring $R$ of upper triangular $2 \times 2$ matrices
over $Z_5$ has two main components. There are other components associated to matrices with 
zero eigenvalues.
}
\end{figure}

Lets stay with $M(2,Z_n)$ and consider maps $T_i(x) = x^2 +c_i$. 
The networks generated like that look more complex but the statistics is not much different
than in the commutative case. 

\begin{figure}[H]
\scalebox{0.32}{\includegraphics{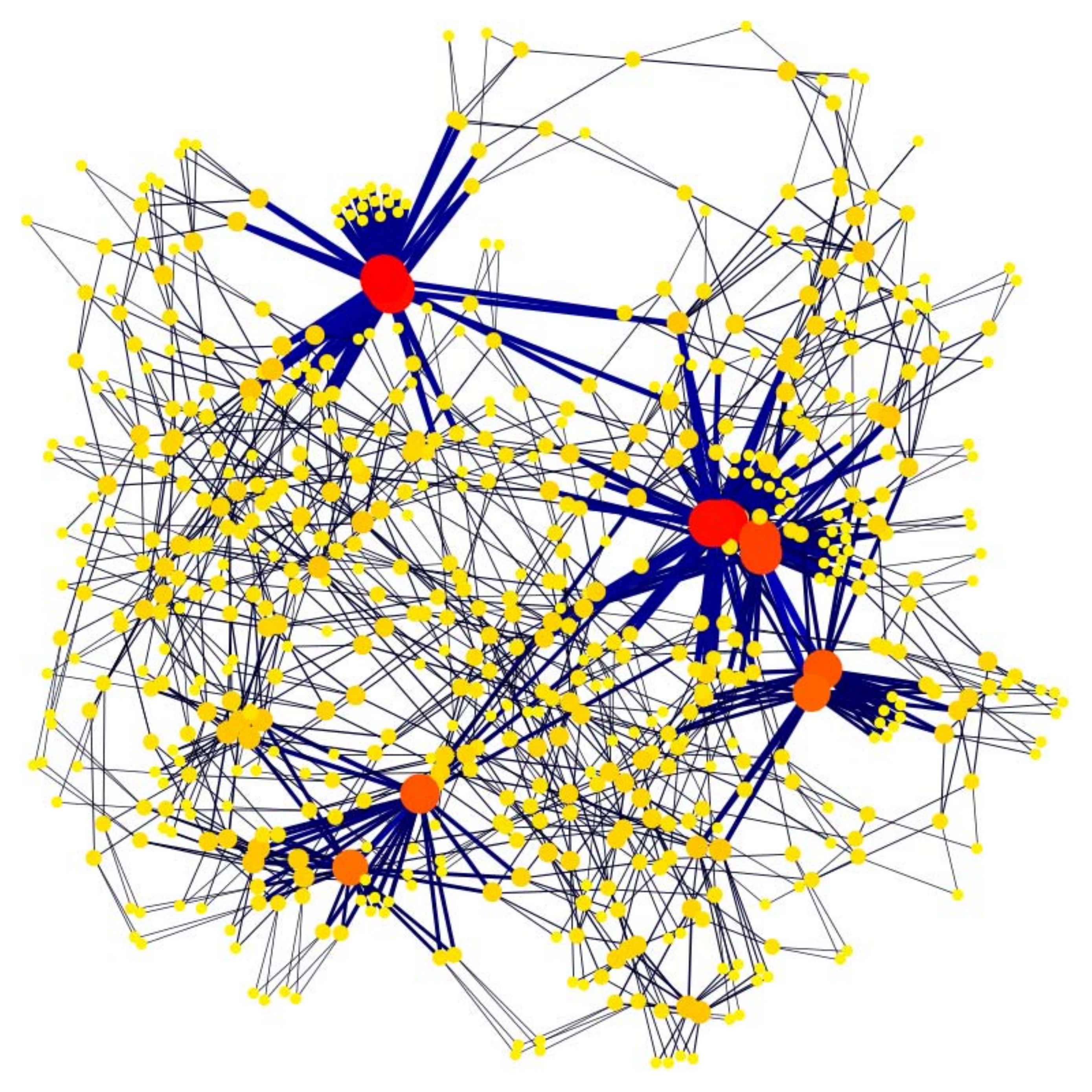}}
\caption{
An example of graph using a noncommutative ring. We take the graph generated by the two
quadratic transformations $x \to x^2 + A$
%\begin{array}{cc} 3 & 2 \\ 1 & 1 \\ \end{array} \right]$
and $x \to x^2 + B$  
%\left[ \begin{array}{cc} 1 & 2 \\ 3 & 4 \\ \end{array} \right]$ 
on the matrix ring $M(2,Z_5)$. 
}
\end{figure}

We can also look at polynomial rings $R[x]/(p)$ where $p$ is a polynomial. 

\begin{figure}[H]
\scalebox{0.2}{\includegraphics{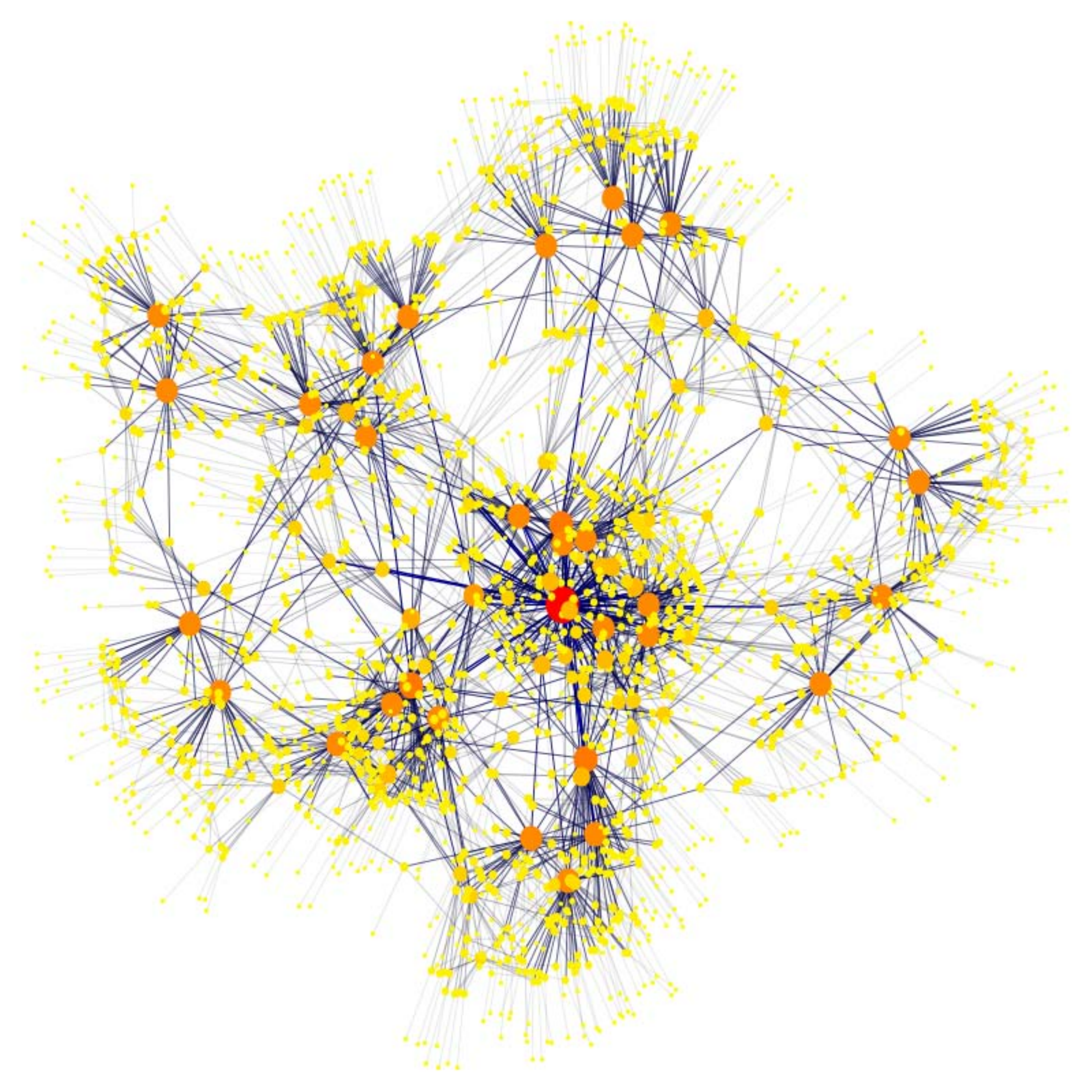}}
\scalebox{0.2}{\includegraphics{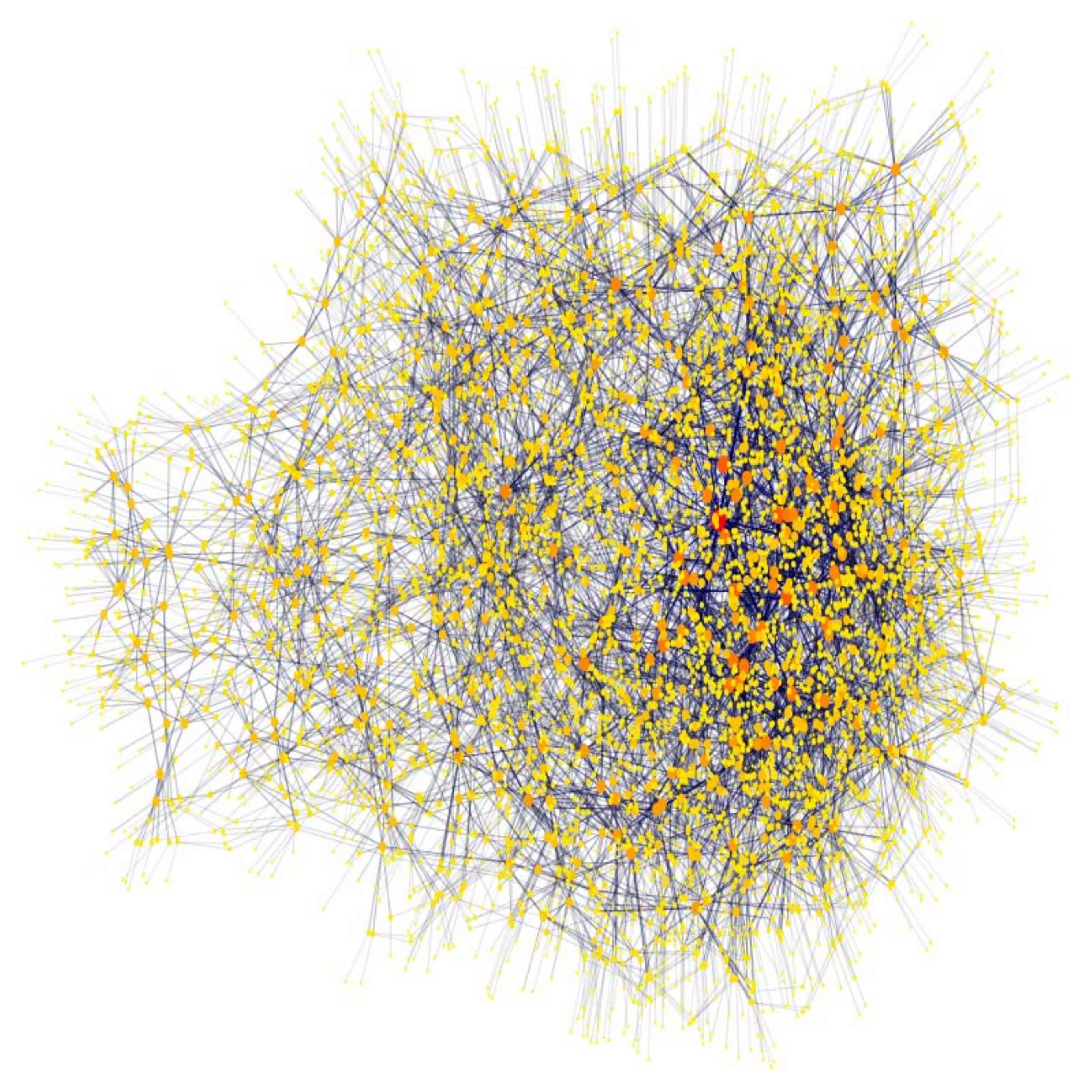}}
\caption{
Graphs on $Z_n[x]/x^6$ (n=4,5) generated by three transformations 
$T_1(f) = f'$, $T_2(f) = f^2$ and $T_3(f) = f+x^4+x^3+x^2+x+1$. 
}
\end{figure}

More generally, one can look at small rings. For example, there are 11
rings of size $4$.  \\

Finally, we can look at transformations of the finite vector space $Z_p^n$.
A natural choice are {\bf elementary cellular automata}. With the Wolfram numbering
for the 256 possible rules, we have for every $n$ a parameter space with $2^{16}$ elements. 
We see experimentally that for some parameters, the connectivity fluctuates with $n$
while for others, the connectivity stays. We could explore this however only for 
$n \leq 25$ because this produces already graphs with number of entries reaching the 
human population on earth. 

\section{Five Mandelbrot Problems}

Here are five connectivity problems which are unsettled at the moment but 
look accessible.  \\

\begin{center} \fbox{\parbox{10cm}{
{\bf 1)} Find necessary and sufficient conditions which assure that the graph 
generated by $T(x)=ax+b$ on $Z_n$ is connected.  \\
}} \end{center}

We have seen that $n$ is a (double) $P$ smooth number where $P$ is a subset of the 
union of primes of $a$ and $a-1$. \\

\begin{center} \fbox{\parbox{10cm}{
{\bf 2)} Are all graphs on $Z_n$ generated by $T(x)=3x+1, S(x)=2x$ connected?  \\
}} \end{center}

We have checked this at the moment only for all graphs up to $n=200'000$ nodes. \\

\begin{center} \fbox{\parbox{10cm}{
{\bf 3)} Find necessary and sufficient conditions which assure that the graph
generated by $T(x)=x^a$ and $S(x)=x^b$ on $Z_n^*$ is connected.  \\
}} \end{center}

We have seen that if $a=2$ and $n$ is a Fermat prime, then the graph is 
connected because it is already connected for one transformation. 
An example is $n=257, a=5, b=56$.  \\

\begin{center} \fbox{\parbox{10cm}{
{\bf 4)} Which $2 \times 2$ matrices $A$ in $R=M(2,Z_n)$ have the property that
$T(x)=x^2+A$ produces a connected graph?   \\
}} \end{center}

An example of a connected graph is obtained for $n=5$ with 
$A = \left[ \begin{array}{cc} 1 & 2 \\ 2 &  4 \\ \end{array} \right]$. 

\begin{center} \fbox{\parbox{10cm}{
{\bf 5)} For which $n,a,b$ do the elementary cellular automata $T_a,T_b$ acting on the field $Z_2^n$
produce connected graphs? 
}} \end{center}

\begin{figure}
\scalebox{0.31}{\includegraphics{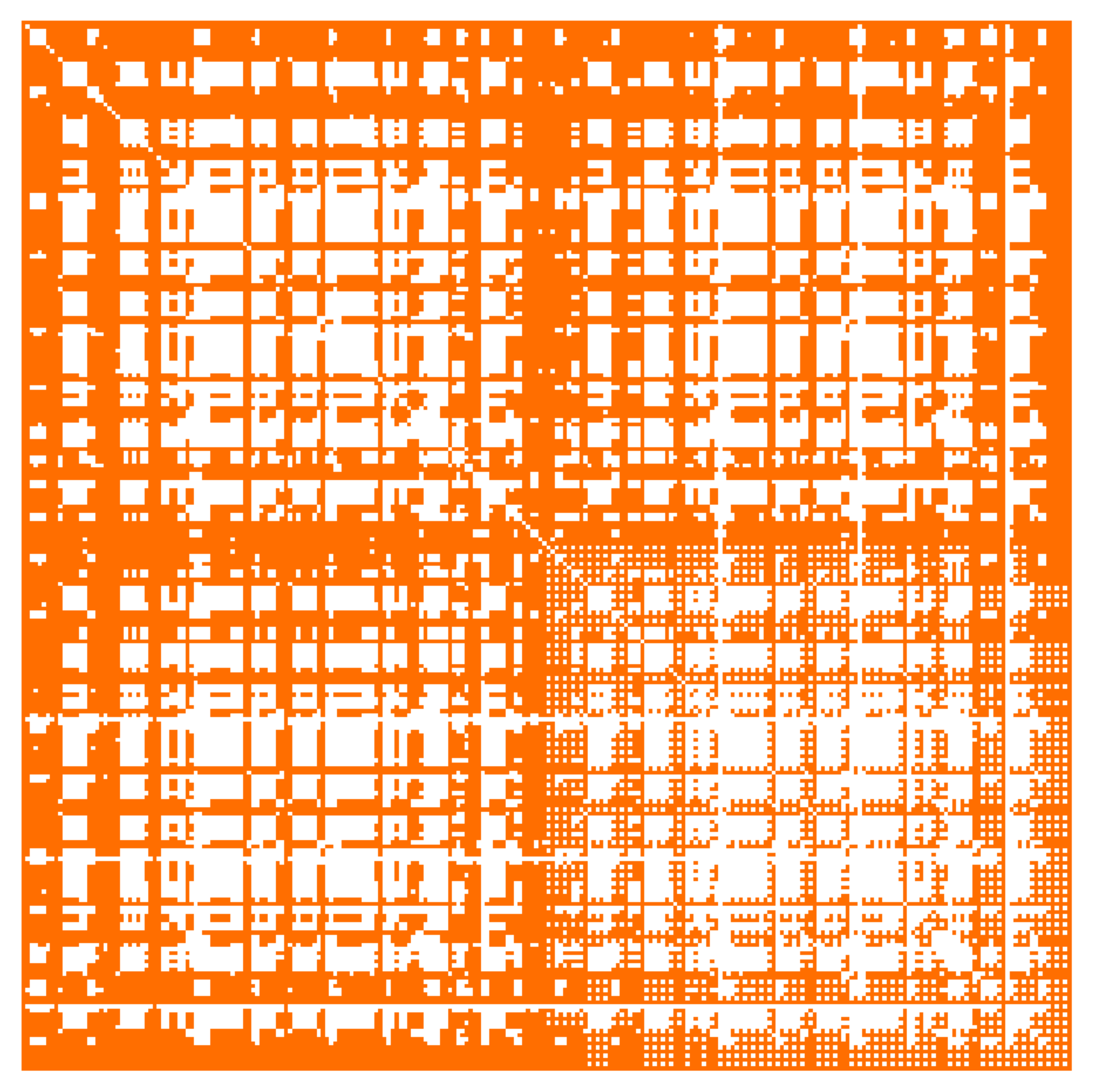}}
\caption{
The CA Mandelbrot set $M$ for pairs of elementary CA on $Z_2^9$. 
Colored are the parameters $(a,b)$ with 
$0 \leq a,b \leq 255$ for which one has a connected graph. 
}
\end{figure}

For every $n$ there are $2^{16}$ possible graphs. One can ask whether we have for fixed $a,b$
convergence of the number of connectivity components when $n \to \infty$. This seems to be the case for $a=101,b=110$. 
For other pairs like $a=3,b=4$ we see fluctuations in $n$ which seem to depend on the prime
factorization of $n$  at least for the $n$ in which we can construct the graphs like $n \leq 25$. 

\section{Pictures}

The following pictures have been produced with code we will submit to
a demonstration project allowing readers to draw the graphs on their own. 
In the mean time the code is also available on the project website. 

\scalebox{0.23}{\includegraphics{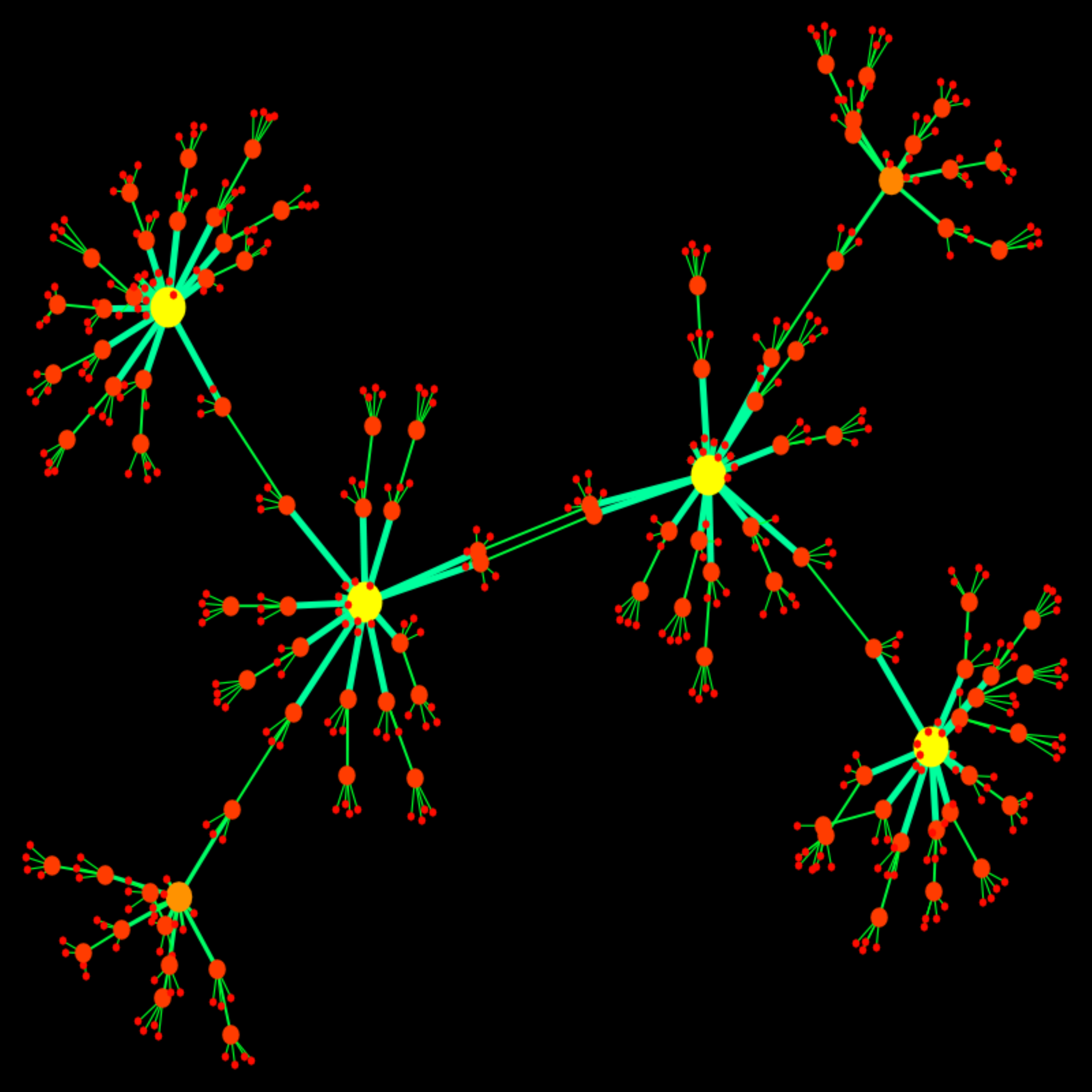}}
\scalebox{0.23}{\includegraphics{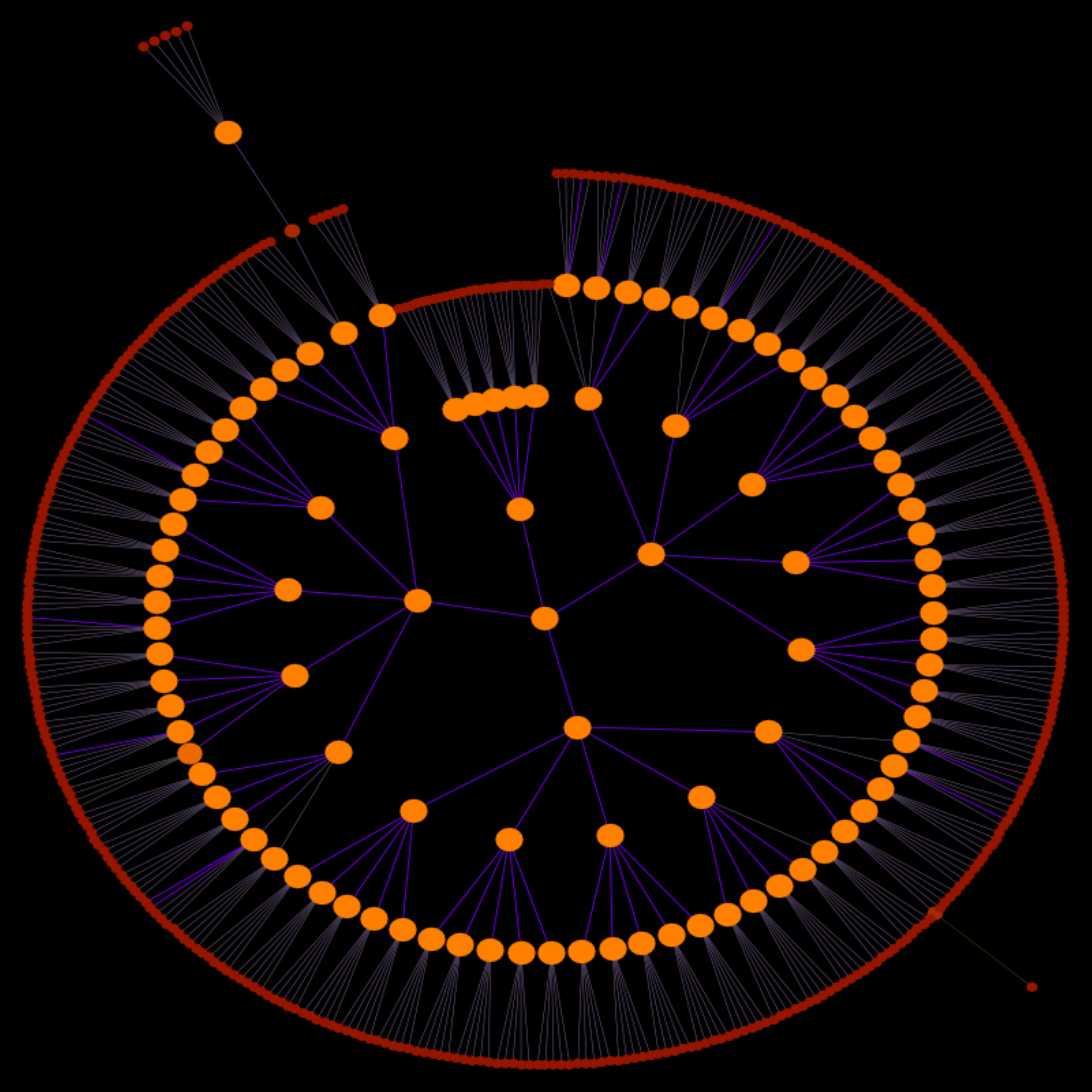}}
\scalebox{0.23}{\includegraphics{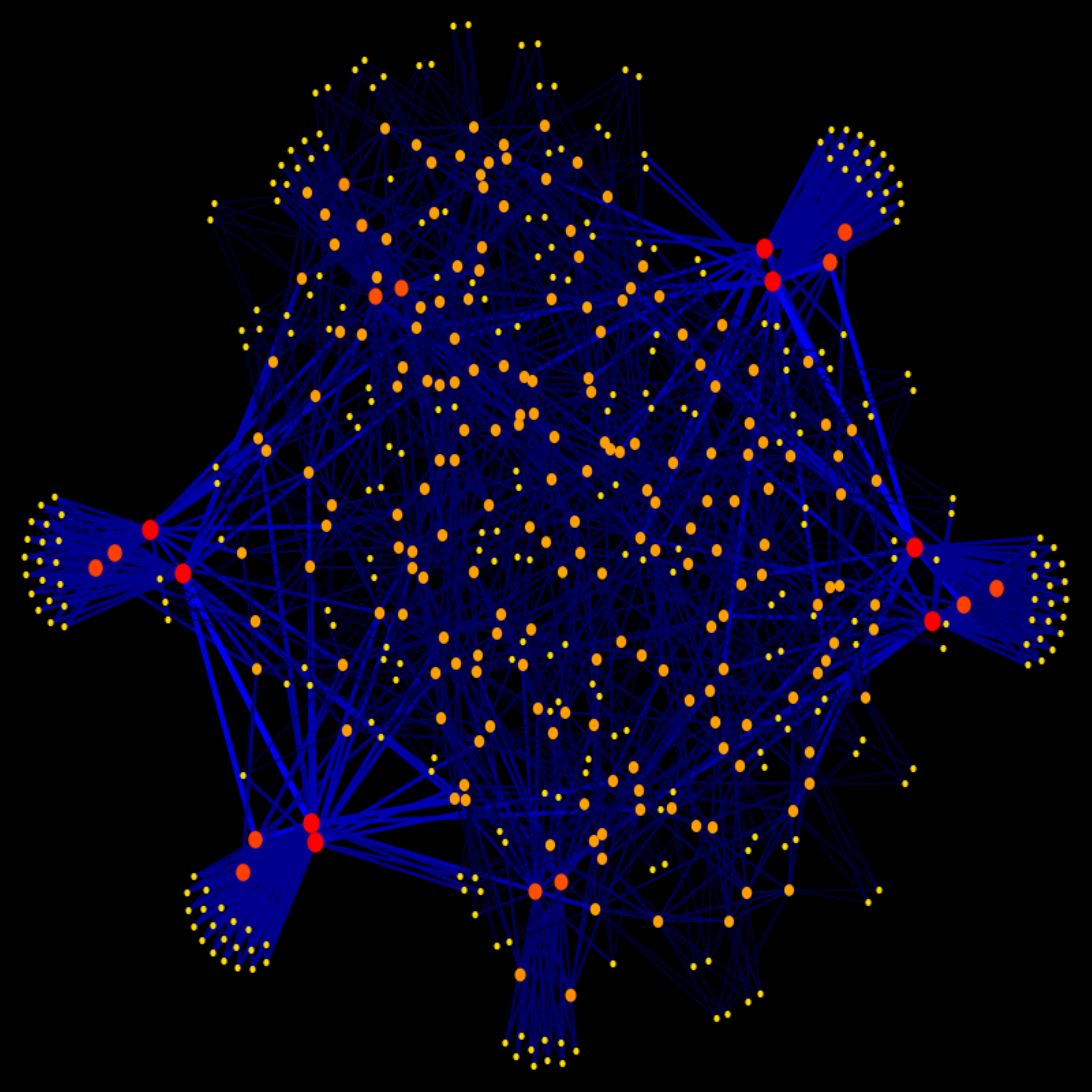}}
\scalebox{0.23}{\includegraphics{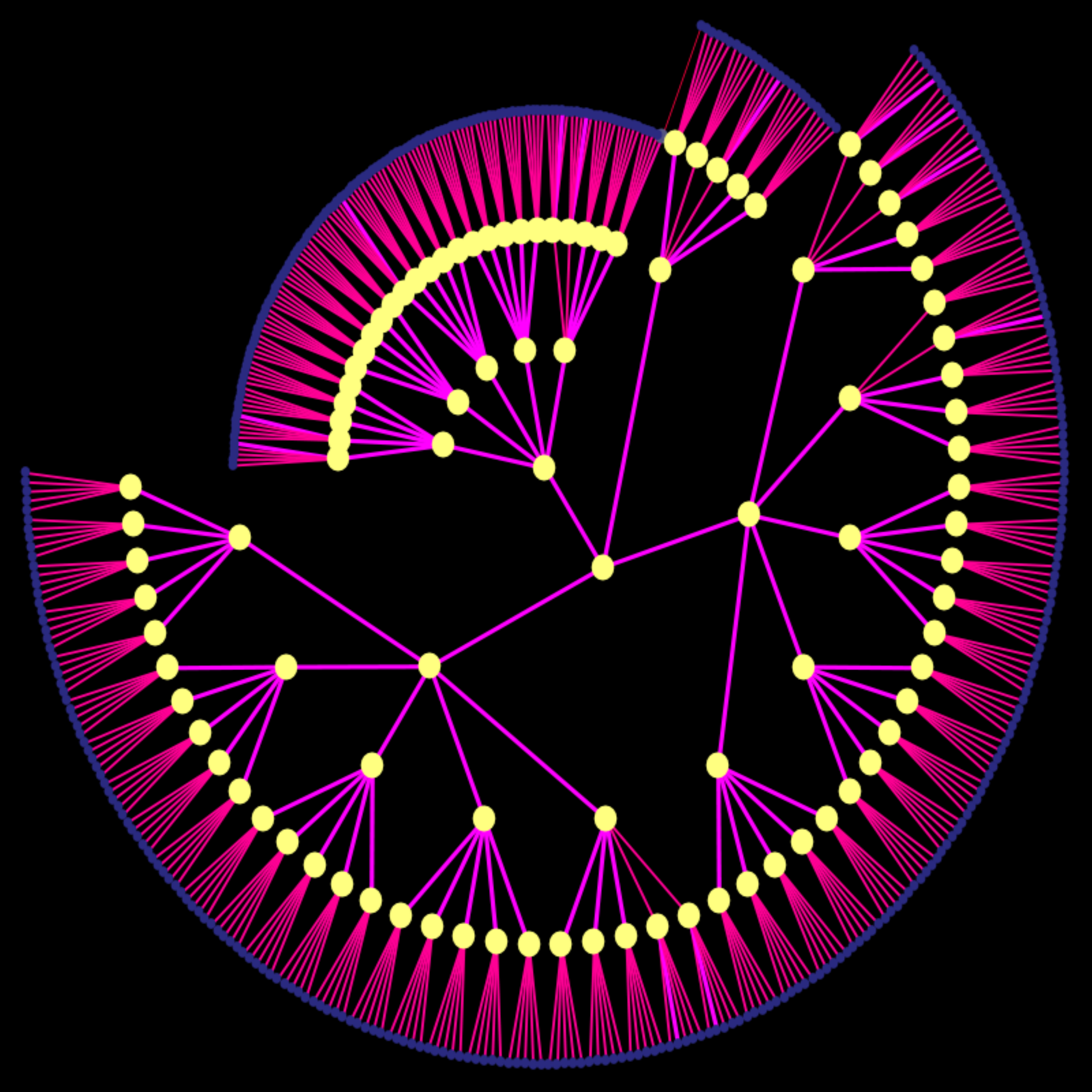}}
\scalebox{0.23}{\includegraphics{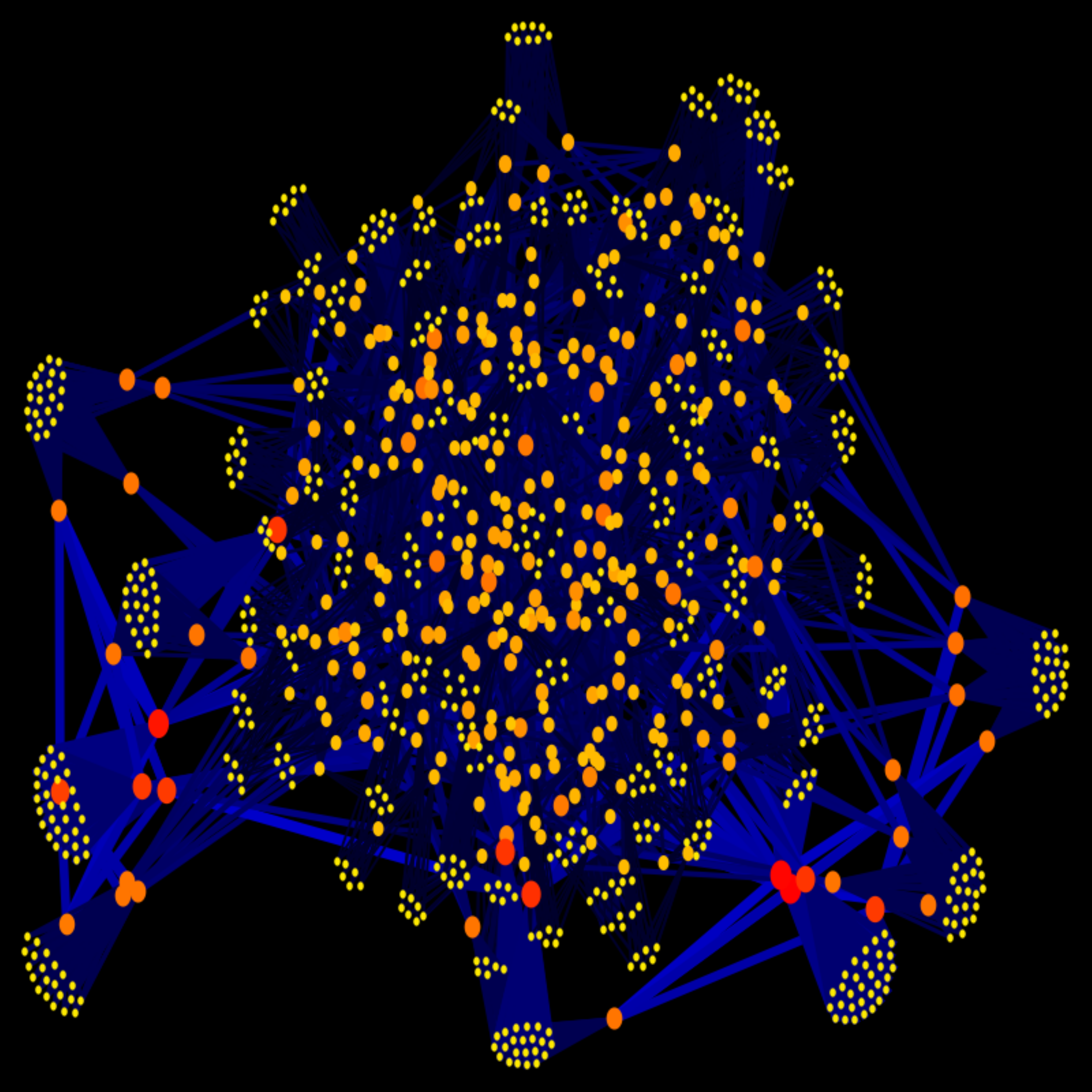}}
\scalebox{0.23}{\includegraphics{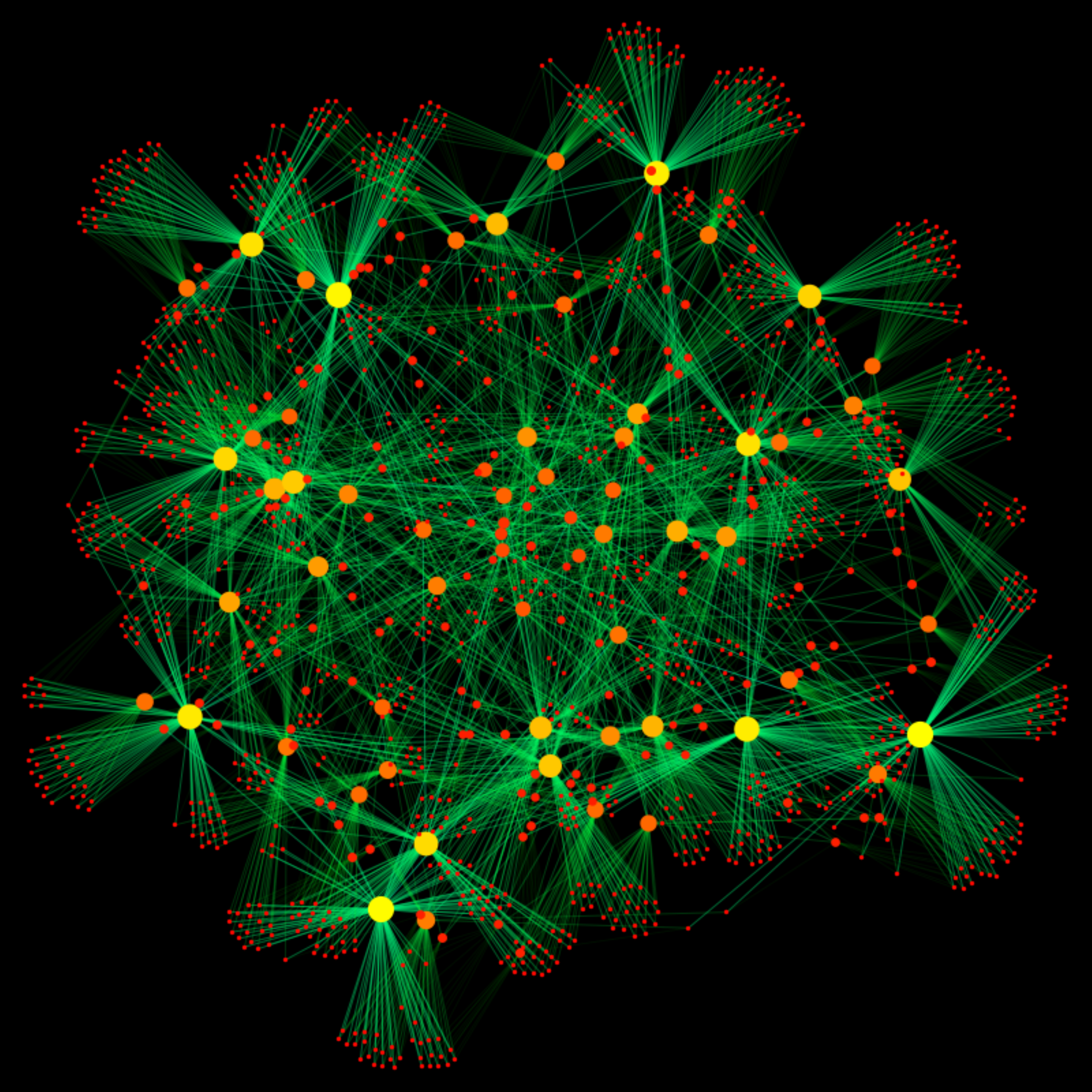}}
\scalebox{0.23}{\includegraphics{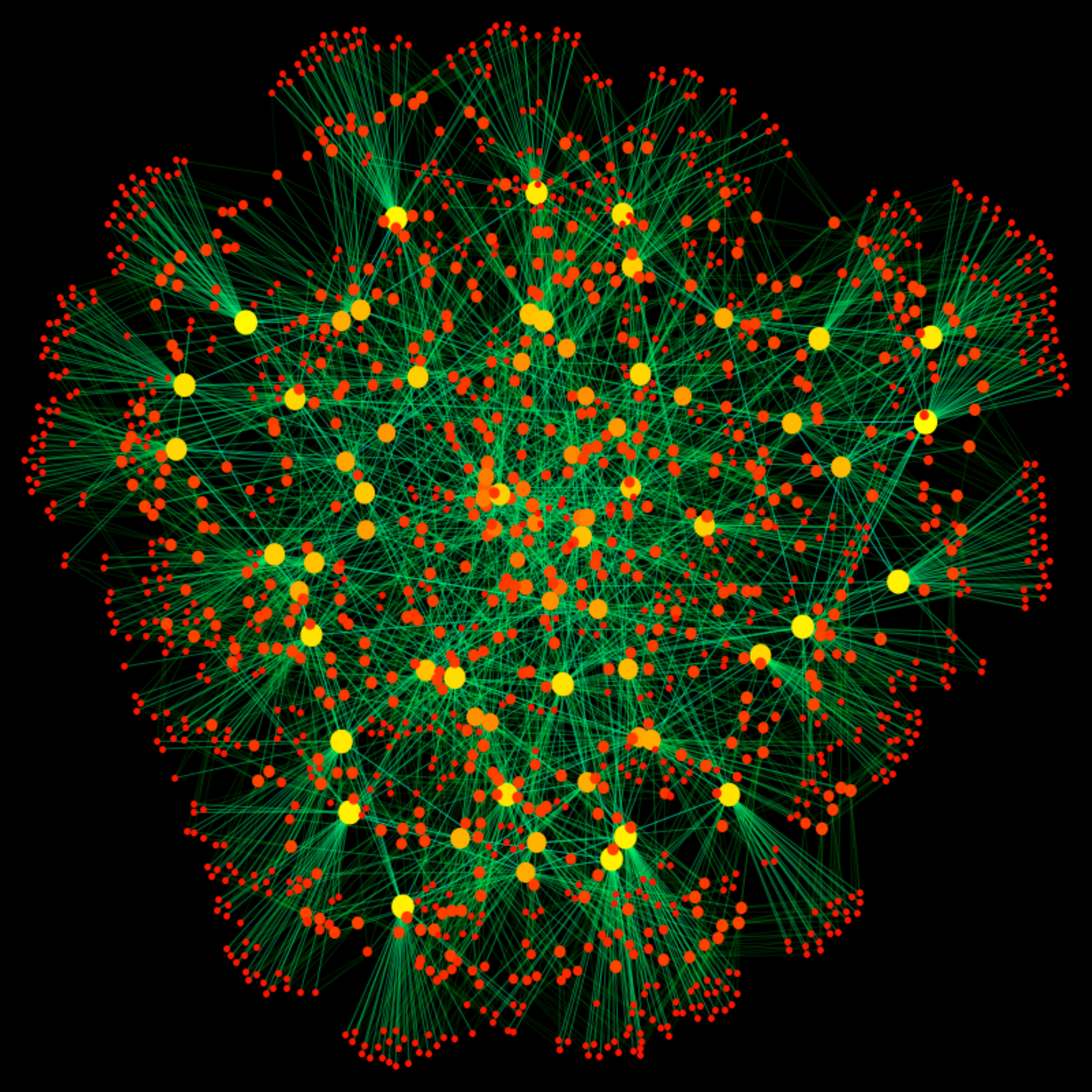}}
\scalebox{0.23}{\includegraphics{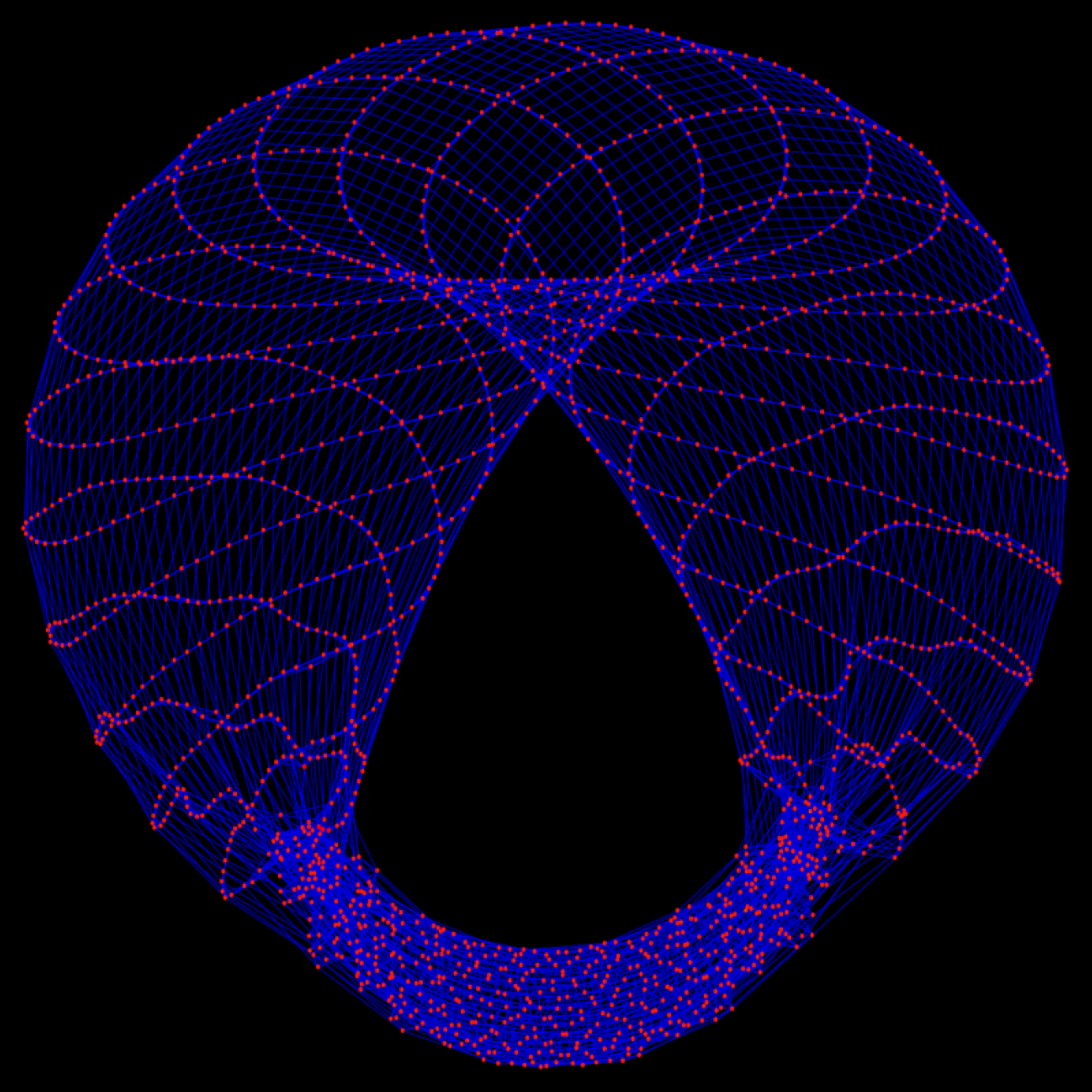}}
\scalebox{0.23}{\includegraphics{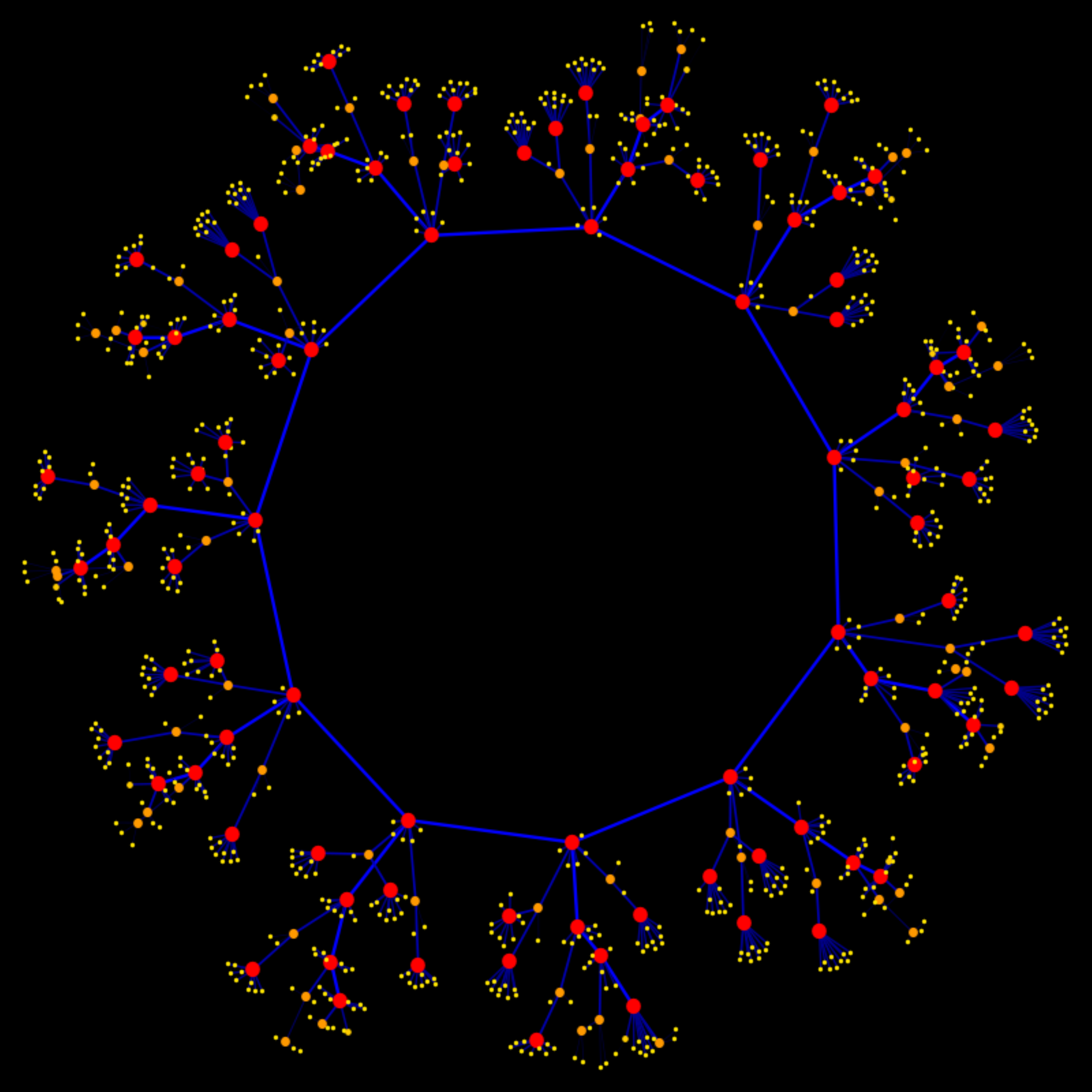}}
\scalebox{0.23}{\includegraphics{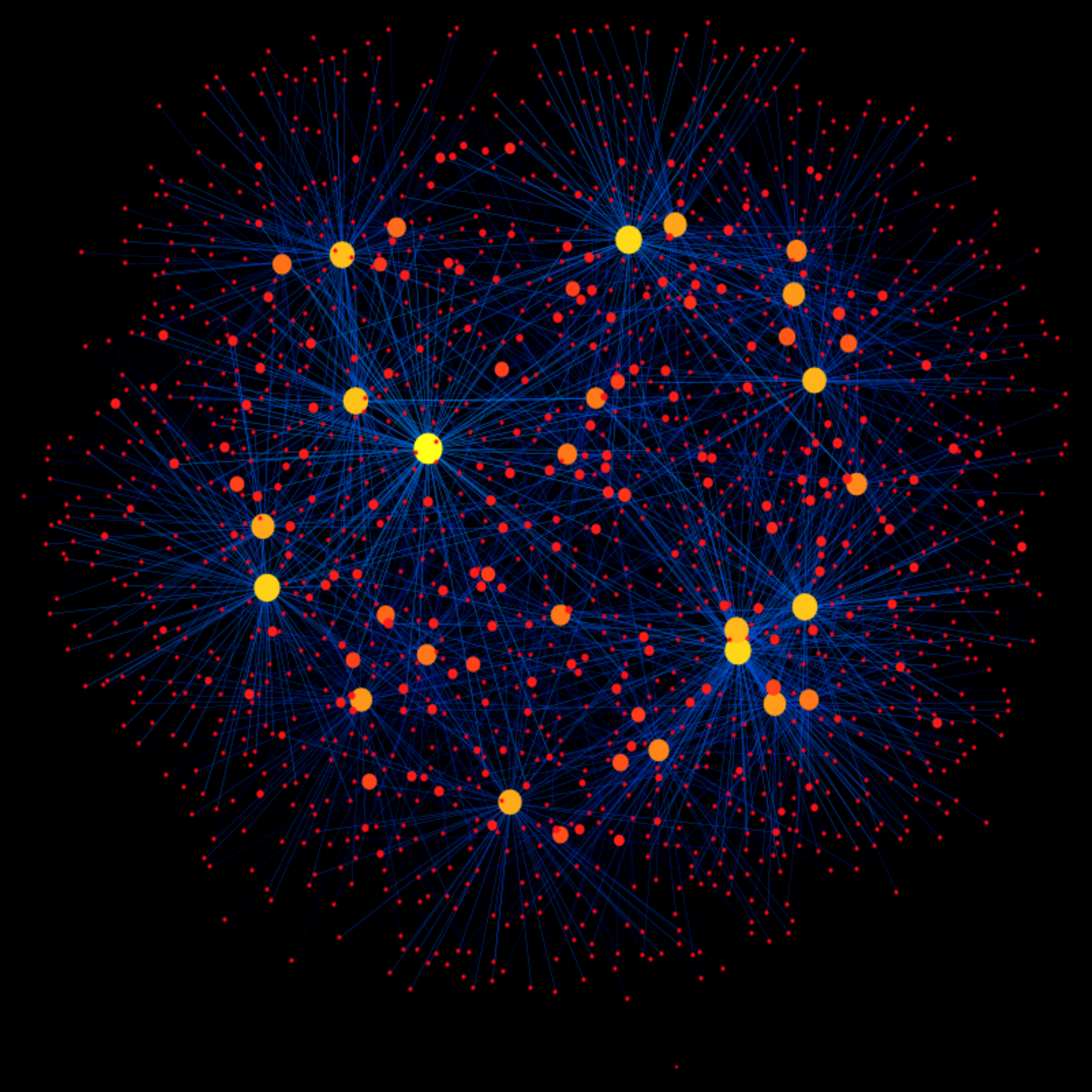}}
\scalebox{0.23}{\includegraphics{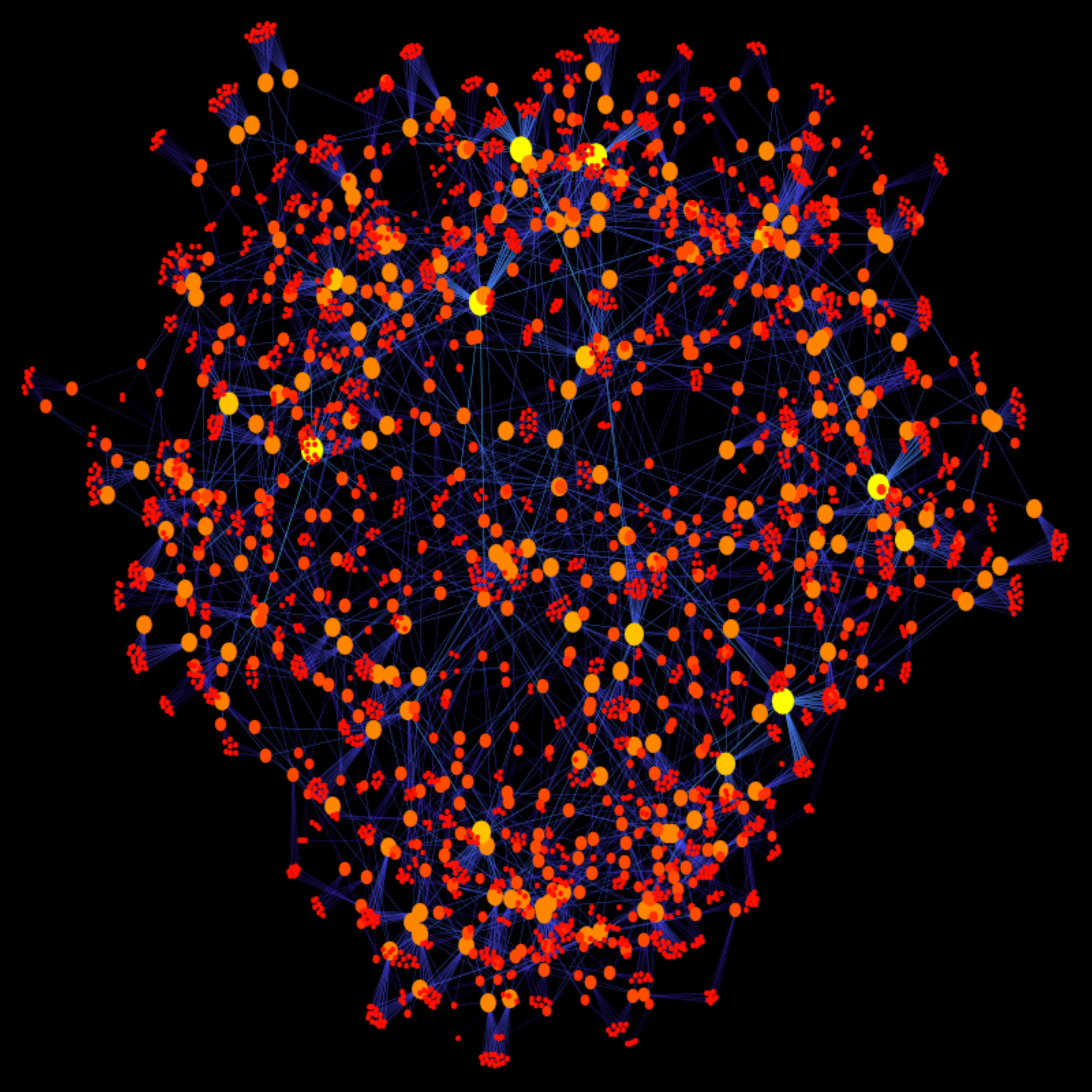}}
\scalebox{0.23}{\includegraphics{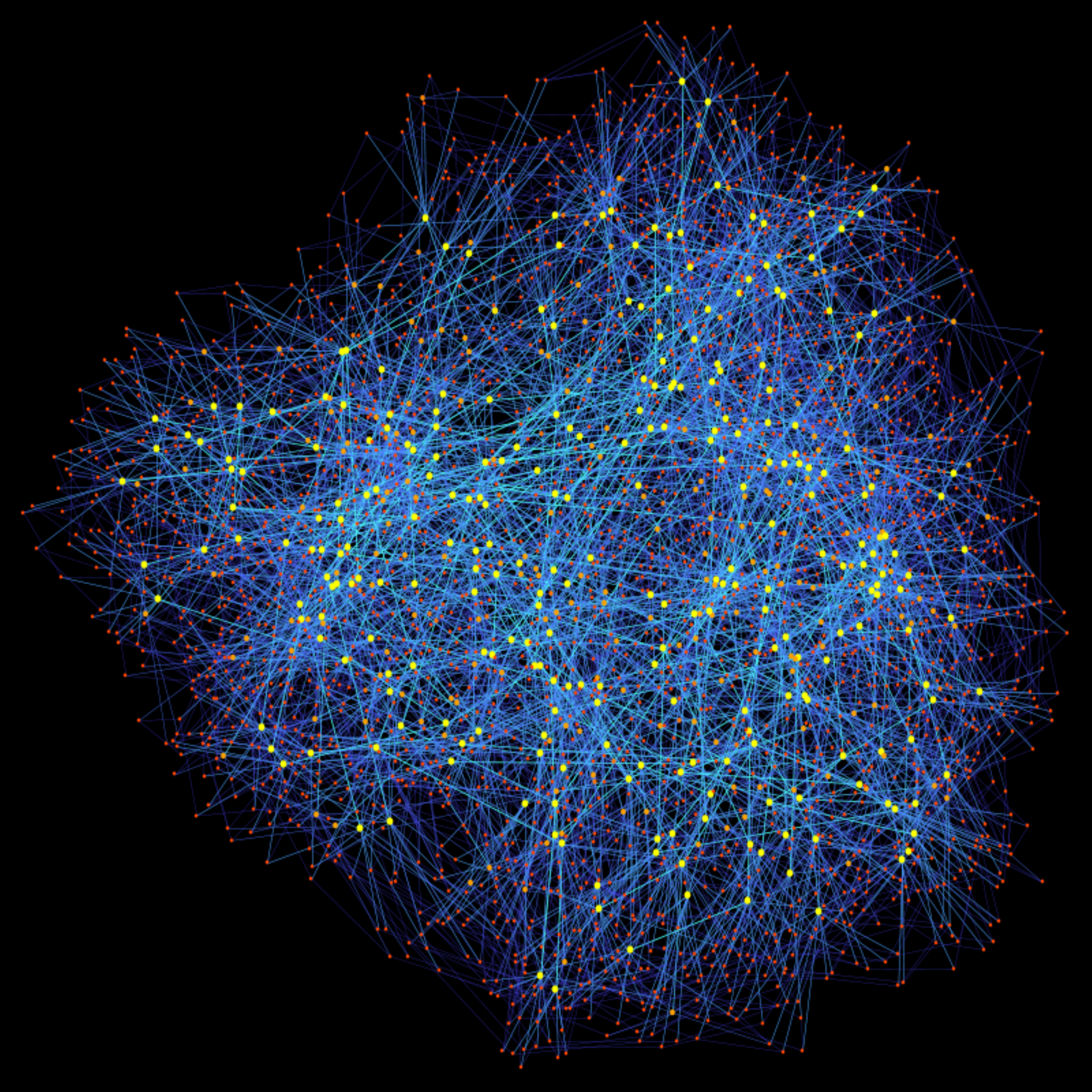}}
\scalebox{0.23}{\includegraphics{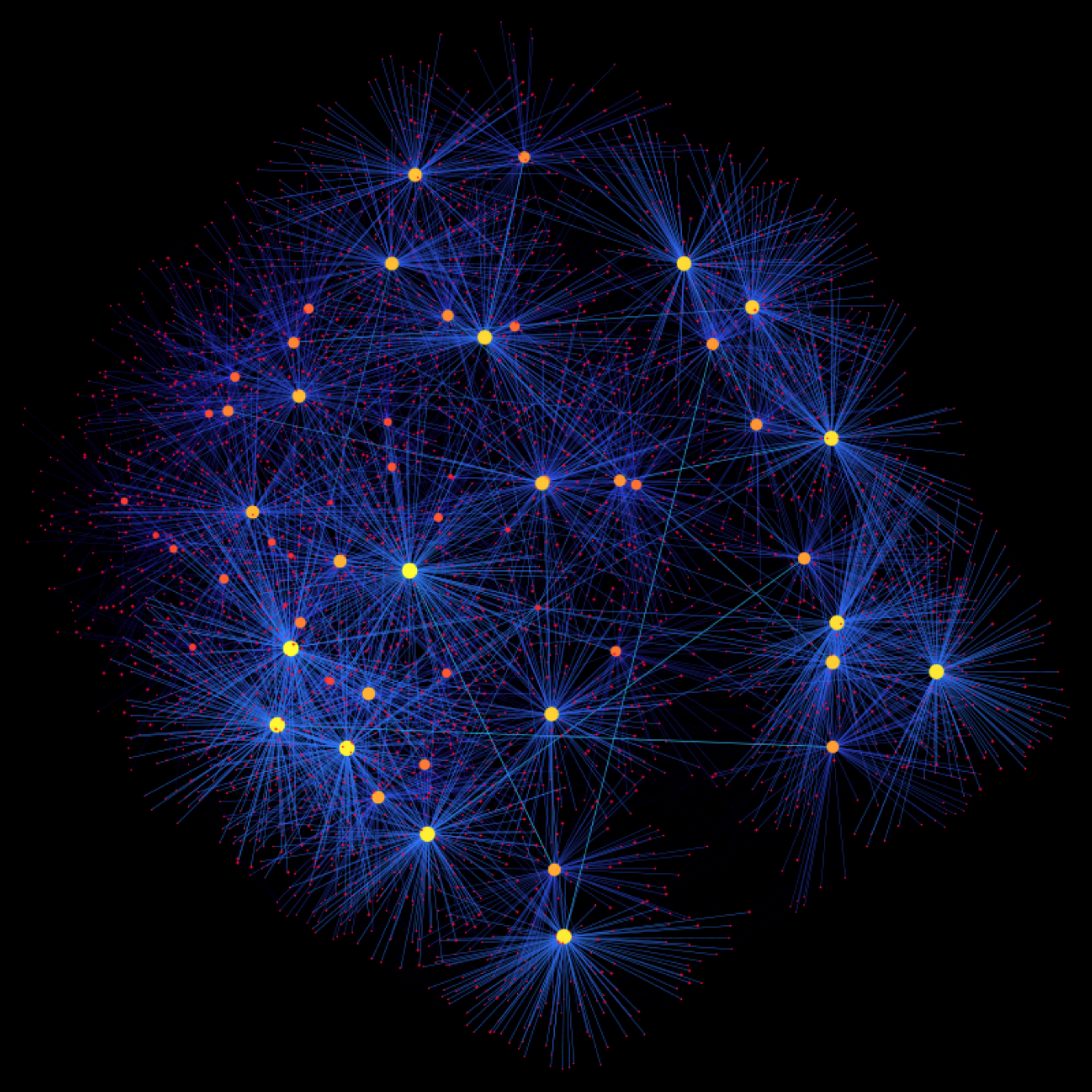}}
\scalebox{0.23}{\includegraphics{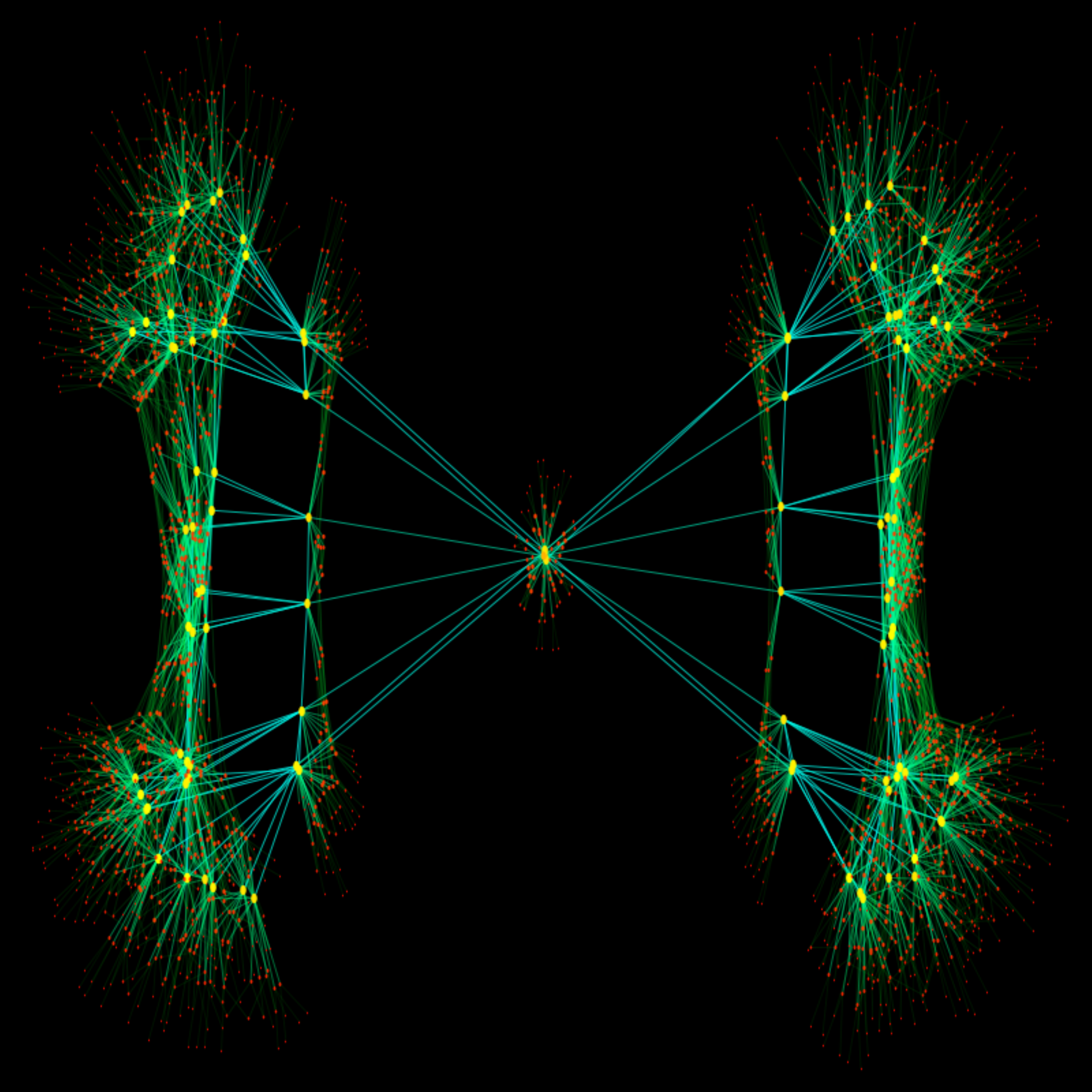}}
\scalebox{0.23}{\includegraphics{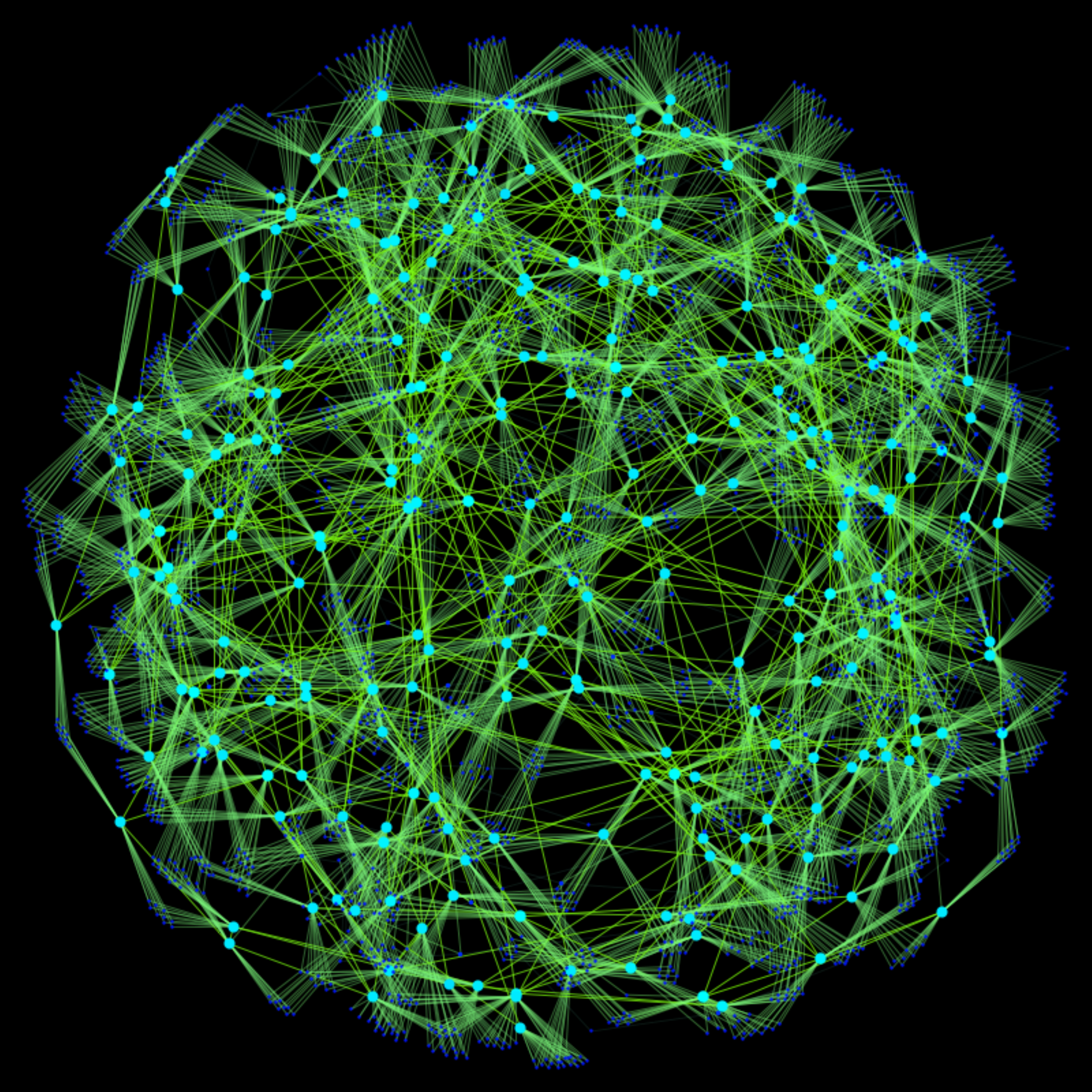}}
\scalebox{0.23}{\includegraphics{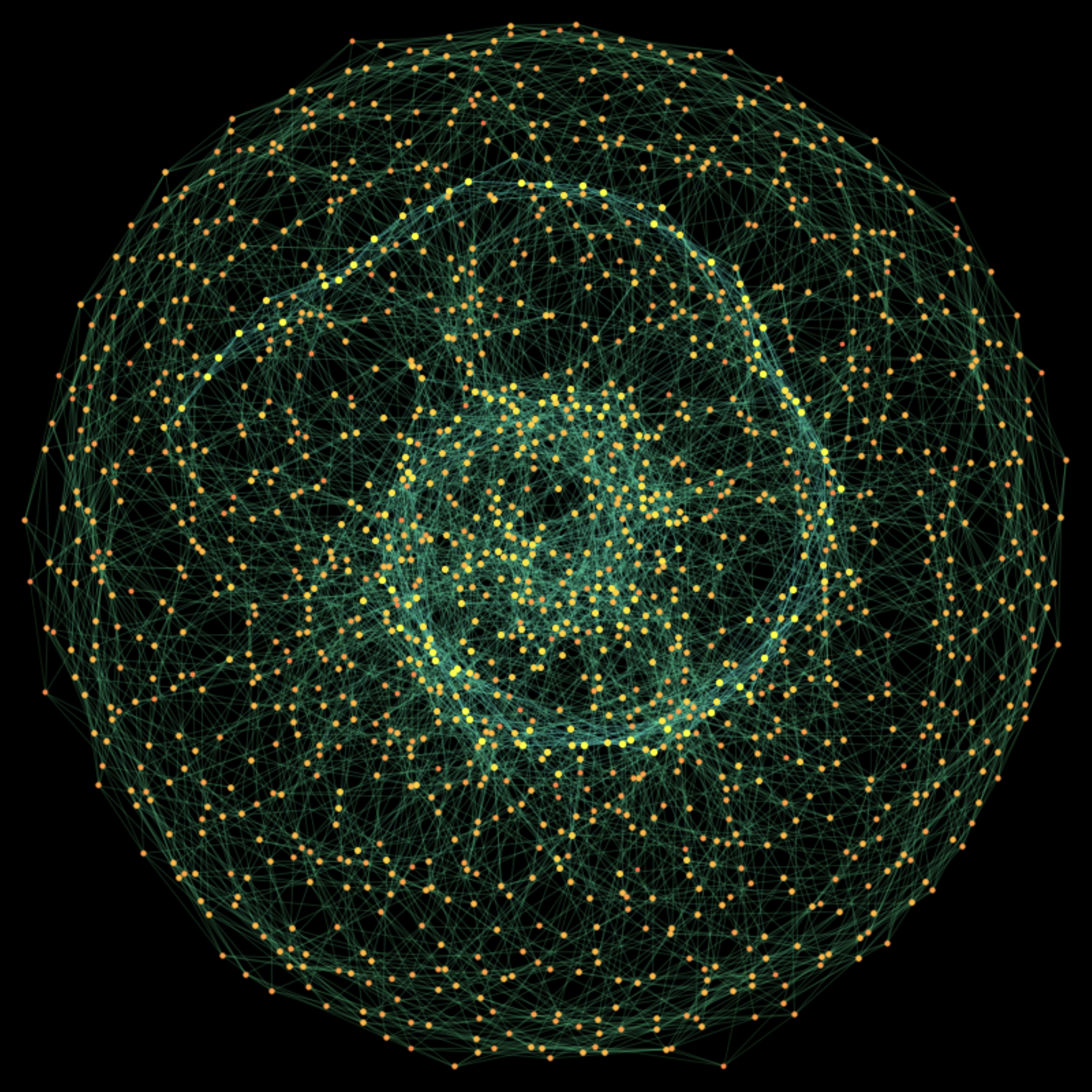}}
\scalebox{0.23}{\includegraphics{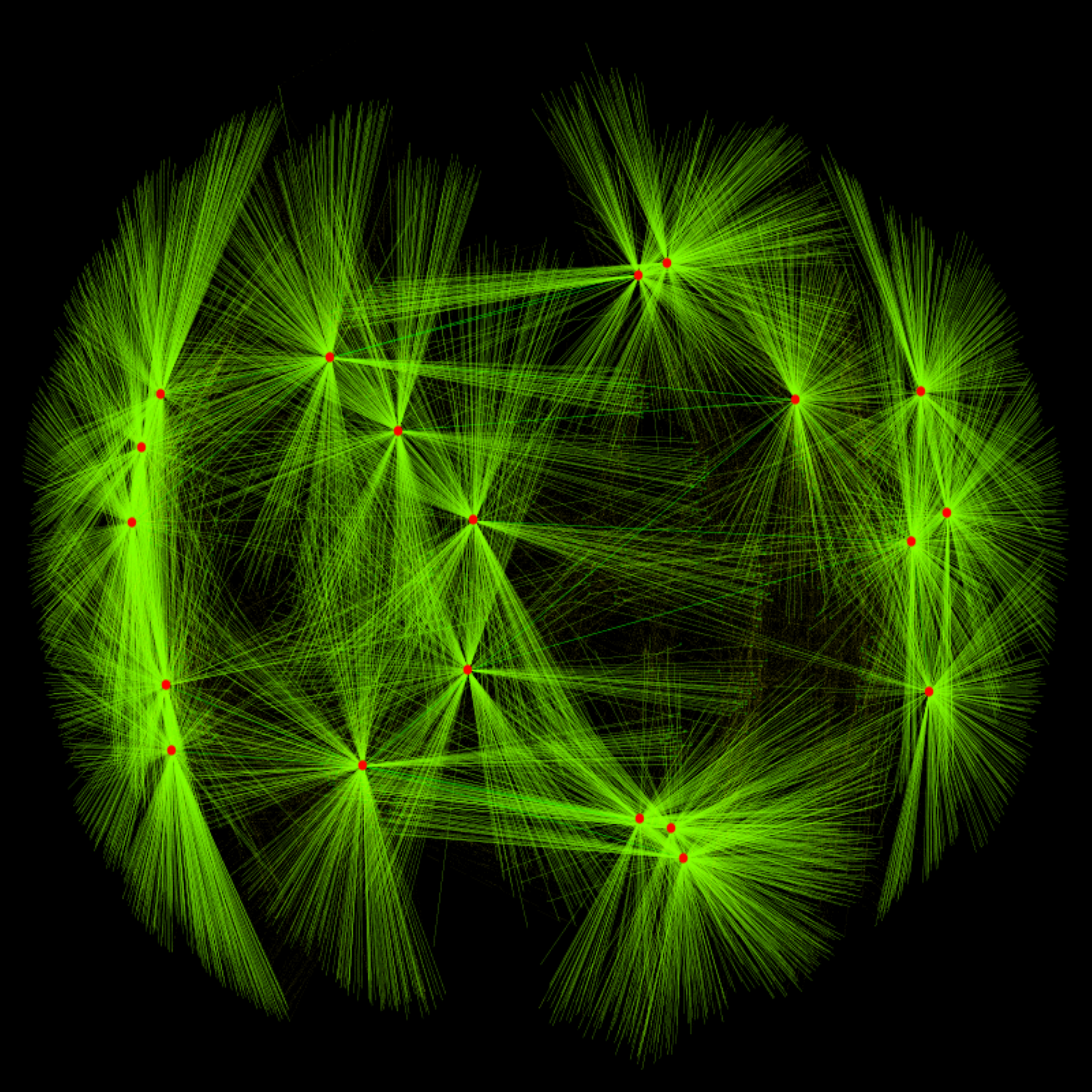}}
\scalebox{0.23}{\includegraphics{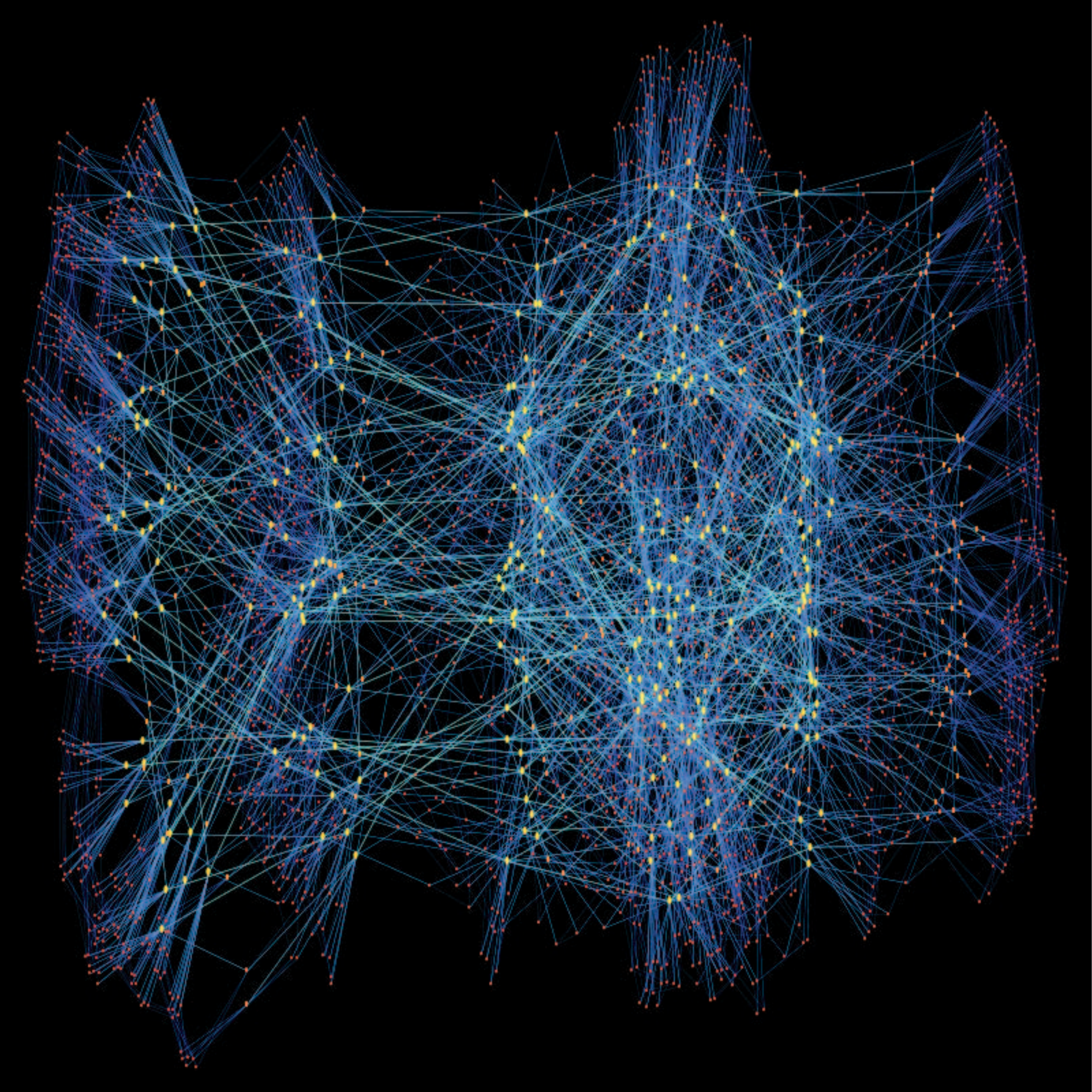}}
\scalebox{0.23}{\includegraphics{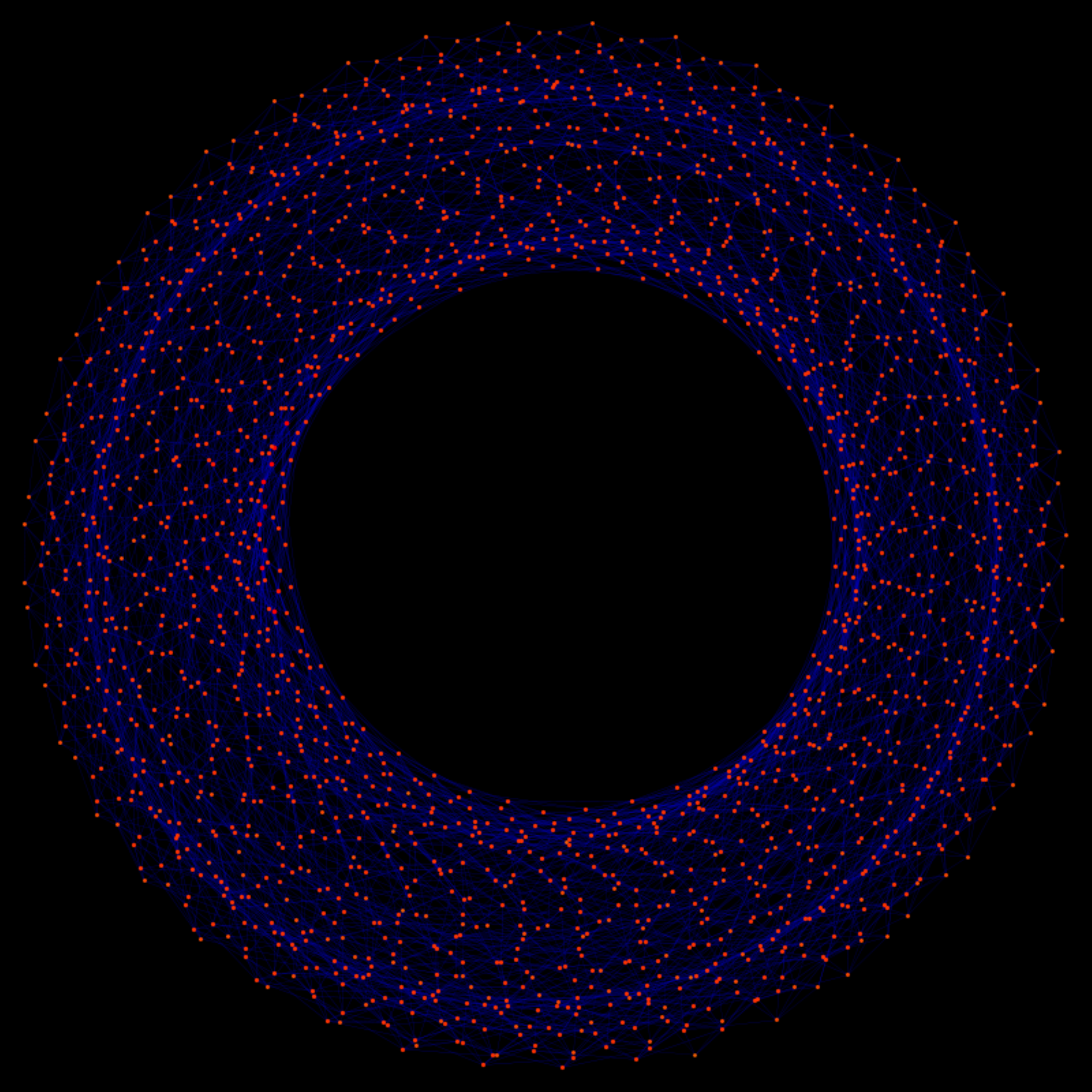}}
\scalebox{0.23}{\includegraphics{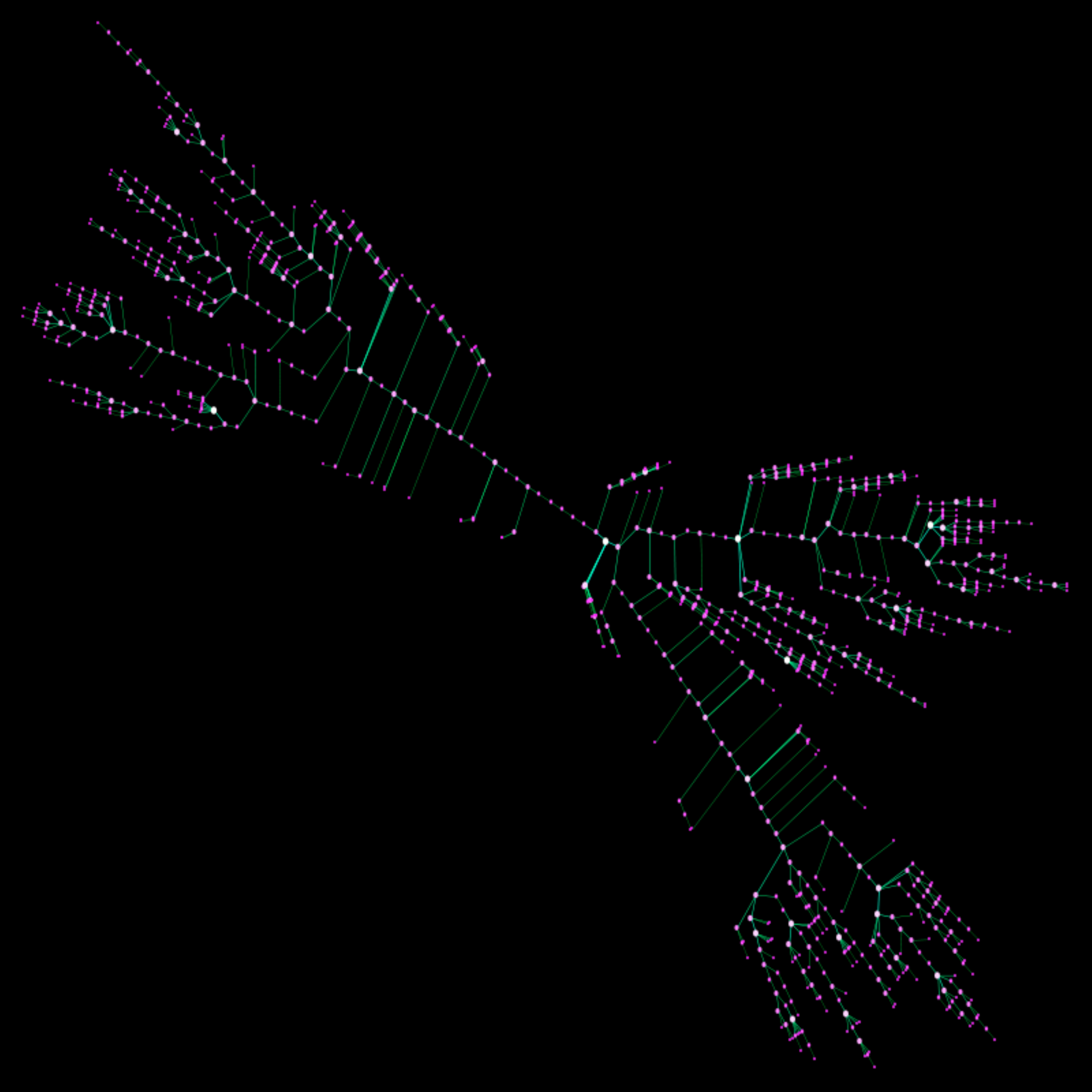}}
\scalebox{0.23}{\includegraphics{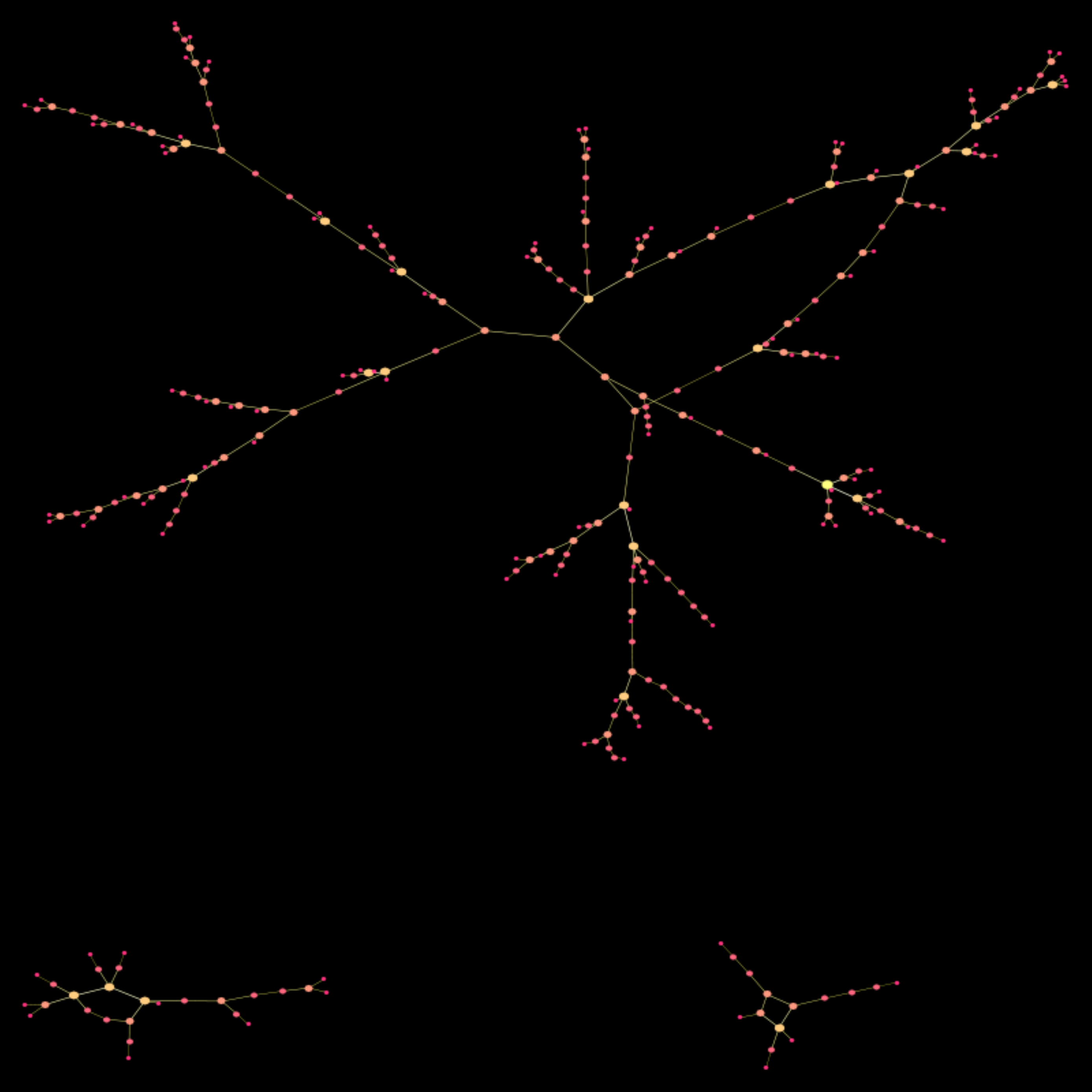}}
\scalebox{0.23}{\includegraphics{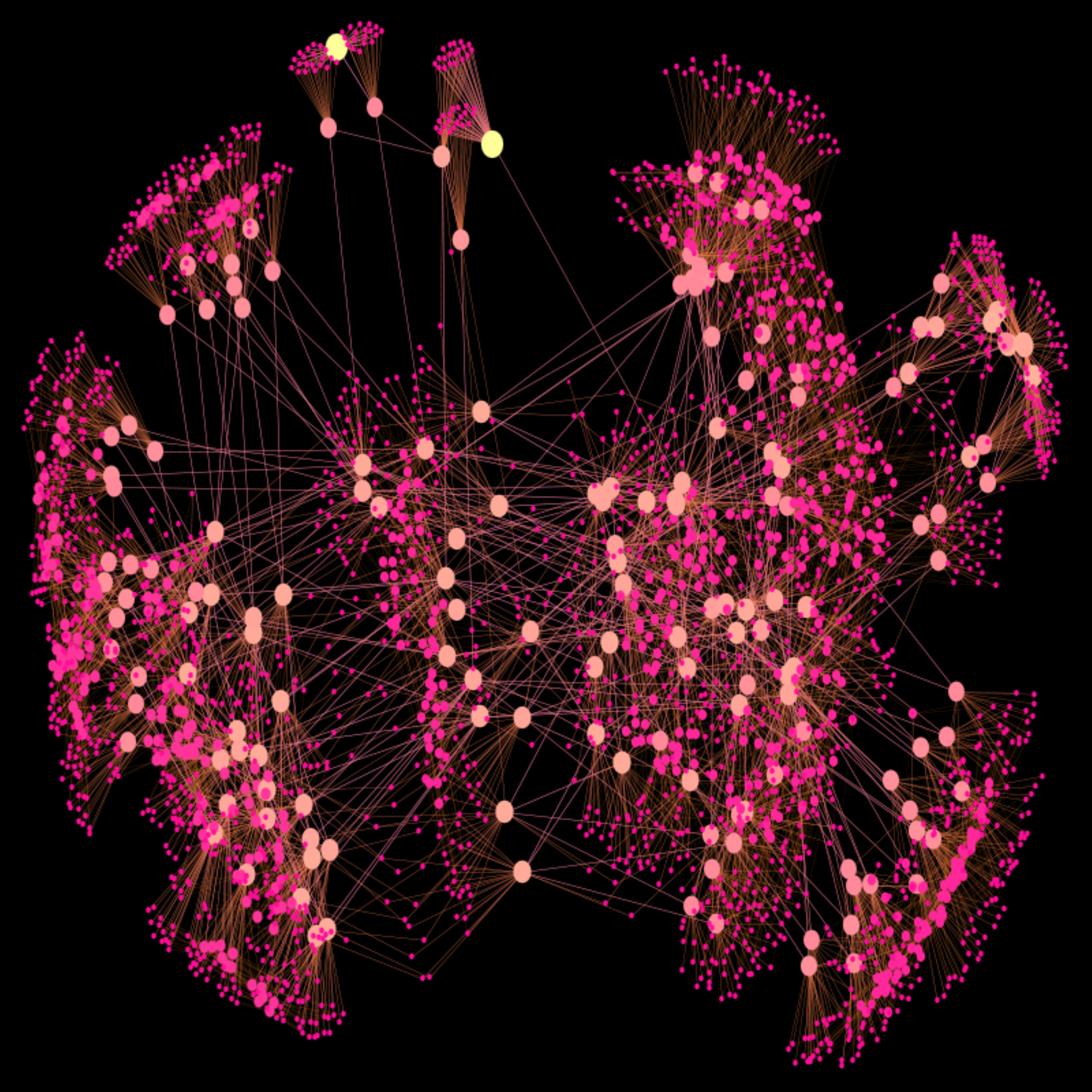}}
\scalebox{0.23}{\includegraphics{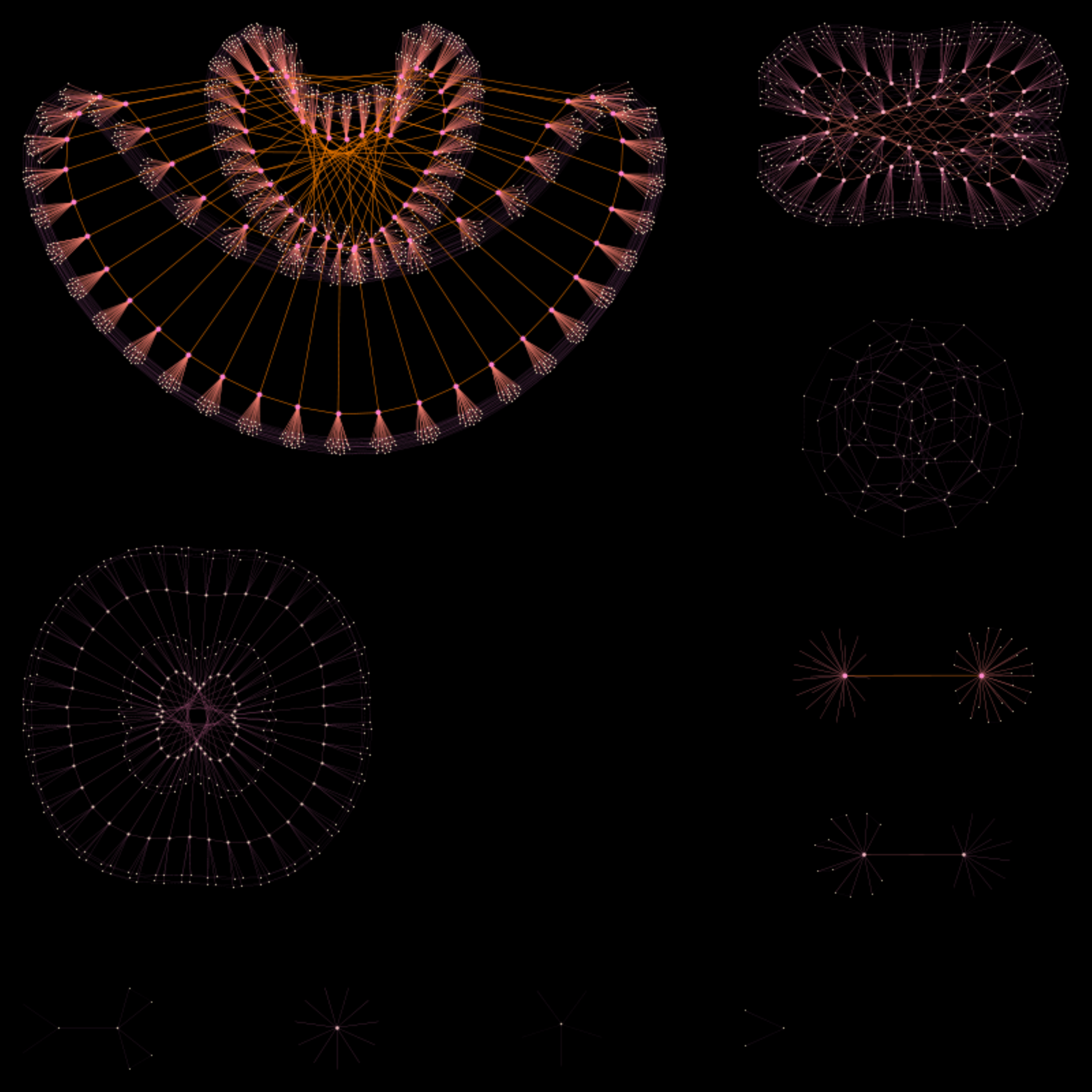}}
\scalebox{0.23}{\includegraphics{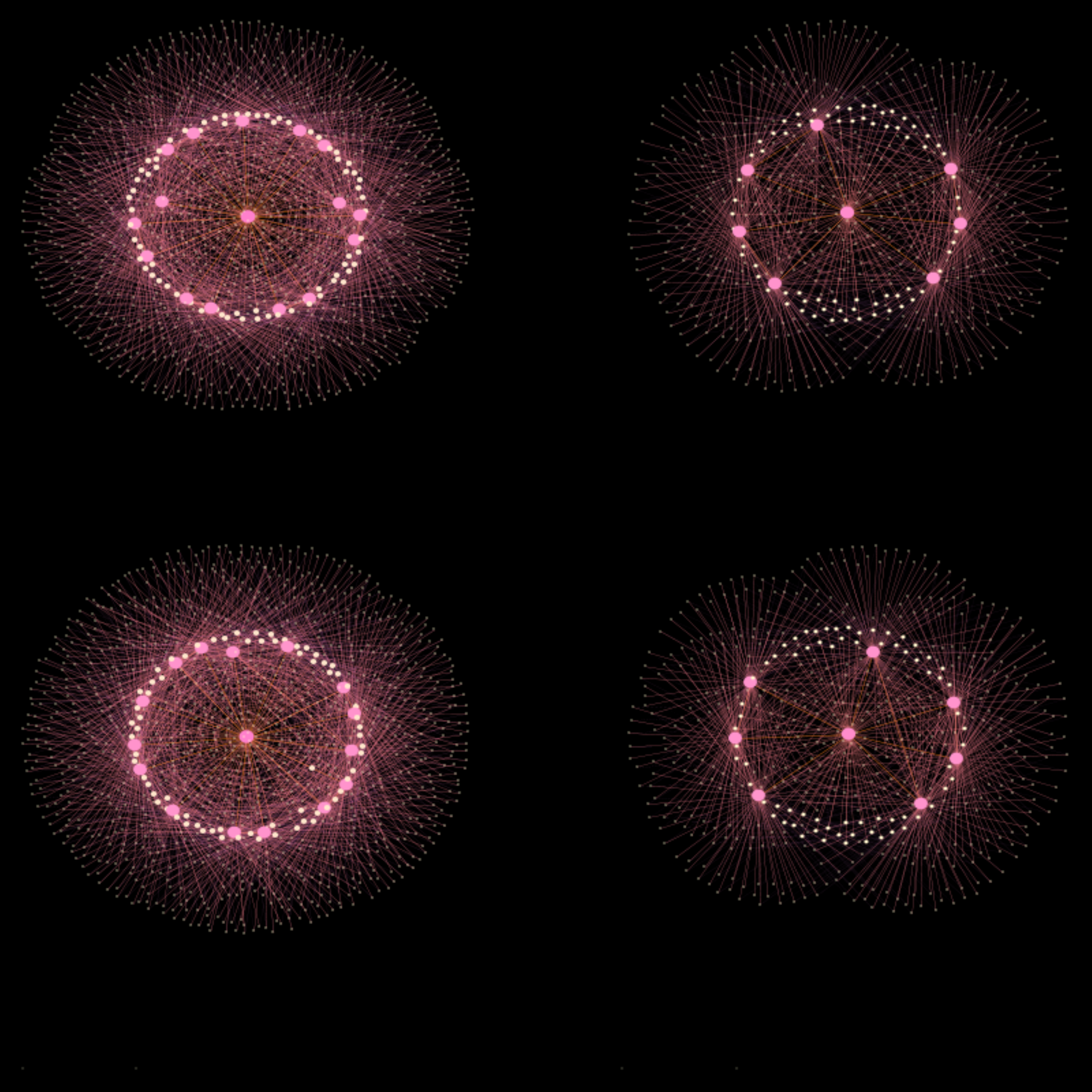}}
\scalebox{0.23}{\includegraphics{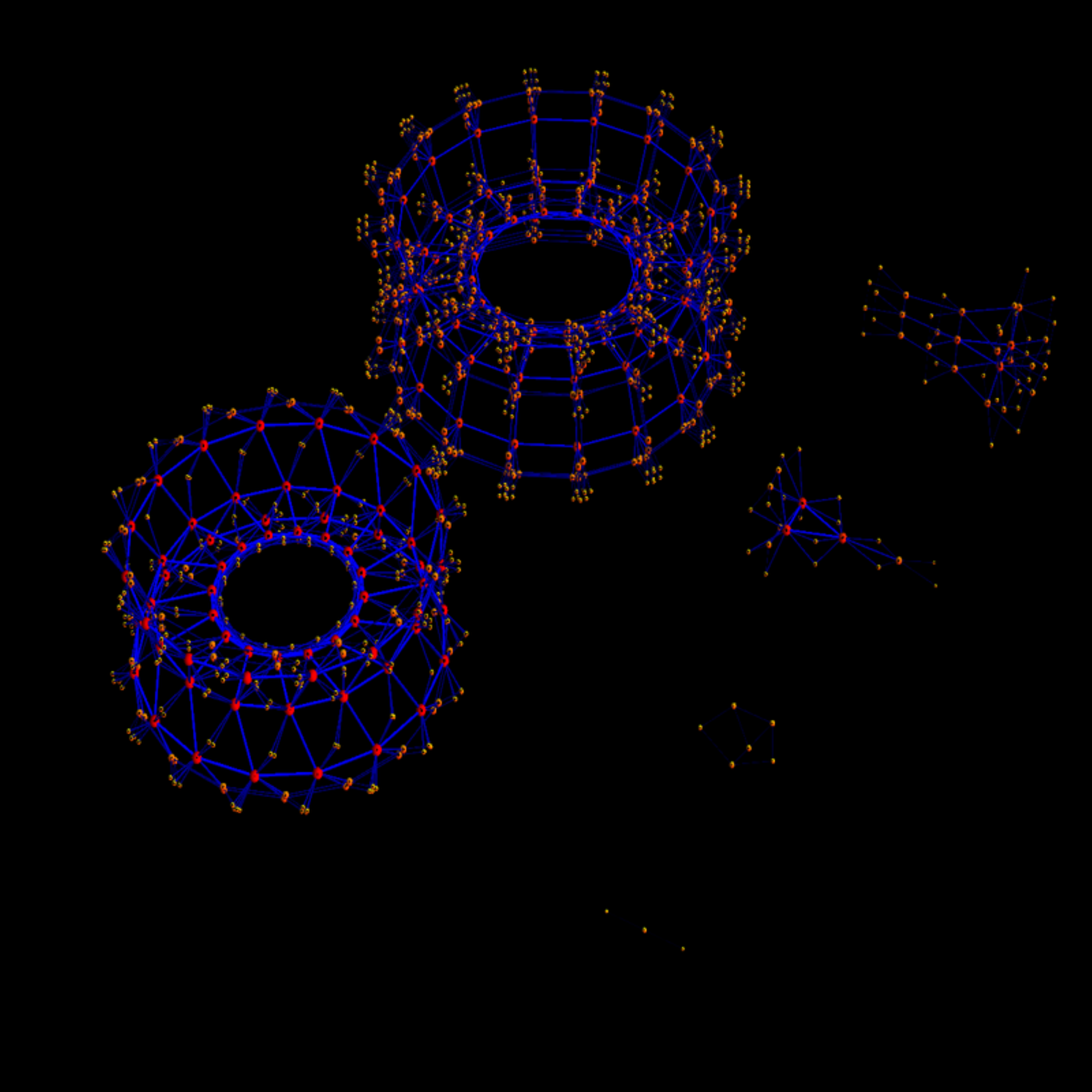}}
\scalebox{0.23}{\includegraphics{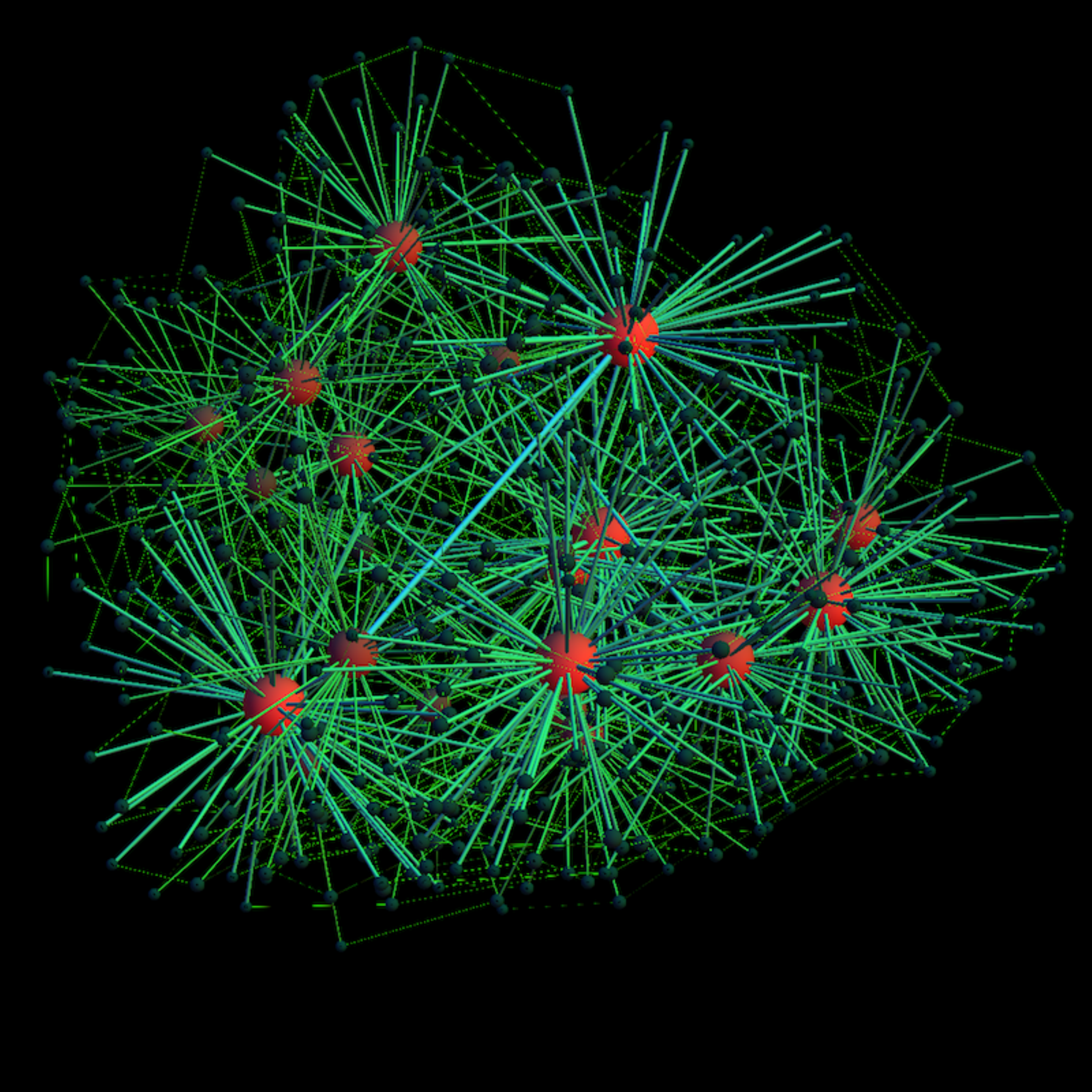}}
\scalebox{0.23}{\includegraphics{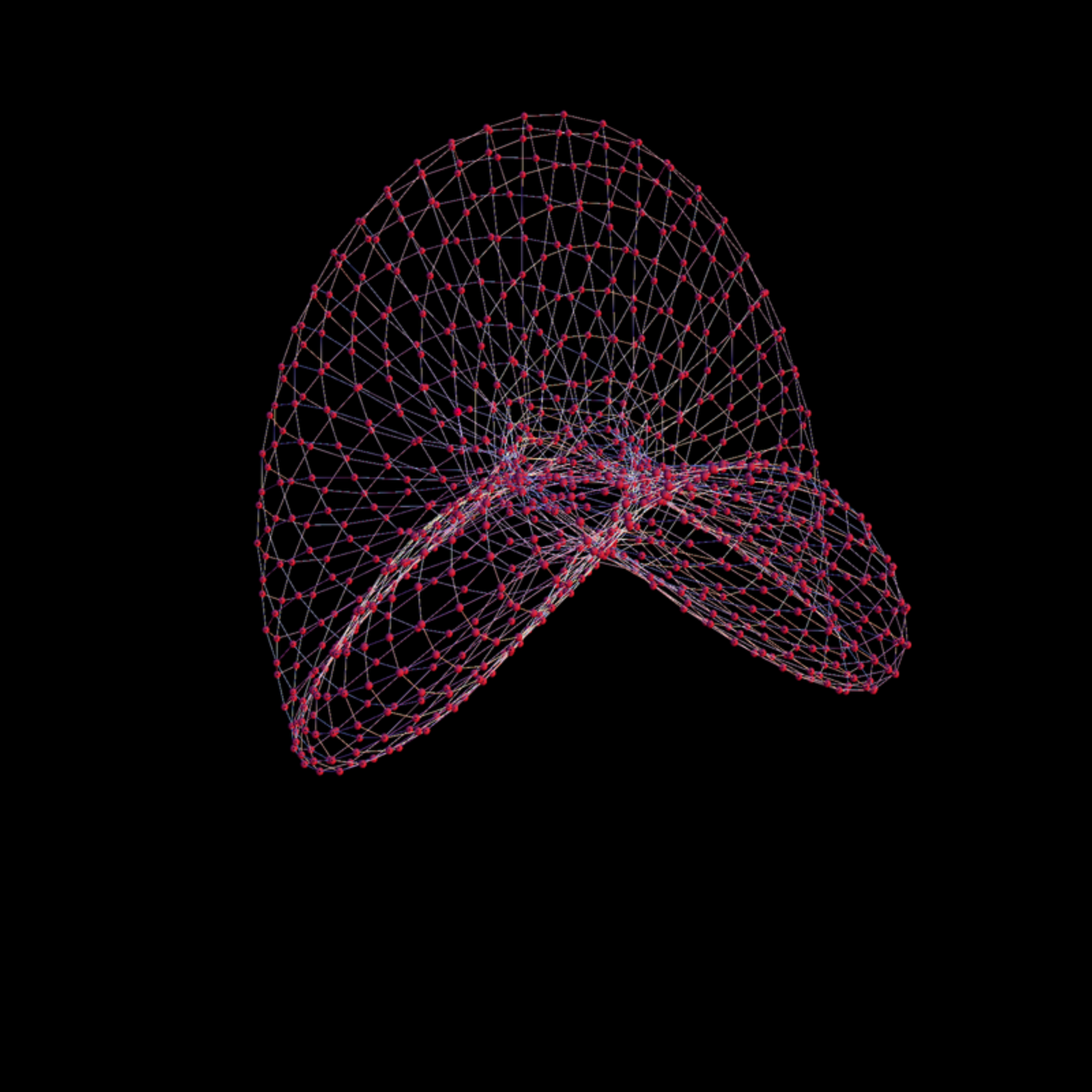}}
\scalebox{0.23}{\includegraphics{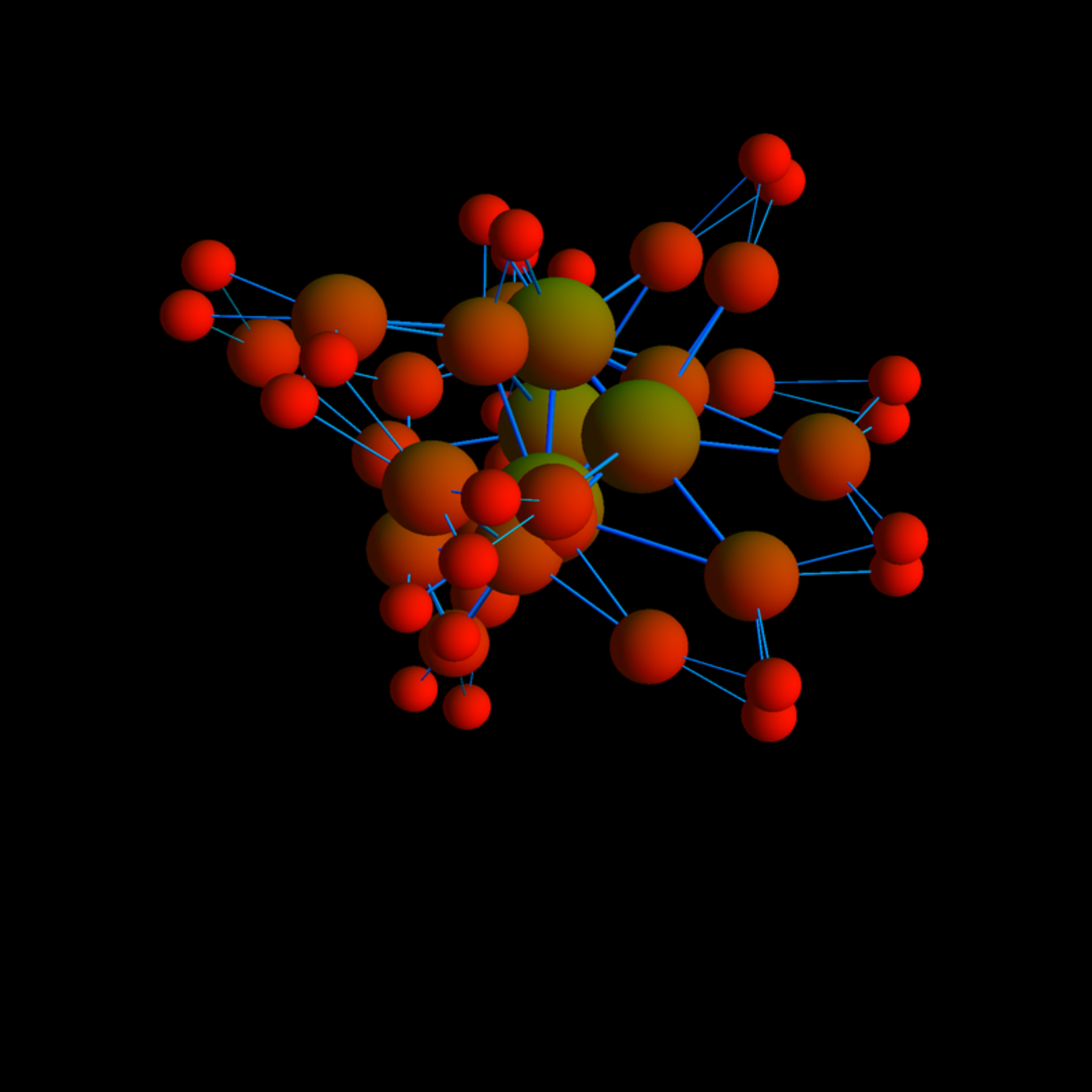}}
\scalebox{0.23}{\includegraphics{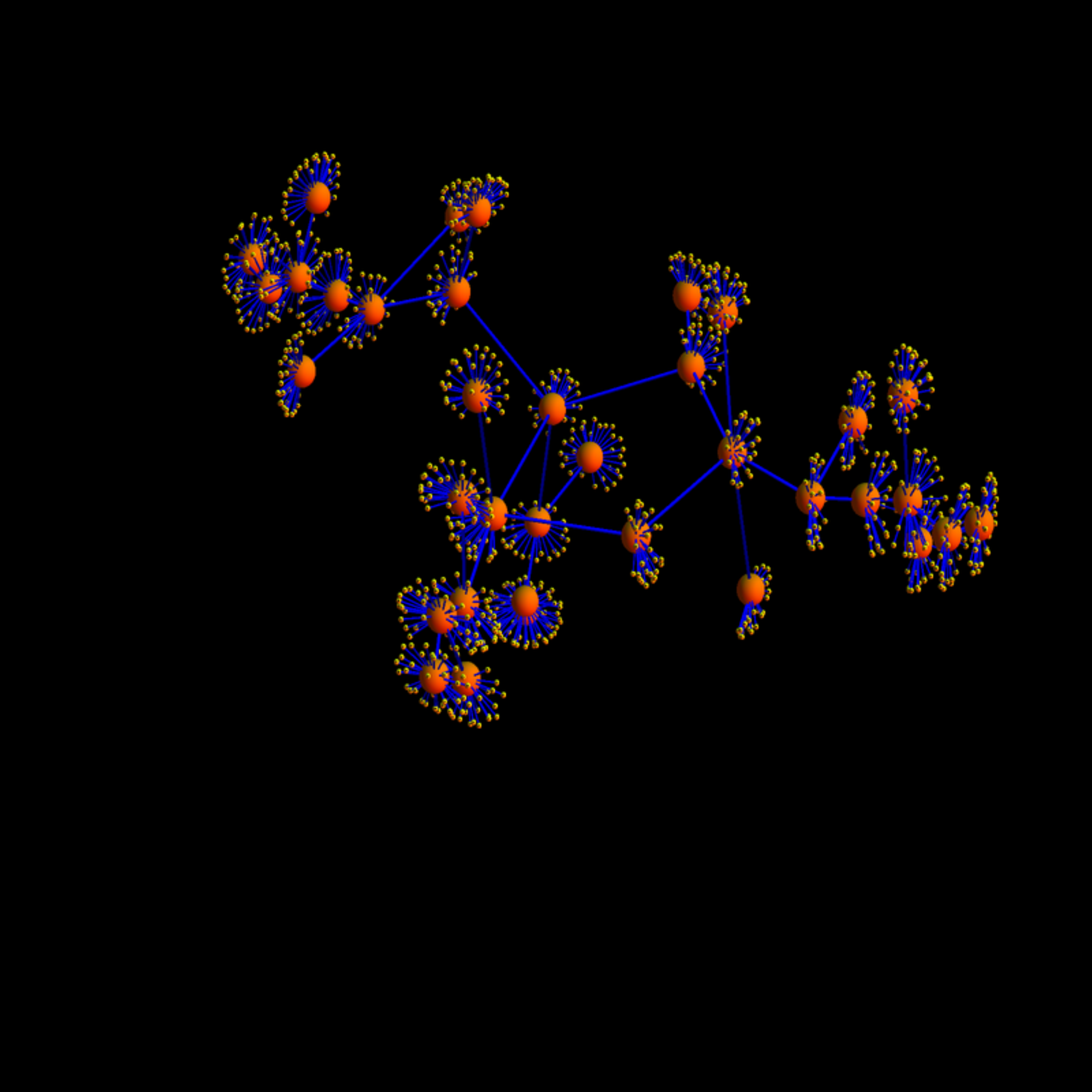}}
\scalebox{0.23}{\includegraphics{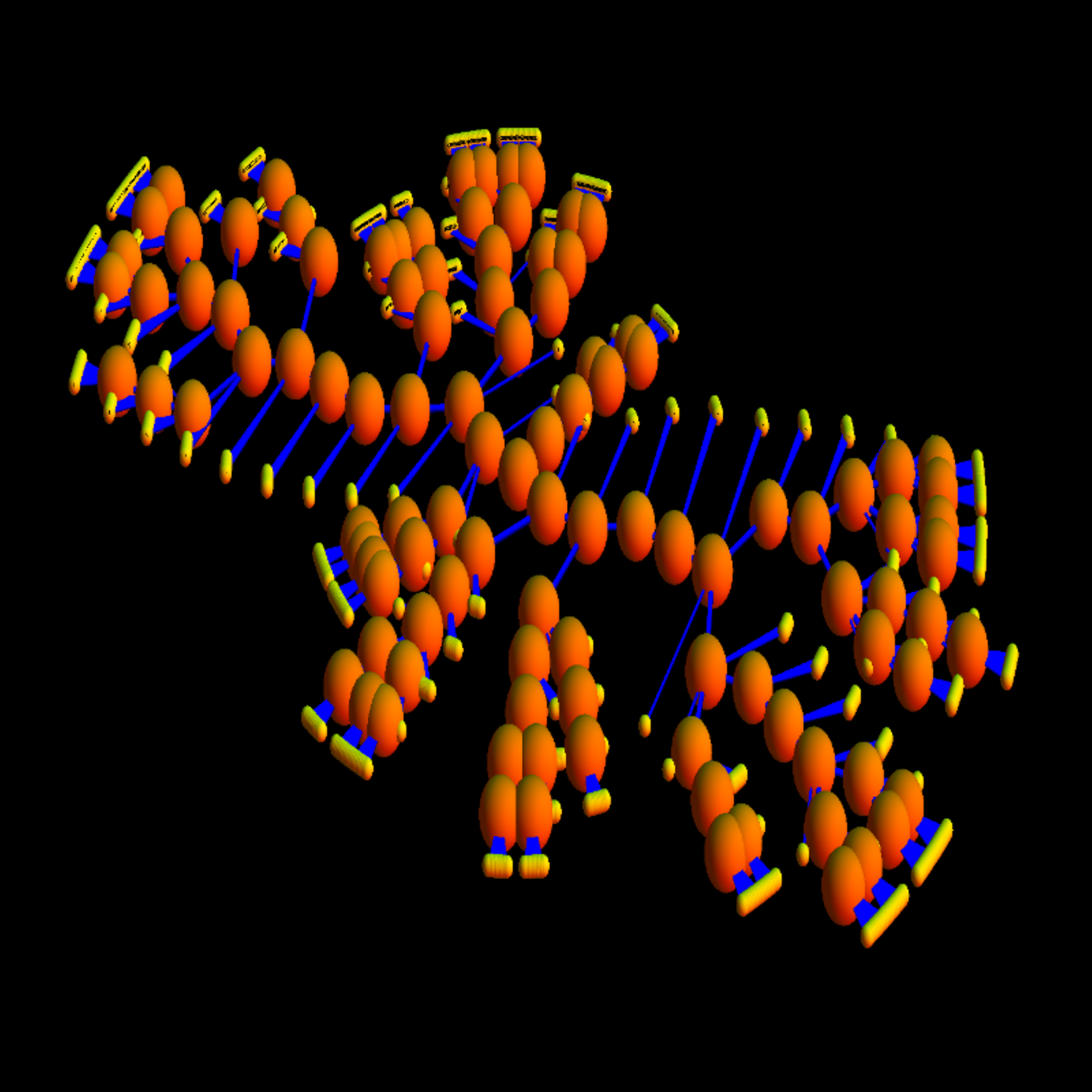}}
\scalebox{0.23}{\includegraphics{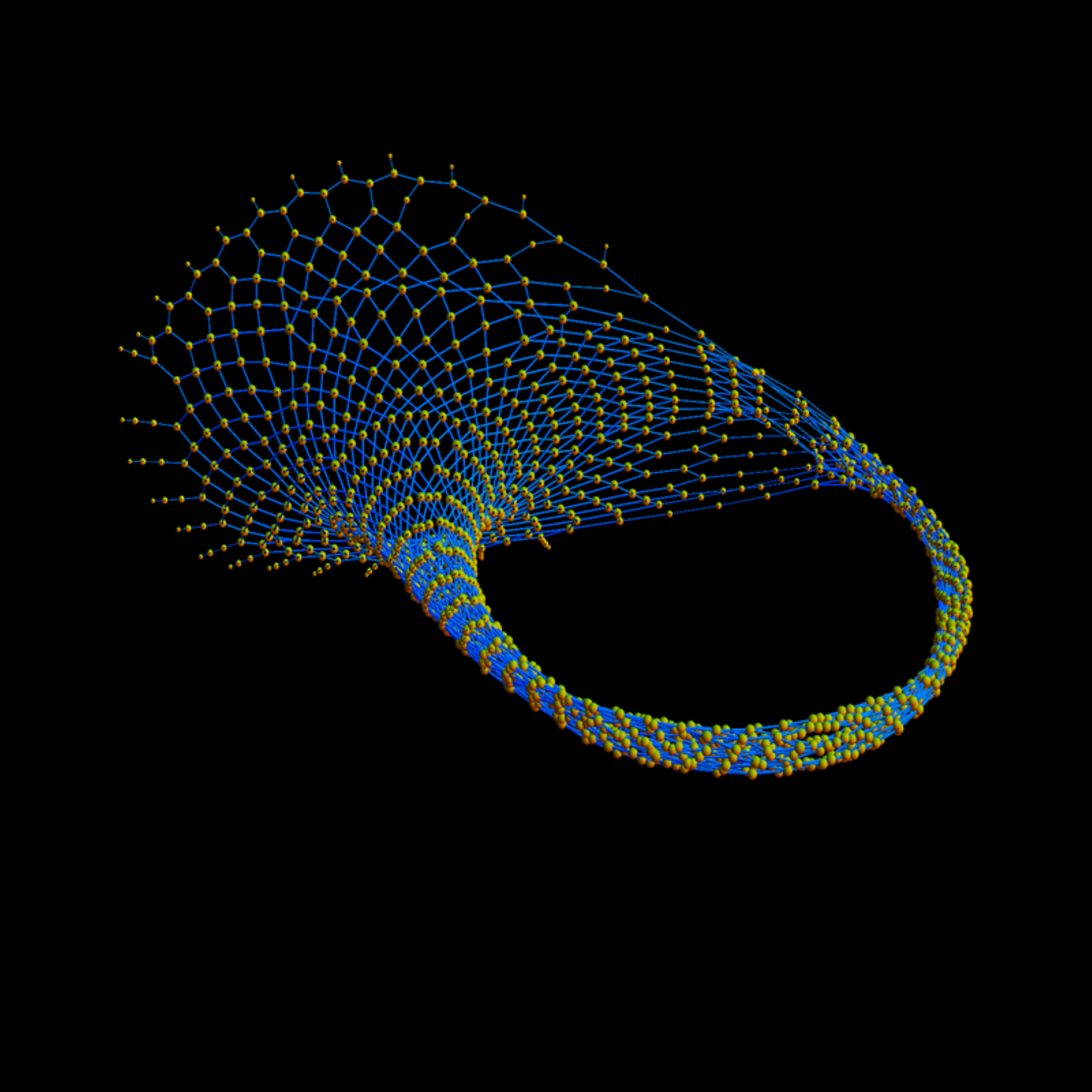}}
\scalebox{0.23}{\includegraphics{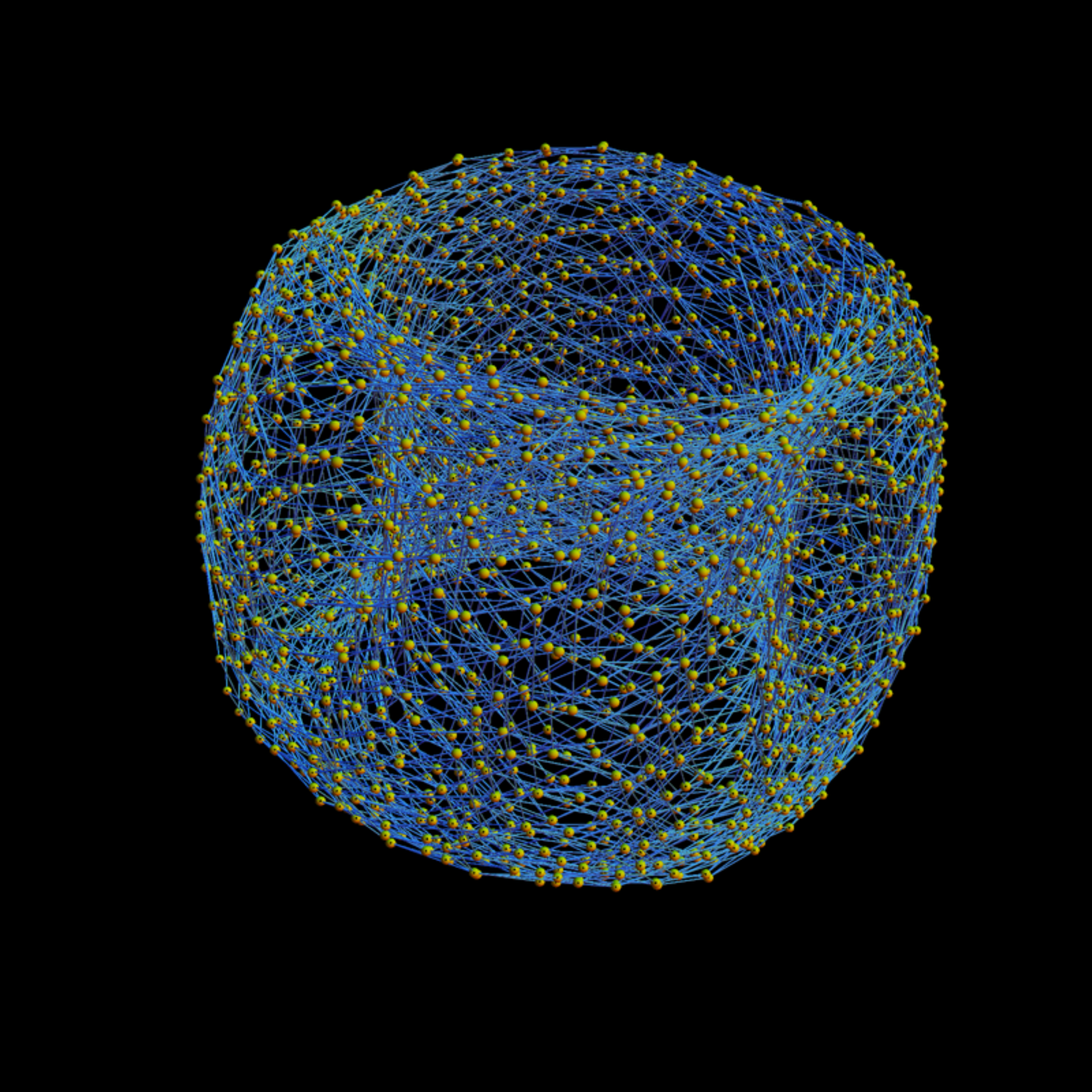}}
\scalebox{0.23}{\includegraphics{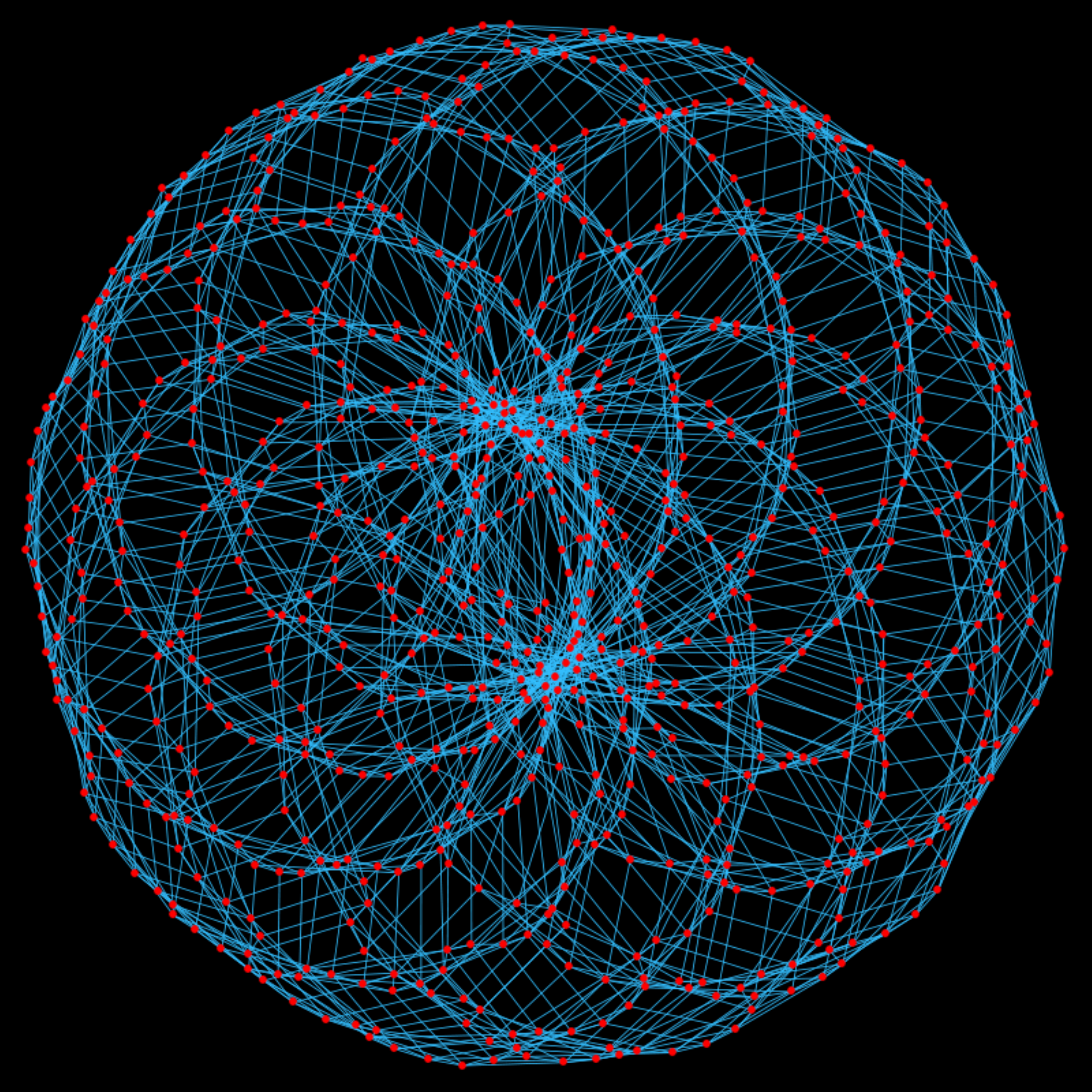}}
\scalebox{0.23}{\includegraphics{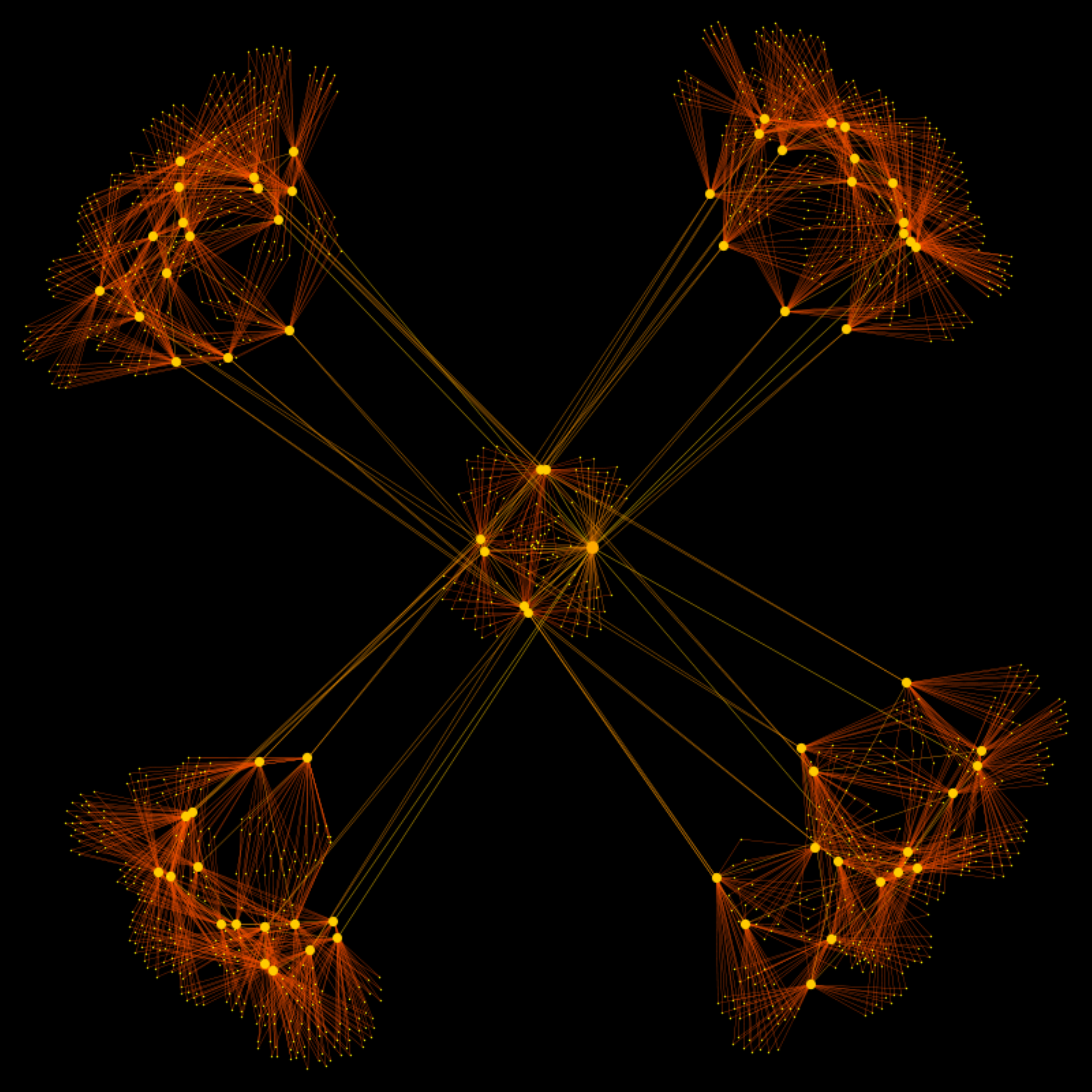}}
\scalebox{0.23}{\includegraphics{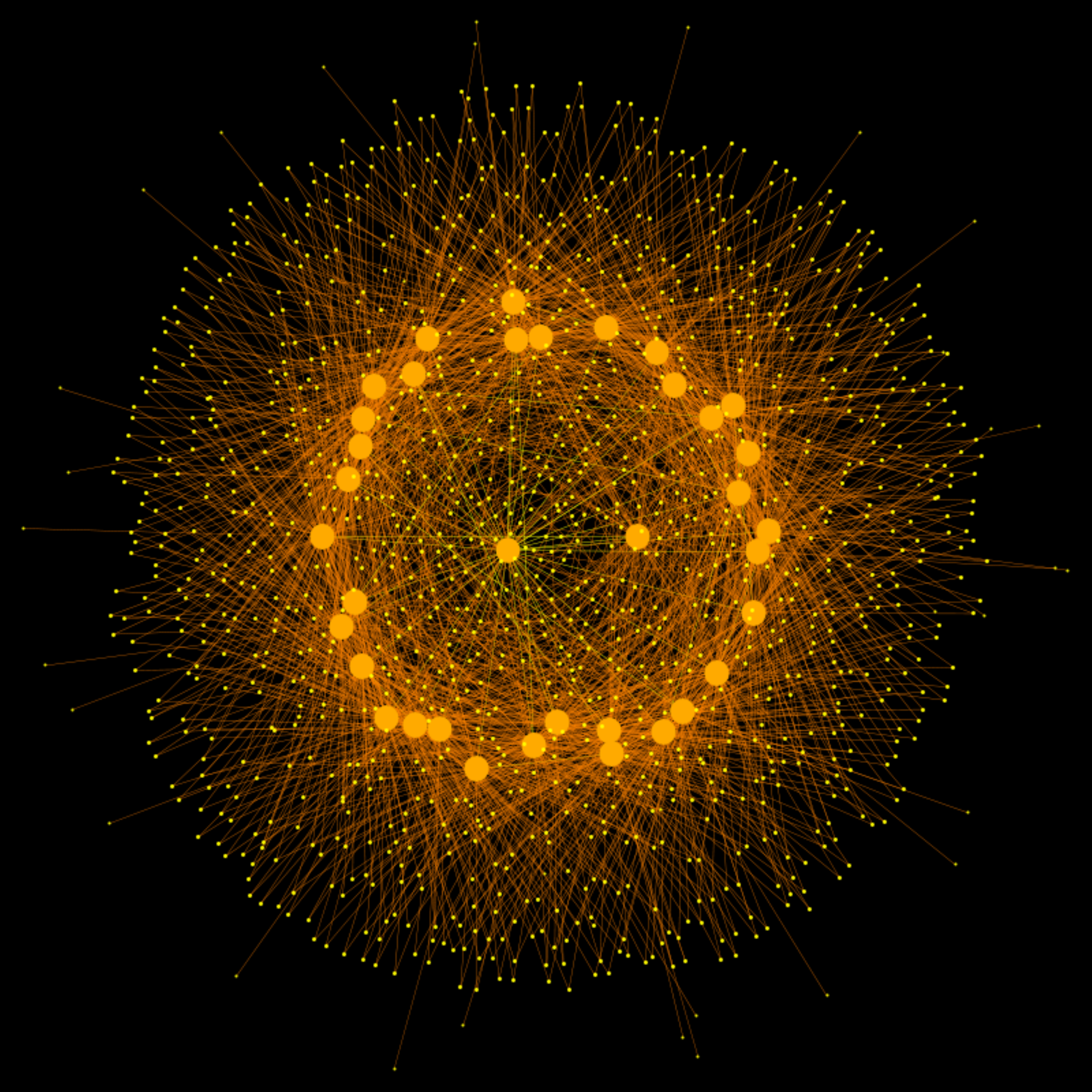}}
\scalebox{0.23}{\includegraphics{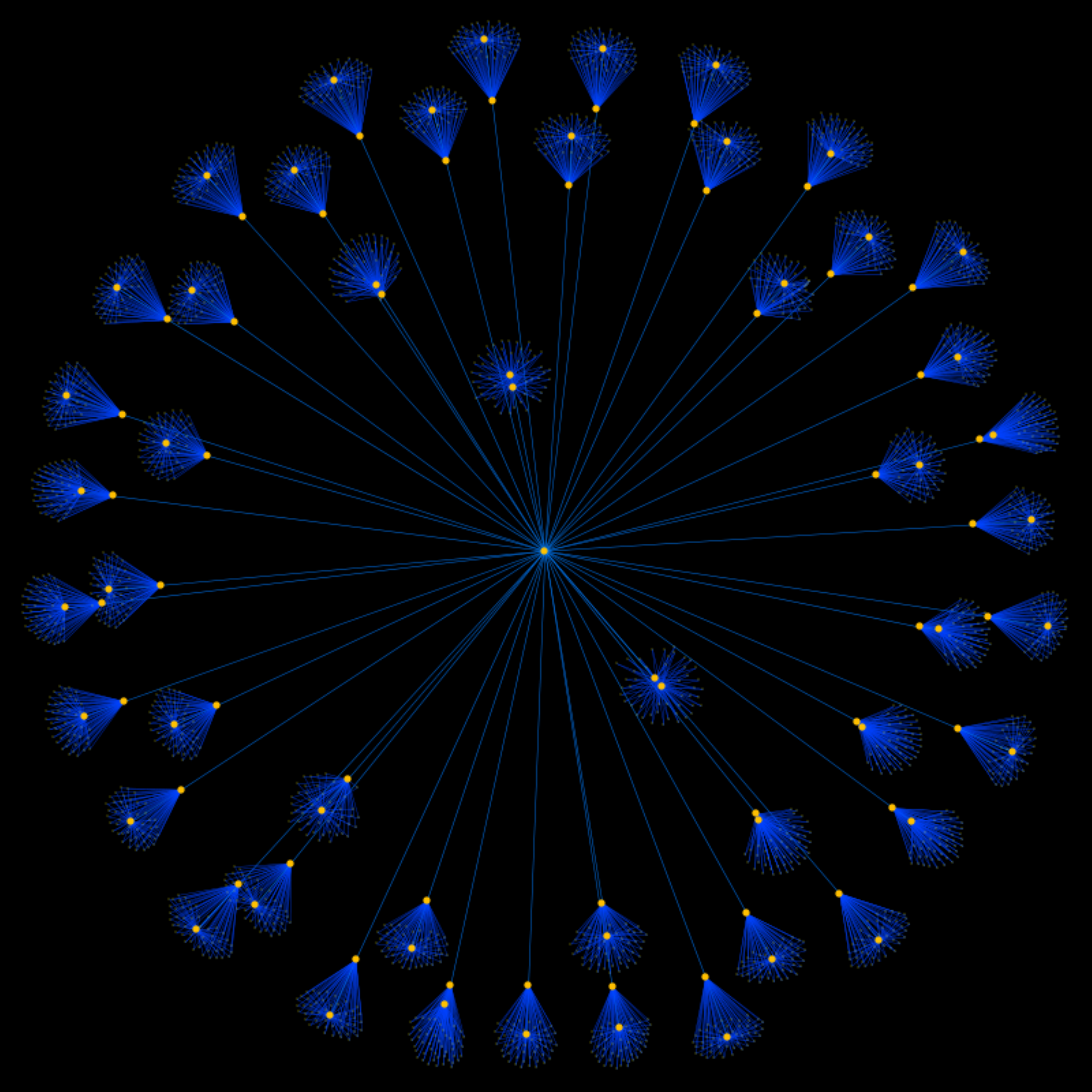}}
\scalebox{0.23}{\includegraphics{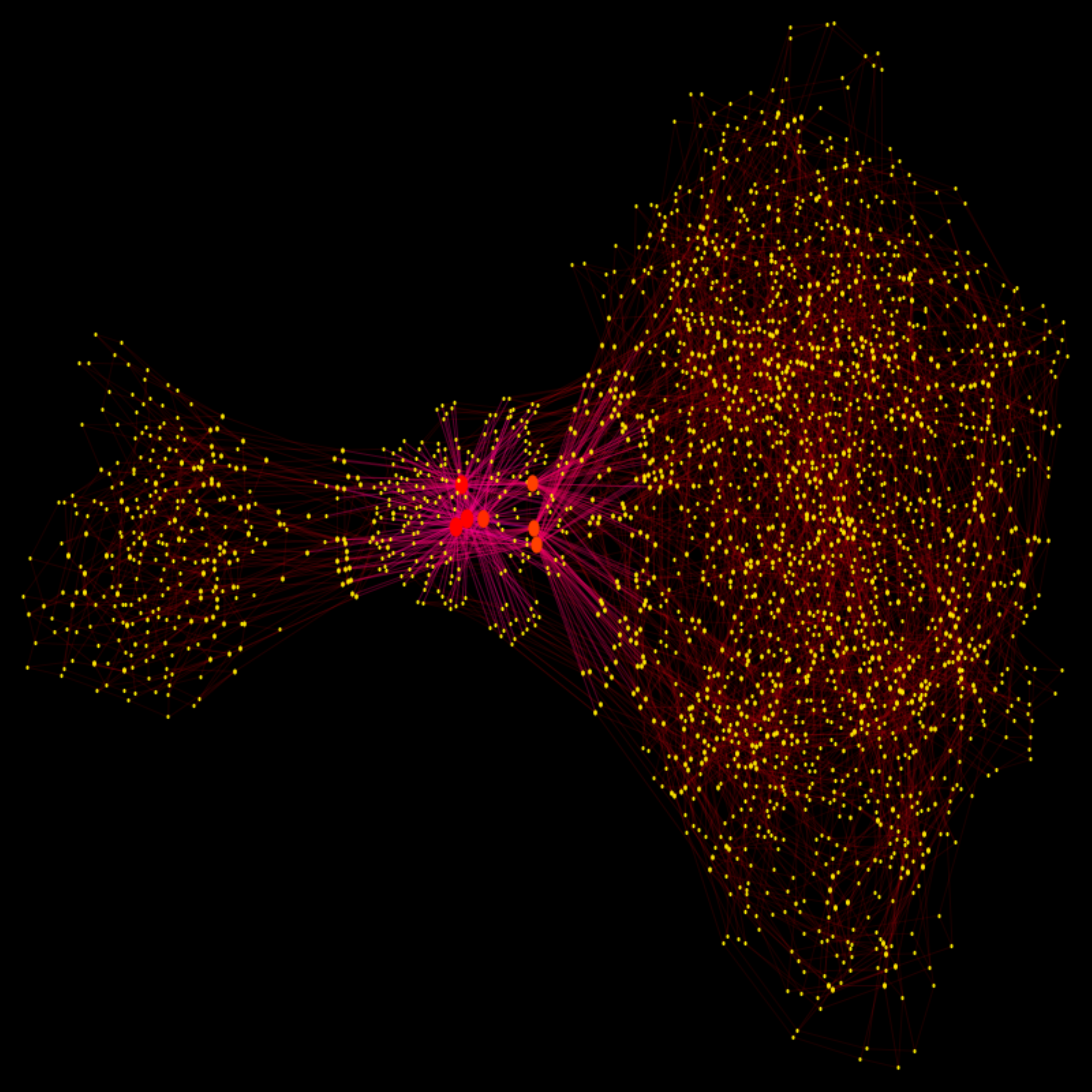}}
\scalebox{0.23}{\includegraphics{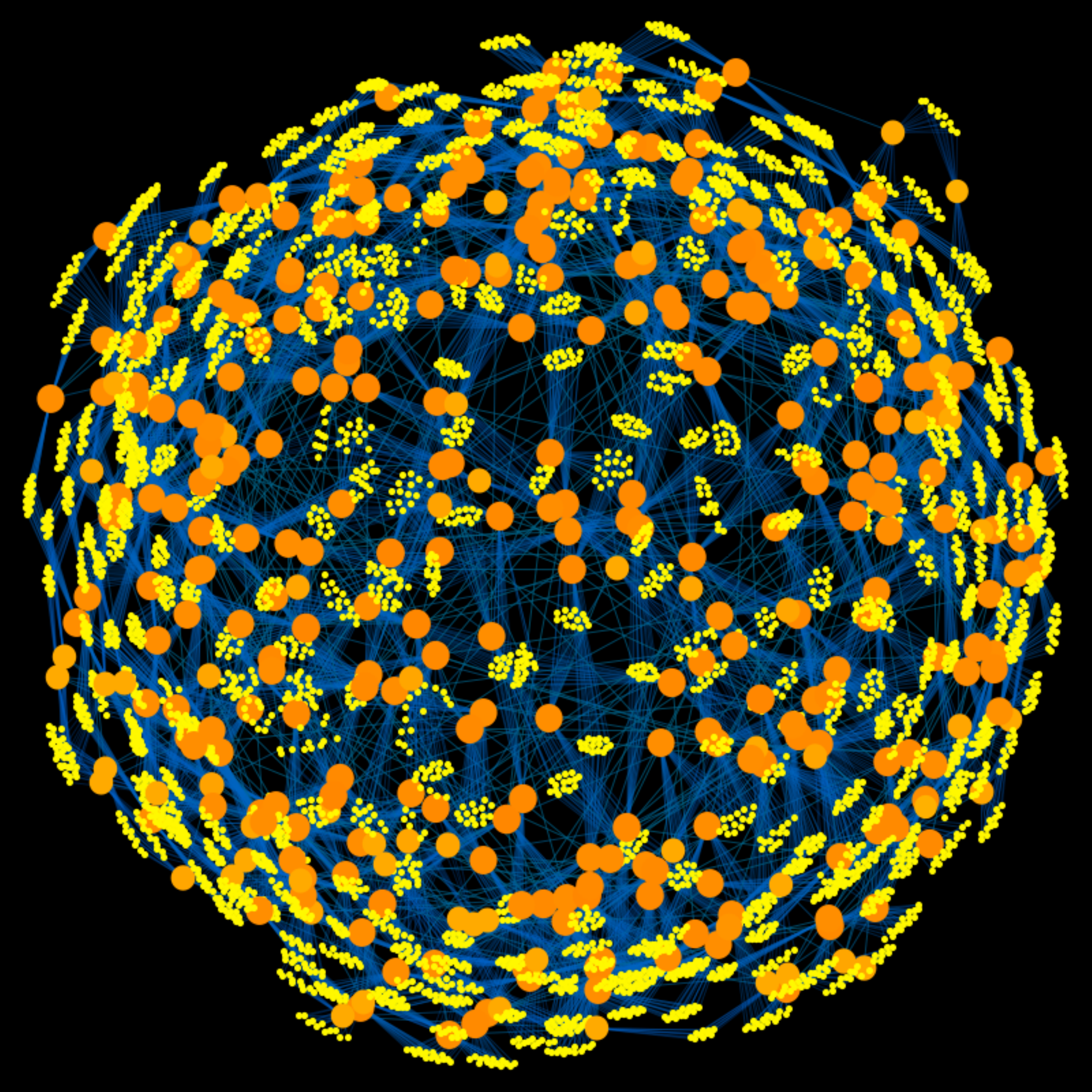}}

\vspace{12pt}
\bibliographystyle{plain}

\begin{thebibliography}{10}

\bibitem{Linked}
A-L. Barabasi.
\newblock {\em Linked, The New Science of Networks}.
\newblock Perseus Books Group, 2002.

\bibitem{BallobasKozmaMiklo}
B.~Bollob{\'a}s, R.~Kozma, and D.~Mikl{\'o}s, editors.
\newblock {\em Handbook of large-scale random networks}, volume~18 of {\em
  Bolyai Society Mathematical Studies}.
\newblock Springer, Berlin, 2009.

\bibitem{BornholdtSchuster}
S.~Bornholdt and H.~Schuster, editors.
\newblock {\em Handbook of Graphs and Networks}.
\newblock Viley-VCH, 2003.

\bibitem{Buchanan}
M.~Buchanan.
\newblock {\em Nexus: small worlds and the groundbreaking science of networks}.
\newblock W.W. Norton and Company, 2002.

\bibitem{Connected}
N.~Christakis and J.H. Fowler.
\newblock {\em Connected}.
\newblock Little, Brown and Company, 2009.

\bibitem{CohenHavlin}
R.~Cohen and S.~Havlin.
\newblock {\em Complex Networks, Structure, Robustness and Function}.
\newblock Cambridge University Press, 2010.

\bibitem{dicksonI}
L.E. Dickson.
\newblock {\em History of the theory of numbers.{V}ol. I:{D}ivisibility and
  primality.}
\newblock Chelsea Publishing Co., New York, 1966.

\bibitem{Easley}
D.~Easley and Jon Kleinberg.
\newblock {\em Networks, crowds and Markets, Reasonings about a highly
  connected world}.
\newblock Cambridge University Press, 2010.

\bibitem{Meester}
M.~Franceschetti and R.~Meester.
\newblock {\em Random networks for communication}.
\newblock Cambridge Series in Statistical and Probabilistic Mathematics.
  Cambridge University Press, Cambridge, 2007.
\newblock From statistical physics to information systems.

\bibitem{GK2}
M.~Ghachem and O.~Knill.
\newblock Deterministic {Watts-Strogatz} type graphs.
\newblock Preliminary notes, 2013.

\bibitem{GK1}
M.~Ghachem and O.~Knill.
\newblock Simple rules for natural networks.
\newblock Preliminary notes, 2013.

\bibitem{gleason88}
Andrew~M. Gleason.
\newblock Angle trisection, the heptagon, and the triskaidecagon.
\newblock {\em Amer. Math. Monthly}, 95(3):185--194, 1988.

\bibitem{GoodmanORourke}
J.E. Goodman and J.~O'Rourke.
\newblock {\em Handbook of discrete and computational geometry}.
\newblock Chapman and Hall, CRC, 2004.

\bibitem{Goyal}
S.~Goyal.
\newblock {\em Connections}.
\newblock Princeton University Press, 2007.

\bibitem{ibe}
O.C. Ibe.
\newblock {\em Fundamentals of Stochastic Networks}.
\newblock Wiley, 2011.

\bibitem{Jackson}
M.O. Jackson.
\newblock {\em Social and Economic Networks}.
\newblock Princeton University Press, 2010.

\bibitem{knillprobability}
O.~Knill.
\newblock {\em Probability Theory and Stochastic Processes with Applications}.
\newblock Overseas Press, 2009.

\bibitem{lagarias}
J.C. Lagarias.
\newblock The {$3x+1$} problem: an overview.
\newblock In {\em The ultimate challenge: the {$3x+1$} problem}, pages 3--29.
  Amer. Math. Soc., Providence, RI, 2010.

\bibitem{nbw2006}
D.~Watts M.~Newman, A-L.~Barab{\'a}si, editor.
\newblock {\em The structure and dynamics of networks}.
\newblock Princeton Studies in Complexity. Princeton University Press,
  Princeton, NJ, 2006.

\bibitem{newman2010}
M.E.J. Newman.
\newblock {\em Networks}.
\newblock Oxford University Press, Oxford, 2010.
\newblock An introduction.

\bibitem{lanford98}
III O.E.~Lanford.
\newblock Informal remarks on the orbit structure of discrete approximations to
  chaotic maps.
\newblock {\em Experiment. Math.}, 7(4):317--324, 1998.

\bibitem{Rannou}
F.~Rannou.
\newblock \'{E}tude num\'erique de transformations planes discr\`etes
  conservant les aires.
\newblock In {\em Transformations ponctuelles et leurs applications ({C}olloq.
  {I}nternat. {CNRS}, {N}o. 229, {T}oulouse, 1973)}, pages 107--122, 138.
  \'Editions Centre Nat. Recherche Sci., Paris, 1976.
\newblock With discussion.

\bibitem{Riesel}
H.~Riesel.
\newblock {\em Prime numbers and computer methods for factorization}, volume~57
  of {\em Progress in Mathematics}.
\newblock {Birkh\"auser} Boston Inc., 1985.

\bibitem{shen}
H-W. Shen.
\newblock {\em Community structure of complex networks}.
\newblock Springer Theses. Springer, Heidelberg, 2013.

\bibitem{Sync}
S.~H. Strogatz.
\newblock {\em Sync: The Ermerging Science of Spontaneous Order}.
\newblock Hyperion, 2003.

\bibitem{vansteen}
M.~van Steen.
\newblock {\em Graph Theory and Complex Networks, An introduction}.
\newblock Maarten van Steen, ISBN: 778-90-815406-1-2, 2010.

\bibitem{vivaldi}
F.~Vivaldi.
\newblock Algebraic number theory and {H}amiltonian chaos.
\newblock In {\em Number theory and physics ({L}es {H}ouches, 1989)}, volume~47
  of {\em Springer Proc. Phys.}, pages 294--301. Springer, Berlin, 1990.

\bibitem{WassermanFaust}
S.~Wasserman and K.~Faust.
\newblock {\em Social Network analysis: Methods and applications}.
\newblock Cambridge University Press, 1994.

\bibitem{SmallWorld}
D.~J. Watts.
\newblock {\em Small Worlds}.
\newblock Princeton University Press, 1999.

\bibitem{SixDegrees}
D.~J. Watts.
\newblock {\em Six Degrees}.
\newblock W. W. Norton and Company, 2003.

\bibitem{WattsStrogatz}
D.~J. Watts and S.~H. Strogatz.
\newblock Collective dynamics of 'small-world' networks.
\newblock {\em Nature}, 393:440--442, 1998.

\bibitem{WolframState1}
S.~Wolfram.
\newblock State transition diagrams for modular powers.
\newblock
  http://demonstrations.wolfram.com/StateTransitionDiagramsForModularPowers/,
  2007.

\bibitem{WolframState2}
S.~Wolfram.
\newblock Cellular automata state transition diagrams.
\newblock
  http://demonstrations.wolfram.com/CellularAutomatonStateTransitionDiagrams,
  2008.

\end{thebibliography}

\end{document}